\documentclass[a4paper,11pt]{amsart}
\usepackage[left=3.0cm,right=3.0cm,bottom=3.0cm]{geometry} 
\usepackage{latexsym,amssymb,calc,amsmath,amsthm,amssymb,color,graphicx,graphics}
\usepackage{bbold, epsfig, euscript,enumitem}
\usepackage{tikz}
\usepackage[utf8]{inputenc}
\usepackage[T1]{fontenc}
\usepackage{fancyhdr}

\newcommand{\Lip}{\operatorname{Lip}}
\newcommand{\gr}{\operatorname{gr}}
\newcommand{\Id}{\mathrm{Id}}

\newtheorem{theorem}{Theorem}[section]
\newtheorem{lemma}[theorem]{Lemma}
\newtheorem{proposition}[theorem]{Proposition}
\newtheorem{corollary}[theorem]{Corollary}
\theoremstyle{remark}
\newtheorem{remark}[theorem]{Remark}
\newtheorem{example}[theorem]{Example}
\newtheorem{definition}[theorem]{Definition}
\newtheorem{assum}[theorem]{Assumption}

\renewcommand{\div}{{\rm div \,}}

\newcommand{\esssup}[1]{\mathop{\rm ess\ sup}}
\newcommand{\essinf}[1]{\mathop{\rm ess\ inf}}
\newcommand{\N}{{\rm I\kern - 2.5pt N}}
\newcommand{\Z}{{\rm Z\kern - 5.5pt Z}}
\newcommand{\Q}{{\rm I\kern - 5.25pt Q}}
\newcommand{\C}{{\rm I\kern - 6.25pt C}}
\newcommand{\R}{{\rm I\kern - 2.5pt R}}

\newcommand{\DBb}{\mathbf{b}}

\newcommand{\DBe}{\mathbf{e}}

\newcommand{\DBf}{\mathbf{f}}
\newcommand{\DBn}{\mathbf{n}}
\newcommand{\DBp}{\mathbf{p}}
\newcommand{\DBx}{\mathbf{x}}
\newcommand{\DBy}{\mathbf{y}}
\newcommand{\DBz}{\mathbf{z}}

\newcommand{\DBv}{\mathbf{v}}
\newcommand{\DBw}{\mathbf{w}}
\newcommand{\DBtau}{\boldsymbol\tau}
\newcommand{\DBF}{\mathbf{F}}
\newcommand{\DBN}{\mathbf{N}}
\newcommand{\DBS}{\mathbf{S}}

\newcommand{\DBdiv}{\:\mathrm{div \,}}

\newcommand{\nbf}{\mathbf{n}}

\newcommand{\Sbf}{\mathbf{S}}

\newcommand{\vbf}{\mathbf{v}}

\newcommand{\na}{\nabla}

\newcommand{\iface}{\Sigma}

\begin{document}
\title[Reynolds transport theorem for two-phase flows]{Co-moving volumes and the Reynolds transport theorem for two-phase flows}
\author[Dieter Bothe]{Dieter Bothe \vspace{0.1in}}
\address{Department of Mathematics and Profile Area Thermofluids \&  Interfaces\\
Technical University of Darmstadt\\
Peter-Grünberg-Str.\ 10\\
D-64287 Darmstadt, Germany}
\email{bothe@mma.tu-darmstadt.de}

\author[Matthias K\"ohne]{Matthias K\"ohne \vspace{0.1in}}
\address{Heinrich-Heine-Uni\-ver\-sit\"at D\"usseldorf,
	Mathematisch-Naturwissenschaftliche Fakult\"at,
	Mathematisches Institut,
	Universit\"atsstr.~1, D-40225 D\"usseldorf, Germany}
\email{matthias.koehne@hhu.de}
\date{August 23, 2026}
\begin{abstract}
\noindent
We consider the local kinematics at fluid interfaces in sharp-interface two-phase flows with phase change and interfacial slip. In this setting the governing velocity field is discontinuous at the phase boundary, with possible jumps of both normal and tangential components, and the associated kinematic initial value problems may fail to be uniquely solvable. A physically consistent sharp-interface example exhibits this non-uniqueness and, in addition, rapid loss of boundary regularity: smooth initial control volumes can instantaneously develop edges, while their phasewise parts may form cusps. Motivated by these phenomena, we use concepts from differential inclusions to define co-moving volumes as attainable sets. For such attainable-set co-moving volumes in three-dimensional two-phase flows, emanating from bounded connected $C^2$ initial control volumes whose boundaries meet the interface transversally, we prove a Reynolds transport theorem, valid at a general initial time, first in boundary-integral form and then in divergence form. A key ingredient is a boundary-integral form of the single-phase Reynolds transport theorem for families of compact regular closed sets whose space-time tubes are Lipschitz domains with the natural time-slice/lateral boundary structure and whose lateral boundaries are a.e. transverse to the time slices. 
We also provide a short proof of this single-phase result by applying the divergence theorem in space-time; this proof does not require the motion to be generated by an ambient velocity field.
\end{abstract}
\maketitle
\section{Introduction}
\noindent
The Reynolds transport theorem  (\cite{ReynoldsThm}, p.\ 12f) is an essential tool for deriving continuum physical balance equations and for numerical discretization of these in particular in the context of the finite-volume method.
In its classic form it assumes a regular velocity field, allowing for unique solvability of the initial value problems for the kinematic ordinary differential equation (ODE) system, 
\begin{equation}\label{DBeq:kinematic-ODE1}
\dot{\DBx} (t) = \DBv (t, \DBx (t)) \; \mbox{ for } t\in J=(a,b), \quad \DBx (t_0)=\DBx_0,
\end{equation}
which governs the motion of so-called fluid elements (passive tracer particles) within the flow field. 
This induces a smooth flow map
\begin{equation}\label{eq-flow-map-sv}
\Phi_{t_0}^t (\DBx_0) := \DBx (t; t_0, \DBx_0),
\end{equation}
where $\DBx (t; t_0, \DBx_0)$ denotes this unique solution of equation (\ref{DBeq:kinematic-ODE1}) for initial time $t_0$ and initial position $\DBx_0$.
In particular, the flow map allows to introduce the concept of co-moving volumes\vspace{-0.05in}
\begin{equation}\label{eq-comoving-volume-sv}
G(t)= \Phi_{t_0}^t (G_0),
\end{equation}
where $G_0$ is some given initial volume at time $t_0$, usually a smoothly bounded compact set.
If $\DBv$ is the flow field associated to a single-phase flow, the co-moving volumes $G(t)$ are composed of the same fluid elements, hence
the $G(t)$ are also called material (control) volumes in this setting.
This makes the integral formulation of the mass balance simple and elegant, viz.\vspace{-0.1in}
\begin{equation}\label{DBeq:mass-balance-material}
\frac{d}{dt} \int_{G(t)} \rho (t,\DBx) d\DBx = 0
\end{equation}
for all co-moving volumes, where $\rho$ denotes the mass density of the fluid.
To derive the local form of the balance equation, a generalized version of the Leibniz rule for the differentiation of integrated quantities is required.
This is precisely the purpose of the Reynolds transport theorem (RTT for short), which reads as\footnote{Surface and curve integrals are
understood with respect to Hausdorff measure. We write $dS=d\mathcal H^{n-1}$ on spatial hypersurfaces,
$d\mathcal H^n$ on space-time hypersurfaces and $d\ell=d\mathcal H^1$ on curves.}
\begin{equation}\label{DBeq:RTT}
\frac{d}{dt} \int_{G(t)} \psi (t,\DBx) d\DBx =  \int_{G(t)} \partial_t \psi (t,\DBx) d\DBx + \int_{\partial G(t)} \psi (t,\DBx) V_n (t,\DBx) dS,
\end{equation}
where $V_n$ denotes the speed of normal displacement of the family of surfaces $\partial G(t)$.
The latter is commonly defined as $V_n(t,\DBx)=\gamma'(t)\cdot \DBn (t,\DBx)$, with $\DBn$ the outer unit normal of $G(\cdot)$, 
for any $C^1$-curve $\gamma (\cdot)$ such that
$\gamma (t)=\DBx$ and $\gamma(s)\in \partial G(s)$ for all $s$ close to $t$; see below for more details.
As the speed of normal displacement is given by $V_n = \DBv \cdot \DBn$ if the transport stems from a velocity field $\DBv$,
application of the divergence theorem to the surface integral yields the second form
\begin{equation}\label{DBeq:RTT2}
\frac{d}{dt} \int_{G(t)} \psi (t,\DBx) d\DBx =  \int_{G(t)} \big( \partial_t \psi (t,\DBx) + \DBdiv (\psi (t,\DBx) \DBv (t,\DBx)) \big) d\DBx.
\end{equation}
Applying (\ref{DBeq:RTT2}) with $\psi=\rho$ and for $G(t_0)=G_0:={B}_R (\DBx_0)$, conservation of mass (\ref{DBeq:mass-balance-material}) yields
\begin{equation}\label{eq:average}
\frac{1}{|B_R(0)|}\int_{B_R(\DBx_0)} \big( \partial_t \rho  + \DBdiv (\rho   \DBv ) \big) d\DBx =0
\end{equation}
for $t=t_0$ and all balls ${B}_R (\DBx_0)$ inside the flow domain.
Hence $R\to 0+$ yields the continuity equation, i.e.\vspace{-0.05in}
\begin{equation}\label{DBeq:continuity}
\partial_t \rho  + \DBdiv (\rho   \DBv ) = 0
\end{equation}
at all $(t,\DBx)$ where the integrand in (\ref{eq:average}) is continuous.
Note that this derivation only requires the RTT for smoothly bounded sets (such as balls) and only at $t=t_0$, i.e.\
only the short-time evolution of co-moving sets $\Phi_{t_0}^t (G_0)$ matters.

The Reynolds transport theorem can be proven in different ways. 
The classical proof uses the flow map $\Phi_{t_0}^t$ to pull back the moving domain of integration $G(t)$ onto the initial domain $G_0$. From the transformation formula for integration, the Jacobian of the flow map, $J(t,\DBx)=\det (\nabla_x \Phi_{t_0}^t (\DBx))$,  appears as a factor under the integral. This allows for differentiation under the integral. Then, using Jacobi's differential formula for determinants, one obtains
$\partial_t J(t,\DBx) = J(t,\DBx) \DBdiv \DBv (t,\Phi_{t_0}^t (\DBx))$. A further application of the transformation rule, this time in the opposite direction, yields the result. For more details on this common approach, we refer to, e.g., \cite{Arisbook}, \cite{lidstrom2011moving}, \cite{amann2009analysis}, \cite{TheinWarnecke} and the references given there.

The classical version of the proof of the RTT requires $C^1$-regularity of the velocity field. It is natural to ask whether one can relax the regularity assumptions. For instance, the velocity field may only display a certain Sobolev space regularity, as the latter appears in weak solution theories. In the seminal paper \cite{diperna1989ordinary}, DiPerna and Lions were able to generalise the treatment of the motion of fluid elements by moving away from the particle viewpoint, replacing the kinematic ODE (\ref{DBeq:kinematic-ODE1}) by a weak description based on the PDE governing the passive transport of a scalar field, i.e.\ PDEs of the type (\ref{DBeq:continuity}) but for a general scalar $\phi$ instead of the mass density $\rho$.
As a result of this theory and further developments in this direction (e.g., \cite{Ambrosio2004}), it is rather clear that a transport theorem can be formulated for velocity fields from appropriate Sobolev spaces or even with $BV$-regularity. It has to be understood in terms of geometric measure theory, meaning that the transported volumes are only implicitly determined via the solutions of transport PDEs. In general, it is not clear whether this still corresponds to the Lagrangian transport of volumes with sufficiently regular boundaries or, in our setting, whether the interface between the two fluid phases remains a hypersurface of the regularity required for a sharp-interface model.

Possible generalisations of the Reynolds transport theorem to irregular transported domains have been developed by Seguin and Fried~\cite{seguin2014roughening} and in related subsequent work. Their approach is based on Harrison's theory of differential chains \cite{harrison2015operator}
and is not restricted to convective evolutions generated by a classical flow map. In particular, it allows for evolving domains that may develop holes, split into pieces, or otherwise change their geometric regularity, and it yields a transport theorem containing a term associated with the evolution of the boundary. However, convective specializations rely on single-valued smooth kinematics, whereas the present work addresses discontinuous two-phase kinematics.

For an RTT for reachable sets of differential inclusions under different assumptions, see Lorenz~\cite{Lorenz2006RTT}.

More recently, Soga \cite{Soga2026-preprint} studied co-moving volumes and the Reynolds transport theorem in a DiPerna-Lions framework on bounded domains, where forward images under the generalized flow need not be measurable and are therefore replaced by images of suitably trimmed 
measurable initial sets. Related background on transport equations and flows for nonsmooth velocity fields is provided by Lions and Seeger \cite{LionsSeeger2024}, who developed a theory for one-sided Lipschitz velocity fields in a setting that may lie beyond the standard DiPerna-Lions framework. 

The present paper is concerned with the fact that in sharp-interface two-phase flows with phase change and slip the transporting velocity field is discontinuous across the moving interface. As a consequence, the kinematic ODE for fluid particles may fail to be uniquely solvable, so that the classical notion of a single-valued flow map and the associated family of material volumes is no longer available. This necessitates a different approach: we replace the discontinuous ODE by a multivalued kinematic description based on differential inclusions and define co-moving sets by means of the associated attainable sets.

As noted above, in two-phase flows with phase change and interfacial slip, the velocity is discontinuous at the evolving phase boundary.
In the sharp-interface continuum models, the discontinuity is assumed to be a jump discontinuity. Mass transfer can, in general, produce a non-vanishing jump of the normal velocity component when the one-sided mass densities differ; cf.\ \cite{Bo-interface-mass}, \cite{Hutter-book}, \cite{Ishii}, \cite{Slattery-Interfaces} for the sharp-interface continuum modeling.
Furthermore, also a jump of the tangential velocity component is possible. This corresponds to so-called slip or slippage of the two fluid phases, i.e.\ one fluid phase slides over the other one. 
Slip between two liquid phases can appear for instance at the interface between polymeric phases \cite{zartman2011particle}.
Other examples can be found, e.g., in \cite{telari2022intrinsic}, \cite{koplik2006slip}.
The phenomenon of slip between liquid phases is analogous to the appearance of slip at the interface between a fluid and a solid phase.

The results described above do not directly provide the interface-resolved transport identity needed for the present sharp-interface setting with discontinuous two-phase kinematics, interfacial slip, phase change, and co-moving sets defined by attainable sets.
Based on the set-valued approach to rigorously define co-moving sets, the main purpose of the present paper is to prove a Reynolds transport theorem tailored to sharp-interface two-phase flows with phase change and slip, while retaining enough geometric regularity to track the moving interface within the sharp-interface framework.
This setting is not to be confused with the RTT in variational formulations, such as in \cite{slodivcka2025reynolds}, where the transporting velocity field is $C^1$, while a weak formulation allows for discontinuous material parameters.

For the present purpose we need a setting that is somewhat weaker than the classical one since smooth initial control volumes, for instance balls, 
can instantaneously lose their regularity by developing edges. Thus the regularity may instantaneously drop to mere Lipschitz regularity. Even cusps may develop after finite time. This limits the expectation on the strength of the desired RTT. On the other hand, the primary goal of this work is to provide the two-phase RTT as a rigorous modeling tool. For this purpose, it suffices to prove the transport identity at the initial time, as this is enough for the derivation of balance equations. We first prove a two-phase Reynolds transport theorem in boundary-integral form,
corresponding to (\ref{DBeq:RTT}). This formulation only requires the partial time derivative of the transported scalar and the normal speeds of the relevant moving boundary pieces. This is the more basic formulation as the interface motion is not given by the adjacent bulk velocity.
The divergence form, which is closer to the usual Eulerian balance-law notation and to (\ref{DBeq:RTT2}), is then recovered as a corollary under additional phasewise assumptions sufficient to apply the divergence theorem. 

To be able to prove the two-phase RTT in boundary-integral form, we need a
corresponding single-phase RTT under weaker assumptions such as Lipschitz domains.
Furthermore, it needs to cover the case in which parts of the boundary of the evolving sets $G(t)$ move by different speeds, i.e.\ cases in which the speed of normal displacement $V_n$ has jump discontinuities along $\partial G(t)$. In such cases, the $G(t)$ are not given via a flow map with a continuous velocity field. A related boundary-driven RTT for nonconvecting domains with Lipschitz reduced boundaries was proved by Seguin~\cite{Seguin2020nonconvecting}, assuming that the evolving boundary admits a common $C^1$-in-time $W^{1,\infty}$ parametrization. For the precise formulation needed here---a Lipschitz space-time tube whose motion is specified only through its geometric normal speed, without an ambient velocity field or differential inclusion---we therefore include a short proof in section~\ref{BIF-RTT} below.
For self-containment, we also collect some basic facts on Lipschitz domains in Appendix~A.
The latter notation is reserved for open sets. As we will usually consider closed (co-moving) sets, we extend this notion in a simple manner:
we call $G\subset \R^n$ a compact Lipschitz domain if $G=\overline{G^\circ}$ with $G^\circ$, denoting the interior of $G$, being a bounded Lipschitz domain.

Our approach to prove the RTT deviates from the standard proof by first showing the integrated-in-time version of (\ref{DBeq:RTT}), resp.\ of (\ref{DBeq:RTT2}). This way, the Reynolds transport theorem becomes an instance of the divergence theorem, but in space-time. This idea has already been used in an Eulerian setting in \cite{eck2017mathematical} under standard regularity assumptions\footnote{Due to personal communication with Harald Garcke, this idea in \cite{eck2017mathematical} goes back to lectures held by Hans-Wilhelm Alt.}. A precursor (in 1D) can be found in the Remark given in \cite{flanders1973differentiation}. 
\begin{figure}
\includegraphics[width=8cm]{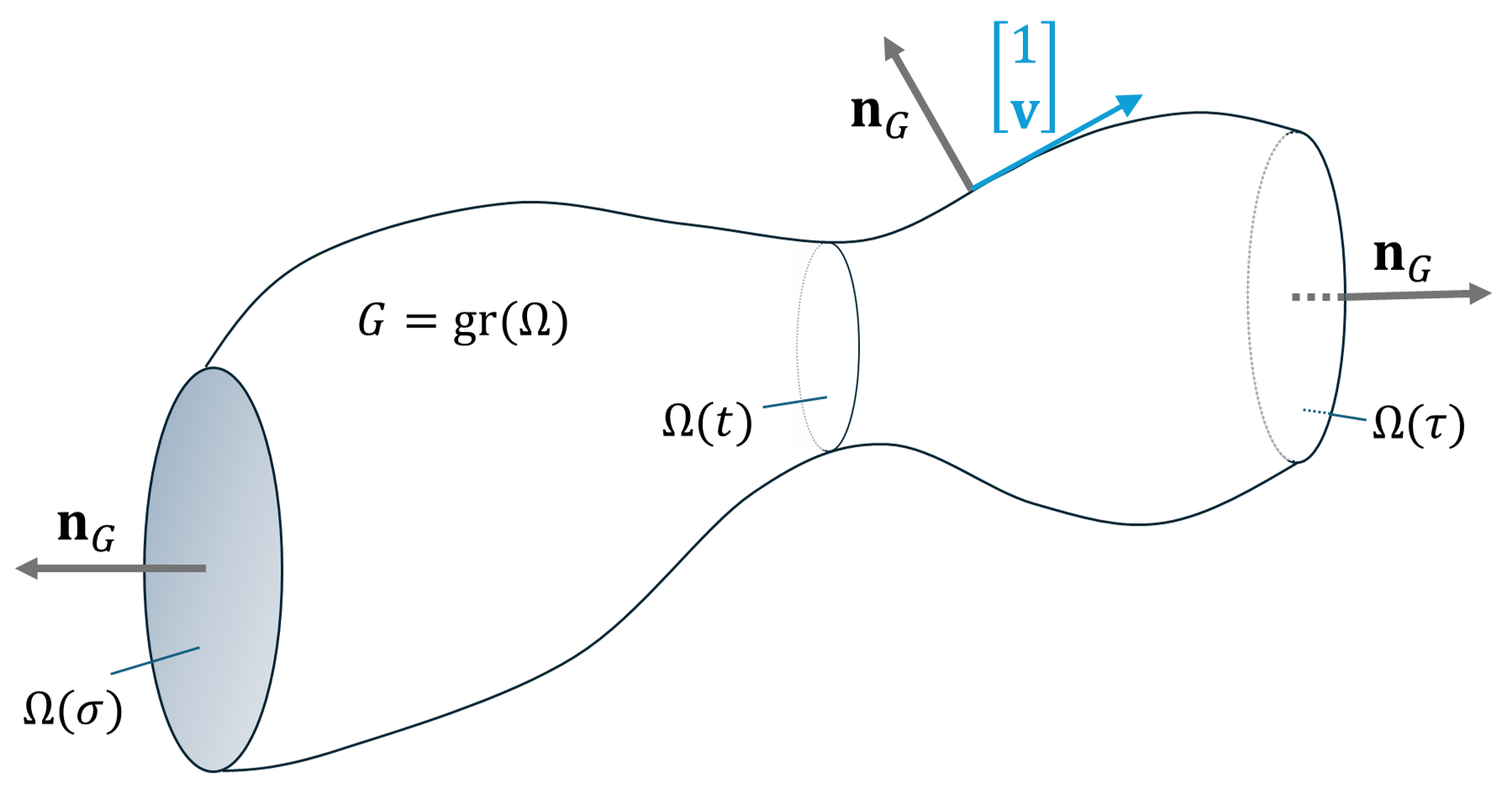}\vspace{-0.1in}
\caption{Space-time graph of the family of co-moving volumes $G(t)$, employed as a domain for the divergence theorem in $\R\times \R^3$.}\label{DBFig_tube}
\end{figure}
To explain the core idea, we give a first result for the integrated-in-time version of (\ref{DBeq:RTT2}) that includes the classical RTT.
\begin{theorem}[Reynolds transport theorem]\label{DBthm:1ph-transport-theorem-comovingCV}
Let $J=(a,b)\subset \R$, $\Omega \subset\R^n$ be open and $\DBv \in C^1 (J\times \Omega ; \R^n)$.
Let $t_0\in J$ and $G_0\subset \Omega$ be a compact Lipschitz domain.
Let $r>0$ be such that the co-moving volumes  $G(t)= \Phi_{t_0}^t (G_0)$ exist for all
$t\in J_0:=(t_0-r,t_0+r)\cap J$. 
Let $\psi \in W^{1,1} (J_0\times \Omega)$, which we identify with its Bochner-Sobolev representative
\[
  \psi\in W^{1,1}\bigl(J_0;L^1(\Omega)\bigr)
  \hookrightarrow C\bigl(J_0;L^1(\Omega)\bigr).\vspace{-0.15in}
\]
Then\vspace{-0.05in}
\begin{equation}\label{eq_int-RTT2}
\int_{G(\tau)} \hspace{-4pt}\psi  (\tau,\DBx) d\DBx
= \int_{G(\sigma)}  \hspace{-4pt}\psi (\sigma,\DBx)  d\DBx +
\int_\sigma^\tau \int_{G(t)}  \hspace{-4pt} (\partial_t \psi + \! \DBdiv (\psi \DBv ) ) (t,\DBx) d\DBx \, dt
\end{equation}
holds for all $[\sigma,\tau]\subset J_0$.

Thus, equation (\ref{DBeq:RTT2}) holds a.e.\ on $J_0$. It holds for all $t\in J_0$ in case $\psi \in C^1(J_0\times \Omega)$.
\end{theorem}
\begin{proof}
Given \(\sigma<\tau\) such that \([\sigma,\tau]\subset J_0\), set
\[
  Q
  =
  \operatorname{gr}(G|_{[\sigma,\tau]})
  =
  \{(t,\DBx):t\in[\sigma,\tau],\ \DBx\in G(t)\}.
\]
Then $Q=\overline{Q^\circ}$ for 
\[
  Q^\circ
  :=
  \{(t,\DBx)\in(\sigma,\tau)\times\Omega:
  \DBx\in G(t)^\circ\}.
\]
To see that $Q^\circ$ is a bounded Lipschitz domain, note that
\[
  Q^\circ
  =
  T\bigl((\sigma,\tau)\times G_0^\circ\bigr)
\quad \mbox{ for } \quad
  T(t,\DBx):=\bigl(t,\Phi^t_{t_0}(\DBx)\bigr).
\]
The map \(T\) is \(C^1\), because \(\DBv\) and
\(\nabla_\DBx\DBv\) are continuous, and the inverse map
\[
  T^{-1}(t,\DBx)
  =
  \bigl(t,\Phi^{t_0}_t(\DBx)\bigr)
\]
exists and is \(C^1\) as well. Hence \(T\) is a \(C^1\)-diffeomorphism from a neighborhood
of \([\sigma,\tau]\times G_0\) onto a neighborhood of \(Q\). Since
\((\sigma,\tau)\times G_0^\circ \) is a bounded Lipschitz domain by Prop.~\ref{prop:Cartesian-product-Lipschitz-domain},
\(Q^\circ\) is a bounded Lipschitz domain in \(\R^{n+1}\), hence \(Q\) is a
compact Lipschitz domain. 

Let $\DBw (t,\DBx)=(1,\DBv (t,\DBx))$ and note that the joint $(t,\DBx)$-divergence of $\psi \DBw$ in $\R^{n+1}$ is
\[
{\rm div}_{(t,\DBx)} (\psi \DBw) (t,\DBx) = \partial_t \psi (t, \DBx) + \DBdiv (\psi \DBv )(t,\DBx).
\]
Integration over $Q^\circ$ yields
\[
\int_\sigma^\tau \int_{G(t)}  (\partial_t \psi + \DBdiv (\psi \DBv ) )d\DBx \, dt= \int_{Q^\circ} {\rm div}_{(t,\DBx)} (\psi \DBw) d(t,\DBx).
\]
Now $\psi \DBw \in W^{1,1} (Q^\circ;\R^{n+1})$ by the assumptions on $\psi$ and $\DBv$.
Hence ${\rm div}_{(t,\DBx)} (\psi \DBw) \in L^1 (Q^\circ)$ and, by standard Sobolev trace theorem, 
$(\psi \DBw )_{|\partial Q} \cdot \DBn_Q \in L^1(\partial Q)$, with $\DBn_Q$ the outer normal field of $Q^\circ$.
Applying the Sobolev Gauss--Green formula on bounded Lipschitz domains
(see Section~1.5.3 of Grisvard~\cite{Grisvard} and Chapter~3,
Section~1, Theorem~1.1 of Ne\v{c}as~\cite{Necas1967}) yields
\[
= \!\!\int_{\partial Q} \hspace{-4pt}\psi \DBw \!\cdot\! \DBn_Q d\mathcal H^n
= - \int_{G(\sigma)} \hspace{-4pt}\psi \DBw \!\cdot\! \DBe_1 d\DBx + \int_{G(\tau)} \hspace{-4pt}\psi \DBw \!\cdot\! \DBe_1 d\DBx
+\int_{{\rm gr}(\partial G)} \hspace{-4pt}\psi \DBw \!\cdot\! \DBn_Q d\mathcal H^n
\]
with the surface measure $d\mathcal H^n$ on $\partial Q$, being $d\DBx$ on the boundary parts corresponding to the time slices at $t=\sigma, \tau$; cf.\ Figure~\ref{DBFig_tube}. Since $\DBw = (1,\DBv)$ is tangential to ${\rm gr}(\partial G)$, because the latter is flow invariant as a consequence of the unique solvability of (\ref{DBeq:kinematic-ODE1}), the last integrand vanishes due to $\DBw \cdot \DBn_Q =0$ $\mathcal H^n$-a.e.\ on ${\rm gr}(\partial G)$. 
We thus obtain
\[
\int_{G(\tau)} \hspace{-4pt}\psi  (\tau,\DBx) d\DBx
= \int_{G(\sigma)}  \hspace{-4pt}\psi (\sigma,\DBx)  d\DBx +
\int_\sigma^\tau \int_{G(t)}  \hspace{-4pt} (\partial_t \psi + \! \DBdiv (\psi \DBv ) ) (t,\DBx) d\DBx \, dt
\]
for all $[\sigma,\tau]\subset J_0$. 
Thus the map \(t\mapsto\int_{G(t)}\psi(t,\DBx)\,d\DBx\) is absolutely continuous on \(J_0\), 
hence the differential identity holds for a.a.\ \(t\in J_0\). If \(\psi\in C^1(J_0\times\Omega)\), the
right-hand side is continuous in \(t\), and the identity holds for every
\(t\in J_0\).
\end{proof}
A few comments are in order.\\
\indent
(i) Since $G_0\Subset\Omega$, no extension of the flow to $\partial\Omega$ is needed. On a compact neighborhood of $G_0$ contained in $\Omega$, the $C^1$-regularity of $\DBv$ gives uniform bounds and a uniform spatial Lipschitz constant. After decreasing the time interval if necessary, all solutions starting in $G_0$ remain in this neighborhood; hence a common non-trivial interval of existence exists and the choice of $r>0$ in the assumptions is possible.\\
\indent
(ii) The identification of \(\psi\) with its Bochner--Sobolev representative,
which is continuous from \(J_0\) to \(L^1(\Omega)\), ensures that
\[
  \int_{G(t)}\psi(t,\DBx)\,d\DBx
\]
is well defined for every \(t\in J_0\). On the endpoint time slices appearing
in the space-time divergence theorem, the Sobolev trace agrees with these
Bochner slice values in the usual sense.\\
\indent
(iii) In Theorem~\ref{DBthm:1ph-transport-theorem-comovingCV}, we used the strong assumption $\DBv\in C^1 (J\times \Omega)$ to infer directly that 
$Q$ is a compact Lipschitz domain, equivalently that $Q^\circ$ is a bounded Lipschitz domain. We will see in section~\ref{BIF-RTT} that for the divergence form of the RTT a local-in-time version follows for $\DBv\in L^\infty (J\times \Omega)$, being Lipschitz continuous in $\DBx$. This works in particular because the flow map is then a small Lipschitz perturbation of the identity for short times.\\
\indent
(iv) In section~\ref{BIF-RTT}, we will obtain the RTT under weaker assumptions as a corollary to Theorem~\ref{thm:single-phase-boundary-rtt},
an appropriate boundary-integral form of the RTT. 
As explained above, the latter will be needed in the proof of the two-phase RTT.\\[1ex]
The remainder of the paper is organised as follows.
The next section~\ref{sec:SIframework} introduces the framework for sharp-interface modeling, in particular the standard
regularity assumptions for a family of moving hypersurfaces and its intrinsic speed of normal displacement.
These notions also (partly) apply to the evolving family $\partial G(\cdot)$. This is used in section~\ref{BIF-RTT} to prove the 
boundary-integral form of the Reynolds theorem, viz.\ equation (\ref{DBeq:RTT}), without assuming an underlying flow map.
The proof uses the approach via the space-time divergence theorem.
Section~\ref{sec:kinematics} provides information about possible complications in the kinematics close to a fluid interface in case phase change
and slippage at the interface occur simultaneously. We give a rather simple example for which non-unique solvability occurs. 
This motivates the introduction of a generalised concept of co-moving sets in section~\ref{sec:flow-map}, based on a set-valued regularisation of
the discontinuous velocity field. This finally allows to formulate and prove the two-phase version of the Reynolds transport theorem in section~\ref{sec:two-phase-RTT}.
In the example mentioned above, an initially smooth set $G_0$ (a disk in $\R^2$) instantaneously loses boundary regularity, remaining a Lipschitz domain for short times before it develops a cusp after finite time. Even more, the part of $G(t)$ inside one phase instantaneously develops a cusp.
This indicates the complexity of the two-phase kinematics in this general case and requires some new ideas to prove the two-phase RTT.
\section{Sharp-Interface Modeling Framework}\label{sec:SIframework}
We consider two fluids, occupying sub-domains $\Omega^\pm(t)$ of a domain $\Omega \subset \R^n$, the so-called bulk phases.
We consider the two fluids to be immiscible on the molecular scale, i.e.\ they are separated by a sharp interface $\Sigma (t)$.
In order to have a well-defined local curvature, we assume $\Sigma(t)$ for any $t\in J=(a,b)$ to be a
$C^2$-surface, where $J=(a,b)\subset \R$. In particular, $\Omega^+(t) \cup \Omega^-(t) \cup \Sigma(t)$ is a disjoint decomposition of $\Omega$.
To avoid technical problems with moving contact lines, we assume that $\Sigma(t)$ is an embedded surface in $\Omega \subset \R^n$ without boundary. 
We employ the following definition of a $\mathcal{C}^{1,2}$-family of moving hypersurfaces which,
in a similar form, can also be found in \cite{Kimura.2008}, \cite{PrSi15}, \cite{barrett2020parametric} and \cite{Giga.2006}.
\begin{definition}\label{def:C12-family}
Let $J=(a,b)\subset \R$ be an open interval. A family $\{\Sigma(t)\}_{t \in J}$ with $\Sigma(t) \subset \R^n$ is called a
$\mathcal{C}^{1,2}$-\emph{family of moving hypersurfaces} if
\begin{enumerate}
\item[(i)]
the graph of $\Sigma$, defined as\vspace{-0.05in}
\begin{equation}
\mathcal{M} := {\rm gr}(\Sigma):=\{(t,\DBx): t \in J, \DBx \in \Sigma(t) \} = \bigcup_{t \in J} \Big( \{t\} \times \Sigma(t) \Big)\vspace{-0.05in}
\end{equation}
is a $\mathcal{C}^1$-hypersurface in $\R \times \R^n$ with (continuous) unit normal field $\DBn_\mathcal{M}$ such that
$\DBn_\mathcal{M}^\DBx (t,\DBx) \neq 0$ for all $(t,\DBx)\in \mathcal{M}$, where $\DBn_\mathcal{M} = (\DBn_\mathcal{M}^t, \DBn_\mathcal{M}^\DBx)$
is the decomposition of $\DBn_\mathcal{M}$ into its time- and space-components;
\item[(ii)]
each section $\Sigma(t)=\mathcal{M}\cap [\{t\}\times \R^n]$ 
is a $\mathcal{C}^2$-hypersurface in $\R^n$, with unit normal field $\DBn_\iface=\DBn_\iface(t,\cdot)$ oriented such that
$\DBn_\iface = \DBn_\mathcal{M}^\DBx / ||\DBn_\mathcal{M}^\DBx||$;
\item[(iii)]
the normal field $\DBn_\iface$ is continuously differentiable on $\mathcal{M}$, i.e.\
$\DBn_\Sigma \in \mathcal{C}^1(\mathcal{M})$.
\end{enumerate}
\vspace{0.05in}
\end{definition}
\noindent
The property of $\mathcal{M}=\gr (\iface)$ to be an oriented $C^1$-hypersurface with $\DBn_\mathcal{M}^\DBx\neq 0$
implicitly contains a kinematic property of $\{\iface(t)\}_{t \in J}$, viz.\ that the time-dependence is associated with a local motion of points in $\iface$ in normal direction in space with a certain speed of normal displacement.
\begin{definition}[Speed of normal displacement]\label{def:speed-of-normal-displacement}
Let $J=(a,b)\subset \R$ be an open interval and $\{\iface(t)\}_{t \in J}$ be a family of moving closed hypersurfaces in $\R^n$ such that $\mathcal{M}=\gr (\iface)$ is
an oriented $C^1$-hypersurface with $\DBn_\mathcal{M}^\DBx \neq 0$ on $\mathcal{M}$.
Then the \emph{speed of normal displacement} $V_\iface (t,\DBx)$ of $\iface(\cdot)$ is defined as\vspace{-0.1in}
\begin{equation}\label{DBdef:speed-of-normal-displacement}
V_\iface (t,\DBx) = -\, \frac{\DBn_\mathcal{M}^t (t,\DBx)}{||\DBn_\mathcal{M}^\DBx (t,\DBx)||}\;
\mbox{ for } (t,\DBx) \in \mathcal{M}.
\end{equation}
From the regularity of the normal field of $\mathcal{M}$, namely $\DBn_\mathcal{M}\in C(\mathcal{M})$,
it is clear that $V_\iface$ is jointly continuous on $\mathcal{M}$.
\end{definition}
\indent
Associated with $V_\iface$ is the \emph{intrinsic velocity field}
\begin{equation}\label{DBeq:normal-velocity-field}
\DBw_\Sigma (t,\DBx) = V_\iface (t, \DBx) \, \DBn_\iface (t,\DBx) \; \mbox{ for } t\in J,\, \DBx\in \iface (t),
\end{equation}
which also is jointly continuous under the assumption that $\mathcal{M}=\gr (\iface)$ is an oriented
$C^1$-hypersurface of $\R^{n+1}$. Then the extended field $(1, \DBw_\Sigma (t,\DBx))$ is tangential to $\mathcal{M}$, since
\[
\langle (1,\DBw_\Sigma (t,\DBx)), \DBn_\mathcal{M} \rangle = 
\DBn_\mathcal{M}^t + \langle  V_\iface (t, \DBx) \, \frac{\DBn_\mathcal{M}^\DBx }{ ||\DBn_\mathcal{M}^\DBx||}, \DBn_\mathcal{M}^\DBx \rangle =0.
\]
The joint continuity of $\DBw_\Sigma (t,\DBx)$ together with the
property that $(1, \DBw_\Sigma (t,\DBx))$ is tangential to $\mathcal{M}$ is sufficient to (locally) solve the
initial value problems
\begin{equation}\label{DBeq:IVP-surface-normal-path}
\dot \DBx^\iface (t) = \DBw_\iface (t, \DBx^\iface (t)) \; \mbox{ for } t\in J,\quad \DBx^\iface (t_0)=\DBx^\iface_0,
\end{equation}
for every $t_0\in J$ and $\DBx^\iface_0\in \iface (t_0)$.
This follows, e.g., from Theorem~1 in \cite{MDE}; cf.\ also \cite{Bo2}.
For the uniqueness and smooth dependence needed below, we use the regularity consequence of
\cite[Lemma~3 and Corollary~1]{Bo-2PH-ODE}: for a $\mathcal C^{1,2}$-family of moving
hypersurfaces, the signed-distance function $d_\iface$ and the nearest-point projection $\pi_\iface$ are $C^1$ in a
tubular neighborhood of $\mathcal M$, $\nabla_\DBx d_\iface$ is $C^1$, and\vspace{-0.05in}
\[
        \nabla_\DBx d_\iface(t,\DBx)=\DBn_\iface(t,\pi_\iface(t,\DBx)).
\]
Moreover, $V_\iface\in C(\mathcal M)$ and $\nabla_\iface V_\iface\in C(\mathcal M)$. Hence
\[
        \DBw_\iface=V_\iface\DBn_\iface\in C(\mathcal M),\qquad
        \nabla_\iface \DBw_\iface\in C(\mathcal M),
\]
so that $\DBw_\iface(t,\cdot)$ is locally Lipschitz on each $\iface(t)$. Consequently the
surface ODE \eqref{DBeq:IVP-surface-normal-path} is locally uniquely solvable. In particular,
for every compact subset of $\iface(t_0)$, after shrinking the time interval if necessary,
it generates a $C^1$ surface flow on a neighborhood of that compact set into $\iface(t)$.
Observe also that solvability of the initial value problems (\ref{DBeq:IVP-surface-normal-path}) implicitly contains the  constraint
$\DBx^\iface (t)\in \iface(t)$ for $t\in J$.
As $\DBx^\iface (t+h)=\DBx^\iface (t)+h \DBw_\iface (t, \DBx^\iface (t)) + o(h)$ as $h\to 0$, this implies the relation
\begin{equation}\label{VSigma}
\lim_{h\to 0+} \frac 1 h {\rm dist} (\DBx+h V_\Sigma (t,\DBx) \DBn_\Sigma (t,\DBx) , \Sigma (t+h)) =0 \quad \mbox{ for } t\in J, \DBx\in \Sigma(t),
\end{equation}
explaining why $V_\Sigma$ is called \emph{speed of normal displacement} of $\Sigma(\cdot)$.

The computation of  $V_\Sigma$ is especially simple
if $\{\Sigma(t)\}_{t \in J}$ has a level set description, i.e.\
\begin{equation}\label{levelset}
\Sigma(t)=\{\DBx \in \R^n: \phi(t,\DBx)=0\}
\end{equation}
with $\phi \in \mathcal{C}^{1,2}(\mathcal N)$ for some open neighborhood $\mathcal N \subset \R \times \R^n$ of $\mathcal M$
such that $\nabla_{\DBx} \phi \not= 0$ on $\mathcal{M}$, oriented so that
$\nabla_{\DBx}\phi/\|\nabla_{\DBx}\phi\|=\DBn_\Sigma$ on $\mathcal M$. Then\vspace{-0.1in}
\begin{equation}\label{E11}
V_\Sigma(t,\DBx)=- \, \frac{\partial_t \phi(t,\DBx)}{\|\nabla_{\DBx} \phi(t,\DBx)\|}\quad\text{ for } t \in J, \, \DBx \in \Sigma(t),
\end{equation}
which precisely resembles (\ref{DBdef:speed-of-normal-displacement}).

Let us also note that Definition~\ref{def:speed-of-normal-displacement} is equivalent to the more common one which employs curves.
Indeed,\vspace{-0.05in}
\[
V_\Sigma (t,\DBx) = \langle \gamma' (t) , \DBn_\Sigma (t, \gamma (t)) \rangle
\]
for any $\mathcal{C}^1$-curve $\gamma$ with $\gamma (t)=\DBx$ and ${\rm gr}(\gamma)\subset \mathcal M$,
and the value does not depend on the choice of a particular curve; cf.\ Chapter~2.5 in \cite{PrSi15}.

Now, in the two-phase flow setting mentioned above with the evolving interface $\{\Sigma(t)\}_{t \in J}$ being a family of moving hypersurfaces $\Sigma(t) \subset \R^n$ and a given (control volume) $V\subset \R^n$ bounded with $\partial V\in C^1$,
let $V^\pm (t):=\Omega^\pm (t) \cap V$. Given also a sufficiently regular $\psi = \psi(t,\DBx)$, application of the
Reynolds transport theorem to both $V^\pm (t)$ yields
\begin{equation}\label{DBeq:RTT_fixedCV}
\frac{d}{dt} \int_{V} \psi \,d\DBx =  \int_{V\setminus \Sigma(t)} \partial_t \psi  \,d\DBx 
- \int_{\Sigma_V(t)} [\![ \psi  ]\!] V_\Sigma  \,dS,
\end{equation}
where $\Sigma_V(t):= \Sigma(t)\cap V$. The orientation of
$\DBn_\Sigma$ is arbitrary but fixed, and the jump bracket
$[\![\, \cdot \,]\!]$ is understood with respect to this orientation. It is
defined as
\begin{equation}\label{jump-bracket}
[\![ \psi]\!] (t,\DBx):=\lim\limits_{h \to 0+}
\big(\psi (t,\DBx+h\DBn_\Sigma (t,\DBx) )
- \psi (t,\DBx-h\DBn_\Sigma(t,\DBx) )\big)
\end{equation}
for $t\in J$, $\DBx \in \Sigma(t)$.

For a fixed control volume $V\subset \R^3$, the integral balance of mass reads
\begin{equation}\label{DBeq:mass-balance-fixed}
\frac{d}{dt} \int_{V} \rho \,d\DBx = - \int_{\partial V} \rho \DBv \cdot \DBn \,dS,
\end{equation}
as the mass flux density is given by $\rho \DBv$ due to (\ref{DBeq:continuity}).
This also holds true for the two-phase flow setting under the assumption that mass cannot be stored on the interface, i.e.\
$\rho^\Sigma \equiv 0$ for the area-specific mass density on the interface\footnote{We use a lower index $\Sigma$ for quantities in which $\Sigma$ enters as a purely geometric or kinematic object, while an upper index $\Sigma$ is used for physical interface quantities.}.
Applying (\ref{DBeq:RTT_fixedCV}), the fundamental lemma of the calculus of variations yields the local equations
\begin{align}
	&\partial_t \rho+ \div(\rho \DBv)=0 \quad \text{ in } \Omega\setminus \Sigma,\label{E1}\\
	&[\![\rho(\DBv-\DBv^\Sigma)\cdot \DBn_\Sigma ]\!] =0 \quad \text{ on } \Sigma.\label{E3}
\end{align}
Here and throughout we use a condensed notation as equations such as (\ref{E1}) are given for both phases in a single equation without use of
phase indices. If the dependence on the phase index is relevant, we use upper index $\pm$, e.g.\ $\rho^\pm$ for the mass densities.

Note that solely the speed of normal displacement enters in (\ref{E3}), as  $\DBv^\Sigma \cdot \DBn_\Sigma=V_\Sigma$.
Letting $\dot{m}^\pm = \rho^\pm (\DBv^\pm -\DBv^\Sigma)\cdot \DBn_\Sigma$ denote the one-sided mass transfer rates,
the jump condition (\ref{E3}) becomes $[\![ \dot m ]\!] =0$. Since $\dot m$ is only defined on $\Sigma$, anyhow,
this is essentially a shorthand notation encoding the equality $\dot{m}^+=\dot{m}^-$. 

The balance of momentum rests upon Newton's second law of motion, stating that within an inertial frame of reference
the rate at which the momentum of a body is changing with time equals the net force on the body.
Here, the term 'body' refers to a certain piece of matter that is made up by the same material during the motion.
In other words, in the context of continuum physics, it refers to a material (control) volume.
Consequently, at least in the first place, the application of Newton's 2$^{\rm nd}$ law of motion requires a co-moving volume.
In a setting in which such a material co-moving volume $G(t)$ is well defined for $t\in J$, the integral form of the balance of momentum reads as
\begin{equation}\label{DBeq:momentum-balance-material}
\frac{d}{dt} \int_{G(t)} \rho \DBv d\DBx = \int_{\partial G(t)} \DBS \cdot \DBn \,dS 
+ \int_{G(t)} \rho \DBb \, d\DBx + \int_{\partial \Sigma_G(t)} \DBS^\iface \cdot \DBN \,d\ell,
\end{equation}
where $\DBS$ and $\DBS^\iface$ are the bulk and interface stress tensors, respectively, $\DBb$ denotes the density of body forces, and $\DBN$ is the outward pointing co-normal at the bounding curve of $\Sigma_G(t)$.
Even if a rigorous definition of the sets $G(t)$ is only to be given below, let us already note the resulting local form of the momentum balance. This reads as
\begin{align}
	&\partial_t(\rho \DBv)+ \div(\rho \DBv \otimes \DBv-\DBS)=\rho \DBb \quad \text{ in } \Omega\setminus\Sigma,\label{E2}\\
	&[\![\rho \DBv \otimes (\DBv-\DBv^\Sigma)-\DBS ]\!] \cdot
	\DBn_\Sigma  ={\rm div}_\Sigma \, \DBS^\Sigma \quad \text{ on } \Sigma,\label{E4}
\end{align}
where ${\rm div}_\Sigma$ denotes the surface divergence. 

Later, in hindsight, the momentum balance can also be formulated for a fixed control volume once the convective transport of momentum
across the boundary of the control volume is understood. Nevertheless, application of Newton's laws of motion motivates
a rigorous definition of co-moving volumes also in a two-phase flow setting.

The system (\ref{E1}), (\ref{E3}), (\ref{E2}), (\ref{E4}) requires several constitutive relations to arrive at a closed model, i.e.\ a system of PDEs for the unknown variables $\rho, \DBv$; see \cite{Slattery-Interfaces}, \cite{bothe2025multi} for more details.
Below, we shall employ the simplest closure for the interfacial stress tensor $\DBS^\Sigma$, given by $\sigma {\bf P}_\iface$ with the interfacial tension $\sigma$ and ${\bf P}_\iface={\bf I} - \DBn_\Sigma \otimes \DBn_\Sigma$. Then (\ref{E4}) becomes
\begin{equation}\label{eq-mom-jump}
\dot m \, [\![ \DBv ]\!]  - [\![ \DBS \cdot \DBn_\Sigma ]\!] 
=  \sigma \kappa_\iface \DBn_\Sigma + \nabla_\Sigma \sigma \quad \text{ on } \Sigma,
\end{equation}
where $\kappa_\iface = {\rm div}_\Sigma ( - \DBn_\Sigma)$ denotes the total curvature of $\iface$, i.e.\ the sum of the principal curvatures.
Here, we are mainly interested in the flow generated by the two-phase velocity field. 
For this purpose we need to add a jump condition for the tangential part, where the latter is written as $\DBv_{||}$.
While the standard no-slip condition requires continuity of $\DBv_{||}$, i.e.
\begin{equation}\label{E5}
[\![ \DBv_{||} ]\!]=0 \quad \text{ on } \Sigma,
\end{equation}
we will focus on the general case of possible slip of the fluid phases at $\iface$.
%
%
%
%
\section{Boundary-integral form of the single-phase RTT}\label{BIF-RTT}
We prove a boundary-integral version of the Reynolds transport theorem
which will be used below for domains whose motion is known only through
the normal velocity of the boundary. This form is independent of whether
the motion is generated by a velocity field in the full surrounding space.

Let \(J=(a,b)\subset\R\), and let \(G(t)\subset\R^n\),
\(t\in J\), be a family of compact regular closed sets. For
\(\sigma,\tau\in J\), \(\sigma<\tau\), set
\[
\operatorname{gr}(G|_{[\sigma,\tau]})
  =
  \{(t,\DBx):t\in[\sigma,\tau],\ \DBx\in G(t)\},\vspace{-0.05in}
\]
and define the associated space-time tube interior and its closure by\vspace{-0.05in}
\[
  Q_{\sigma,\tau}^\circ
  :=
  \{(t,\DBx)\in(\sigma,\tau)\times\R^n:
  \DBx\in G(t)^\circ\},
  \qquad
  Q_{\sigma,\tau}:=\overline{Q_{\sigma,\tau}^\circ}.\vspace{-0.05in}
\]
We call\vspace{-0.05in}
\[
M_{\sigma,\tau}
 :=
 \{(t,\DBx)\in(\sigma,\tau)\times\R^n:\ \DBx\in\partial G(t)\}
\]
the lateral boundary of this space-time tube\footnote{Throughout, we call the graph of a family of moving sets a space-time tube. This is not to be confused with a tubular neighborhood.}.
We assume that \(Q_{\sigma,\tau}^\circ\) is a bounded Lipschitz domain. Its (measure-theoretic)
outer unit normal, understood as the outer normal to \(Q_{\sigma,\tau}^\circ\),
is denoted by \(\DBn_{Q}\). Whenever \(M_{\sigma,\tau}\) represents the lateral part of \(\partial Q_{\sigma,\tau}\), as assumed in the theorem below, we write
\[
\DBn_M:=\DBn_{Q|_{M_{\sigma,\tau}}} =
 \bigl(\DBn^t_{M},\DBn^\DBx_{M}\bigr)
 \in \R\times\R^n .
\]
Thus the orientation of the graph \(M_{\sigma,\tau}\) is induced by the exterior direction of the swept-out
space-time region \(Q_{\sigma,\tau}\).

On the regular part of \(M_{\sigma,\tau}\), and hence
\(\mathcal H^n\)-almost everywhere on \(M_{\sigma,\tau}\), we assume
that
\[
\DBn^\DBx_{M}\neq 0 .
\]
In analogy with Definition~2.2, we define the spatial outer unit normal
to \(G(t)\) and the speed of normal displacement of \(\partial G(\cdot)\)
by
\[
\DBn_G(t,\DBx)
 :=
 \frac{\DBn^\DBx_{M}(t,\DBx)}
      {\|\DBn^\DBx_{M}(t,\DBx)\|},
 \qquad
V_n(t,\DBx)
 :=
 -
 \frac{\DBn^t_{M}(t,\DBx)}
      {\|\DBn^\DBx_{M}(t,\DBx)\|}.
\]
Thus, on a smooth moving boundary patch,
\[
\DBn_{M}
 =
 \frac{(-V_n,\DBn_G)}
      {\sqrt{1+V_n^2}}.
\]
In particular, \(V_n\) is positive if \(\partial G(t)\) moves in the
direction of the spatial outer normal \(\DBn_G\).

\begin{theorem}[Boundary form of the single-phase RTT]\label{thm:single-phase-boundary-rtt}
Let \(G(t)\subset\R^n\) be compact, regular closed\,\footnote{A set $K\subset \R^n$ is called regular closed if $K=\overline{K^\circ}$.} sets for $t\in J:=(a,b)\subset\R$. Given \(\sigma,\tau\in J\) with \(\sigma<\tau\), assume that
\[
  Q_{\sigma,\tau}^\circ =
  \{(t,\DBx)\in(\sigma,\tau)\times\R^n:
  \DBx\in G(t)^\circ\}
\]
is a bounded Lipschitz domain, whose boundary allows for the decomposition
\[
  \partial Q_{\sigma,\tau}
  =
  (\{\sigma\}\times G(\sigma))
  \cup
  (\{\tau\}\times G(\tau))
  \cup
  M_{\sigma,\tau}
\]
up to \(\mathcal H^n\)-null sets.
Moreover, assume that
\(\DBn_M^\DBx\neq0\) \(\mathcal H^n\)-almost everywhere on \(M_{\sigma,\tau}\).

Let \(\phi\in C(Q_{\sigma,\tau})\), and assume that \(\phi\) has a
weak time derivative \(\phi_t\in L^1(Q_{\sigma,\tau}^\circ)\), i.e.
\[
  \int_{Q_{\sigma,\tau}^\circ}\phi\,\partial_t\zeta\,d(t,\DBx)
  =
  -\int_{Q_{\sigma,\tau}^\circ}\phi_t\,\zeta\,d(t,\DBx)
  \qquad
  \text{for all }\zeta\in C_c^\infty(Q_{\sigma,\tau}^\circ).\vspace{-0.05in}
\]
Then\vspace{-0.05in}
\begin{equation}\label{eq:int-RTT-bf}
\begin{aligned}
  \int_{G(\tau)}\phi(\tau,\DBx)\,d\DBx
  &=
  \int_{G(\sigma)}\phi(\sigma,\DBx)\,d\DBx
  +
  \int_\sigma^\tau\int_{G(t)}\phi_t(t,\DBx)\,d\DBx\,dt  \\
  &\quad+
  \int_\sigma^\tau\int_{\partial G(t)}
  \phi(t,\DBx)V_n(t,\DBx)\,dS\,dt .
\end{aligned}
\end{equation}
The lateral integral is understood through the signed-measure identity
\[
  \int_\sigma^\tau\int_{\partial G(t)}
  \phi V_n\,dS\,dt
  :=
  -\int_{M_{\sigma,\tau}}\phi\,\DBn_M^t\,d\mathcal H^n .
\]
In particular, no additional boundedness assumption on \(V_n\) is required.

If the assumptions hold on every compact subinterval of \(J\), then
\(t\mapsto\int_{G(t)}\phi(t,\DBx)\,d\DBx\) is absolutely continuous on compact
subintervals of \(J\), and
\[
  \frac{d}{dt}\int_{G(t)}\phi(t,\DBx)\,d\DBx
  =
  \int_{G(t)}\phi_t(t,\DBx)\,d\DBx
  +
  \int_{\partial G(t)}\phi(t,\DBx)V_n(t,\DBx)\,dS \quad \mbox{for a.a. } t\in J.
\]
If the right-hand side admits a continuous representative,
this identity holds for every \(t\in J\).
\end{theorem}

\begin{proof}
We suppress the subscript ${\sigma,\tau}$ for better readability.

1. We first assume in addition that \(\phi_{|Q^\circ}\in W^{1,1}(Q^\circ)\), and define
\[
  F(t,\DBx):=(\phi(t,\DBx),0)\in\R\times\R^n.
\]
Then \(F\in W^{1,1}(Q^\circ;\R^{n+1})\) with $\operatorname{div}_{(t,\DBx)}F=\phi_t$.
The Sobolev Gauss--Green formula on the bounded Lipschitz domain \(Q^\circ\)
(see Section~1.5.3 of Grisvard~\cite{Grisvard} and Chapter~3,
Section~1, Theorem~1.1 of Ne\v{c}as~\cite{Necas1967}) gives
\[
  \int_{Q^\circ} \phi_t\,d(t,\DBx)
  =
  \int_{\partial Q}\phi\,\DBn_Q^t\,d\mathcal H^n.
\]
Since \(\phi\in C(Q)\), its Sobolev trace on \(\partial Q\) agrees
with its pointwise boundary values. Using the boundary decomposition of \(Q\),
with outer normal \((-1,0)\) on the lower time slice and \((1,0)\) on the
upper time slice, this gives
\[
  \int_{Q^\circ} \phi_t\,d(t,\DBx)
  =
  -\int_{G(\sigma)}\phi(\sigma,\DBx)\,d\DBx
  +
  \int_{G(\tau)}\phi(\tau,\DBx)\,d\DBx
  +
  \int_M \phi\,\DBn_M^t\,d\mathcal H^n .
\]
On \(M\), the coarea formula (see, e.g., \cite{leoni2017first}) for the time
projection \((t,\DBx)\mapsto t\) yields
\[
  d\mathcal H^n
  =
  \frac{1}{\|\DBn_M^\DBx\|}\,dS\,dt .
\]
Together with
$V_n=- \DBn_M^t / \|\DBn_M^\DBx\|$,
this gives, as signed measures on \(M\),
\[
  \DBn_M^t\,d\mathcal H^n=-V_n\,dS\,dt .
\]
Hence
\[
  \int_M \phi\,\DBn_M^t\,d\mathcal H^n
  =
  -\int_\sigma^\tau\int_{\partial G(t)}
  \phi(t,\DBx)V_n(t,\DBx)\,dS\,dt .
\]
This identity also shows why no separate assumption on \(V_n\) is needed:
\[
  |V_n|\,dS\,dt
  =
  |\DBn_M^t|\,d\mathcal H^n
  \le d\mathcal H^n
  \quad\text{on }M,
\]
and \(M\) has finite \(\mathcal H^n\)-measure because \(Q^\circ\) is a bounded
Lipschitz domain. The preceding identity hence gives (\ref{eq:int-RTT-bf}) under the additional
assumption \(\phi_{|Q^\circ}\in W^{1,1}(Q^\circ)\).

To remove this auxiliary assumption, we apply
Lemma~\ref{lem:one-directional-approximation} with \(U=Q^\circ\subset\R^{n+1}\) and \(D_i=\partial_t\).
Thus there are \(\phi_j\in C^\infty(\R\times\R^n)\) such that
\[
  \phi_j\to\phi
  \quad\text{uniformly on }Q,
  \qquad
  \partial_t\phi_j\to\phi_t
  \quad\text{in }L^1(Q^\circ).
\]
The identity already proved under the auxiliary assumption \(\phi_{|Q^\circ}\in W^{1,1}(Q^\circ)\) applies to
each \(\phi_j\), and the limit \(j\to\infty\) gives the asserted formula.

2. If the assumptions hold for all $[\sigma ,\tau ]\subset J$, the integral identity (\ref{eq:int-RTT-bf}) holds on every such subinterval. 
The function\vspace{-0.05in}
\[
  t\mapsto \int_{G(t)}\phi_t(t,\DBx)\,d\DBx
\]
belongs to \(L^1_{\mathrm{loc}}(J)\), and the lateral contribution is locally
integrable since
\[
  \int_\sigma^\tau\int_{\partial G(t)}
  |\phi V_n|\,dS\,dt
  =
  \int_M |\phi\DBn_M^t|\,d\mathcal H^n
  \le
  \|\phi\|_{L^\infty(M)}\mathcal H^n(M)<\infty .
\]
Therefore \(t\mapsto\int_{G(t)}\phi(t,\DBx)\,d\DBx\) is absolutely continuous
on compact subintervals, and differentiation gives the asserted identity for
almost all \(t\in J\).

Finally, if the right-hand side as a function of \(t\) has a continuous
representative \(R(t)\), then (\ref{eq:int-RTT-bf}) gives\vspace{-0.05in}
\[
  \int_{G(\tau)}\phi(\tau,\DBx)\,d\DBx
  =
  \int_{G(\sigma)}\phi(\sigma,\DBx)\,d\DBx
  +
  \int_\sigma^\tau R(t)\,dt .
\]
The fundamental theorem of calculus therefore shows that
\(t\mapsto\int_{G(t)}\phi(t,\DBx)\,d\DBx\) is \(C^1\), with derivative \(R(t)\)
for every \(t\in J\).\vspace{0.1in}
\end{proof}

\begin{remark}
The assumption that \(Q_{\sigma,\tau}\) is a compact Lipschitz domain with the
natural time-slice/lateral boundary decomposition is not merely a pointwise
spatial regularity assumption on the slices \(G(t)\). When it is imposed on
every subinterval \([\sigma,\tau]\subset J\), it also controls the geometry
created by the artificial time faces. This is a strong space-time regularity
requirement, although it does not force the individual slices \(G(t)\) to be
Lipschitz domains in space, as the following example shows.
Let
\(K\subset\R^2\) be a compact regular closed set with an inward cusp,
hence not a compact Lipschitz domain, and let \(\rho\) be the signed distance to \(\partial K\), chosen negative in
\( K^\circ \) and positive in \(\R^2\setminus K\).
If
\[
        G(t):=\{x:\rho(x)\le t\},
\]
then \(G(0)=K\). However, for every \([\sigma,\tau]\) close to \(0\),
\[
        \operatorname{gr}(G_{|[\sigma,\tau]})
        =
        \{(t,x):\sigma\le t\le\tau,\ \rho(x)\le t\},
\]
whose lateral boundary is the Lipschitz graph \(t=\rho(x)\), with the natural
time-slice/lateral boundary decomposition. The RTT is hence formulated directly
in terms of the space-time tube, rather than by imposing Lipschitz regularity
on the individual spatial slices.

Thus the hypothesis on the space-time tube  contains information on the time evolution of the lateral
boundary $M_{\sigma,\tau}$
and not primarily on the individual sets \(G(t)\). Geometrically, it encodes a certain
transversality of the lateral boundary with respect to the time slices.
Equivalently, the restricted space-time tube must satisfy a two-sided
cone condition near the lateral boundary, in particular at the intersections
with the artificial time faces.
This relates to uniform bounds on the speed of normal displacement. Indeed, the estimate \(|V_n|\le C\) implies\vspace{-0.1in}
\[
        |\DBn_M^\DBx|\ge \frac{1}{\sqrt{1+C^2}},
\]
so that the space-time normal cannot become purely temporal.
In particular, an instantaneous spatial jump of the sets \(G(t)\) is excluded:
such a jump would create an additional time-slice type boundary part inside
the interval, with purely temporal space-time normal.
\end{remark}

\begin{remark}
One could avoid the approximation step in the proof of Theorem~\ref{thm:single-phase-boundary-rtt} by using a
Gauss--Green theorem for divergence-measure fields; see, for instance,
\cite{ChenTorresZiemer, Schurichtbook}. Indeed, for \(F=(\phi,0)\) one has\vspace{-0.05in}
\[
  \operatorname{div}_{(t,\DBx)}F=\phi_t\in L^1(Q^\circ),
\]
but \(F\) need not belong to \(W^{1,1}(Q^\circ;\R^{n+1})\) as no spatial
weak derivatives of \(\phi\) are assumed. The approach above keeps the argument
within the classical Sobolev Gauss--Green formula.
\end{remark}

\begin{remark}\label{rem:Sobolev-form}
There is also a simpler Sobolev-trace version of Theorem~\ref{thm:single-phase-boundary-rtt}.
If \(\phi\in W^{1,1}(J\times\Omega)\), then
\(\phi\in W^{1,1}(J;L^1(\Omega))\) and we identify \(\phi\) with its
Bochner--Sobolev representative in \(C(J;L^1(\Omega))\). The same proof,
without the approximation step, gives the integrated identity with
\(\phi_t=\partial_t\phi\). In this formulation the lateral term is understood
with the Sobolev trace of \(\phi_{|Q_{\sigma,\tau}^\circ}\) on \(M_{\sigma,\tau}\),
namely
\[
  \int_\sigma^\tau\int_{\partial G(t)}
  \operatorname{Tr}_M\phi\,V_n\,dS\,dt
  :=
  -\int_{M_{\sigma,\tau}}
  \operatorname{Tr}_M\phi\,\DBn_M^t\,d\mathcal H^n.\vspace{0.1in}
\]
\end{remark}
\noindent
We record the following divergence form RTT in case the motion comes from a flow map.

\begin{corollary}[Flow-map version in divergence form]
\label{cor:flow-map-rtt}
Let \(J=(a,b)\subset\R\), let \(\Omega\subset\R^n\) be open, and let
\(\DBv\in L^\infty(J\times\Omega;\R^n)\) be uniformly Lipschitz in space, i.e.
\[
  |\DBv(t,\DBx)-\DBv(t,\DBy)|
  \le L|\DBx-\DBy|
  \quad \mbox{for a.e.\ \(t\in J\) and all } \DBx,\DBy\in\Omega.
\]
Let
\(t_0\in J\), \(G_0\subset\Omega\) be a compact Lipschitz domain and choose $r>0$ such that $G(t):=\Phi^t_{t_0}(G_0)$ is well-defined for all $t\in J_0:=(t_0-r, t_0+r)\cap J$.

Then there exists \(r^\ast\in(0,r)\) such that, with $ J_0^\ast:=(t_0-r^\ast,t_0+r^\ast)\cap J$, one has 
\[
  \frac{d}{dt}\int_{G(t)}\phi(t,\DBx)\,d\DBx
  =
  \int_{G(t)}
  \bigl(\partial_t\phi+\operatorname{div}(\phi\DBv)\bigr)(t,\DBx)\,d\DBx \quad \mbox{a.e.\ on } J_0^\ast
\]
for every\vspace{-0.1in}
\[
  \phi\in W^{1,1}(J_0^\ast\times\Omega),
\]
where $\phi$ is identified with its Bochner--Sobolev representative in
\[
  W^{1,1}(J_0^\ast;L^1(\Omega))
  \hookrightarrow C(J_0^\ast;L^1(\Omega)).\vspace{-0.1in}
\]
If, in addition,\vspace{-0.05in}
\[
  \DBv\in C(J_0\times\Omega;\R^n),
  \qquad
  \nabla_\DBx\DBv\in C(J_0\times\Omega;\R^{n\times n}),
\]
then one may take \(r^\ast=r\). If also
\(\phi\in C^1(J_0\times\Omega)\), then the identity holds for every
\(t\in J_0\).
\end{corollary}

\begin{proof}
For short times, the flow map is a small Lipschitz perturbation of the
identity. Indeed, Gronwall's inequality gives
\[
  \operatorname{Lip}(\Phi^t_s-\operatorname{Id})
  \le e^{L|t-s|}-1 .
\]
Choosing \(r^\ast>0\) sufficiently small,
Corollary~\ref{cor:closed-lipschitz-small-perturbation} implies that
\(G(t)=\Phi^t_{t_0}(G_0)\) is a compact Lipschitz domain for
\(t\in J_0^\ast\).

The space-time tube requires one additional point: the time variable has to be
rescaled before applying the small-perturbation result. Indeed, for
\(|s-t_0|,|t-t_0|<r^\ast\),
\[
\begin{aligned}
&\left|
  \bigl(\Phi^t_{t_0}(\DBx)-\DBx\bigr)
  -
  \bigl(\Phi^s_{t_0}(\DBy)-\DBy\bigr)
\right|  \\
&\qquad\le
\left|
  \bigl(\Phi^t_{t_0}(\DBx)-\DBx\bigr)
  -
  \bigl(\Phi^t_{t_0}(\DBy)-\DBy\bigr)
\right|
+
\left|\Phi^t_{t_0}(\DBy)-\Phi^s_{t_0}(\DBy)\right|  \\
&\qquad\le
\bigl(e^{Lr^\ast}-1\bigr)|\DBx-\DBy|
+
\|\DBv\|_{L^\infty}\,|t-s|.
\end{aligned}
\]
The last coefficient is not made small by decreasing \(r^\ast\). We therefore
fix a number \(a>0\) and write \(t=t_0+a\theta\). For
\[
  \mathcal S_a(\theta,\DBx)
  :=\bigl(\theta,\Phi^{t_0+a\theta}_{t_0}(\DBx)\bigr)
\]
one obtains
\[
\begin{aligned}
&\left|
  \bigl(\Phi^{t_0+a\theta}_{t_0}(\DBx)-\DBx\bigr)
  -
  \bigl(\Phi^{t_0+a\eta}_{t_0}(\DBy)-\DBy\bigr)
\right| \\
&\qquad\le
\bigl(e^{Lr^\ast}-1\bigr)|\DBx-\DBy|
+
 a\|\DBv\|_{L^\infty}\,|\theta-\eta|.
\end{aligned}
\]
Hence, choosing first \(a>0\) sufficiently small and then decreasing
\(r^\ast>0\) if necessary, the perturbation
\(\mathcal S_a-\operatorname{Id}\) has Lipschitz constant below the
small-perturbation threshold from
Corollary~\ref{cor:closed-lipschitz-small-perturbation}. This threshold can be
chosen uniformly for the Cartesian products
\[
  [ (\sigma-t_0)/a, (\tau-t_0)/a ]\times G_0,
  \qquad [\sigma,\tau]\subset J_0^\ast,
\]
by the compact-domain consequence of
Proposition~\ref{prop:Cartesian-product-Lipschitz-domain} and its uniform
cone-angle statement. Thus Corollary~\ref{cor:closed-lipschitz-small-perturbation}
shows that the rescaled tubes
\[
  \widetilde Q_{\sigma,\tau}
  :=
  \mathcal S_a
  \bigl([ (\sigma-t_0)/a, (\tau-t_0)/a ]\times G_0\bigr)
\]
are compact Lipschitz domains.

The original space-time tubes are obtained from these rescaled tubes by the
linear \(C^1\)-diffeomorphism\vspace{-0.05in}
\[
  R_a(\theta,\DBx):=(t_0+a\theta,\DBx),
  \qquad
  Q_{\sigma,\tau}=R_a(\widetilde Q_{\sigma,\tau}).
\]
By Lemma~\ref{DBthm:C1-invariance-bounded-Lipschitz}, each
\(Q_{\sigma,\tau}\) is therefore a compact Lipschitz domain. Moreover,
\(R_a\circ\mathcal S_a\) is the original time-preserving homeomorphism
\[
  (\theta,\DBx)\mapsto
  (t_0+a\theta,\Phi^{t_0+a\theta}_{t_0}(\DBx)).
\]
Consequently, the product boundary decomposition of the rescaled cylinder is
mapped onto\vspace{-0.05in}
\[
        (\{\sigma\}\times G(\sigma))
        \cup
        (\{\tau\}\times G(\tau))
        \cup
        M_{\sigma,\tau},
\]
up to the edge sets of codimension two. Hence \(Q_{\sigma,\tau}\) has the
required boundary decomposition into time-slice and lateral parts.
On the lateral boundary, the vector \((1,\DBv)\) is tangential to the
space-time graph of the transported boundary. Hence
\[
  \DBn_M^t+\DBv\cdot\DBn_M^\DBx=0
  \quad\mathcal H^n\text{-a.e.},
\]
and therefore $V_n=\DBv\cdot\DBn_G$.

The Sobolev-trace version of Theorem~\ref{thm:single-phase-boundary-rtt} from Remark~\ref{rem:Sobolev-form}
therefore gives the boundary form. For almost every $t$, the lateral Sobolev trace occurring there agrees with the spatial Sobolev trace of $\phi(t,\cdot)$ on $\partial G(t)$; this follows, for instance, by approximating $\phi$ in $W^{1,1}$ by smooth functions and using the coarea formula on the lateral graph. Since
\(\DBv(t,\cdot)\) is Lipschitz and bounded on the relevant compact set, one has
\[
  \phi(t,\cdot)\DBv(t,\cdot)\in W^{1,1}(G(t)^\circ;\R^n)
\]
for a.a. \(t\in J_0^\ast\). The spatial Sobolev Gauss--Green formula on bounded Lipschitz domains
converts the boundary term into
\[
  \int_{G(t)}\operatorname{div}(\phi\DBv)(t,\DBx)\,d\DBx .
\]
This yields the asserted a.e. differential identity.

Under the additional regularity assumption, the flow map
\[
  (t,\DBx_0)\mapsto (t,\Phi^t_{t_0}(\DBx_0))
\]
is a \(C^1\)-diffeomorphism on the whole interval of existence \(J_0\). For
every compact subinterval \([\sigma,\tau]\subset J_0\), the cylinder
\([\sigma,\tau]\times G_0\) is a compact Lipschitz domain by the compact-domain
consequence of Proposition~\ref{prop:Cartesian-product-Lipschitz-domain}.
Therefore its image, the swept tube \(Q_{\sigma,\tau}\), is again a compact
Lipschitz domain by Lemma~\ref{DBthm:C1-invariance-bounded-Lipschitz}. Hence one
may take \(r^\ast=r\). If \(\phi\in C^1(J_0\times\Omega)\), the right-hand side
is continuous in \(t\), so the identity holds for every \(t\in J_0\).
\end{proof}

\section{Two-Phase Flow Kinematics}\label{sec:kinematics}
The investigation of the flow kinematics is based on the concept of \emph{fluid elements}, i.e.\ infinitesimal fluid volumes that move according to the underlying velocity field. 
Since the velocity field, given as\vspace{-0.1in}
\begin{equation}\label{DBeq:2ph-velocity1}
\DBv (t, \DBx) =
\begin{cases}
\DBv^+ (t, \DBx) & \mbox{if } \DBx \in \Omega^+ (t),\\
\DBv^- (t, \DBx) & \mbox{if } \DBx \in \Omega^- (t),
\end{cases}
\end{equation}
will, in general, not be continuous at $\iface (t)$, it is not immediately clear what happens if a fluid particle reaches the interface.
To understand the impact of the jump conditions between two fluid bulk phases, we are going to study the flow kinematics close to the interface. For this purpose, we first complement the jump conditions (\ref{E3}), (\ref{E4}) by a transmission condition for momentum transfer that describes the effect of possible (one-sided) friction between the bulk fluids and the interface. The constitutive bulk-interface momentum transfer relations need to be consistent with the second law of thermodynamics in the sense that they render the local one-sided entropy productions non-negative, i.e.
\begin{equation}\label{eq-mom-trans-entropy}
- (\vbf^\pm -\vbf^\Sigma)_{||} \cdot (\Sbf^\pm  \nbf^\pm)_{||} \geq 0,
\end{equation}
where $\nbf^\pm$ denotes the outer normal to $\Omega^\pm$ at $\iface$; see \cite{Bo-interface-mass} for a detailed derivation. 
Employing a linear closure in the co-factors yields 
\begin{equation}\label{E42}
\alpha^\pm (\vbf^\pm -\vbf^\Sigma)_{||}+ (\Sbf^\pm \cdot \nbf^\pm)_{||}=0 \quad\text{ with }\; \alpha^\pm \geq 0.
\end{equation}
Equations (\ref{E42}) are two one-sided Navier-type boundary conditions, where $\alpha^\pm$ are material-dependent parameters. When we call the counterexample below physically consistent, this refers to compatibility with incompressibility, the bulk and interfacial mass and momentum balances, Newtonian bulk stresses, and the entropy-production condition \eqref{eq-mom-trans-entropy}; no complete thermal or compositional constitutive closure is intended. Now notice that the tangential interface velocity $\vbf^\Sigma_{||}$ is not defined from a momentum balance as the interface is considered massless ($\rho^\iface \equiv 0$). 
In order to give meaning to the one-sided Navier-type boundary conditions, we combine 
(\ref{E42}) with the tangential component of (\ref{eq-mom-jump}). We focus on the case of constant $\sigma$
and let $\dot{m}_-^+$ denote the mass transfer rate from $\Omega^-$ to $\Omega^+$. 
Moreover, $-[\![\DBS\cdot\DBn_\Sigma]\!]=\DBS^+\cdot\DBn^+ + \DBS^-\cdot\DBn^-$. Hence the tangential part of \eqref{eq-mom-jump} yields
\begin{equation}\label{eq-mom-jump-tan}
\dot{m}_-^+ \, ( \DBv^+_{||} - \DBv^-_{||} )  + (\DBS^+ \cdot \DBn^+)_{||} + (\DBS^- \cdot \DBn^-)_{||} 
= 0 \quad \text{ on } \Sigma.
\end{equation}
Summation of \eqref{E42} for both phase indices $\pm$ and use of \eqref{eq-mom-jump-tan} gives
\begin{equation}\label{eq-interface-velo}
\vbf^\Sigma_{||} =\frac{\alpha^+ - \dot{m}_-^+}{\alpha^+ + \alpha^-}\, \DBv^+_{||}
+ \frac{\alpha^- + \dot{m}_-^+}{\alpha^+ + \alpha^-}\, \DBv^-_{||},
\end{equation}
where we assume that $\alpha^+ + \alpha^->0$. Note that the case $\alpha^+ = \alpha^- =0$ refers to free slip of both fluid phases against the interface, in which case the bulk and interface tangential velocities are unrelated. The velocity in \eqref{eq-interface-velo} is a convex combination whenever $-\alpha^-\le \dot m_-^+\le\alpha^+$; in particular this holds under the symmetric sufficient condition $|\dot m_-^+|\le\min\{\alpha^+,\alpha^-\}$.

Substituting \eqref{eq-interface-velo} into the plus-side relation in \eqref{E42} yields
\begin{equation}\label{E42b}
\frac{\alpha^+ (\alpha^- + \dot{m}_-^+)}{\alpha^+ + \alpha^-} (\vbf^+ -\vbf^-)_{||}
+ (\Sbf^+ \cdot \nbf^+)_{||}=0.\vspace{0.05in}
\end{equation}
In the general case with possibly non-constant surface tension, the right-hand side of \eqref{E42b} is $\alpha^+\na_\Sigma\sigma/(\alpha^+ + \alpha^-)$; the coefficient on the left remains as displayed above.
\vskip3mm
To study the interface-near kinematics, we assume the one-sided bulk velocity fields and their traces to be jointly continuous in $(t,{\bf x})$ and locally uniformly Lipschitz continuous in the spatial variable ${\bf x}$ on the closure of the respective bulk phase, as in Assumption~\ref{DBass:two-phase-kinematic-setting} below. Then the initial value problems (\ref{DBeq:kinematic-ODE1}), governing the kinematics of fluid particles, have unique local solutions for any initial value not lying on the initial interface. To study the behavior of solutions at the interface, we distinguish between four  cases of increasing complexity.\\[1ex]
{\bf 1. No-slip and no phase change.}
In this case, the velocity field $\DBv$ is continuous across the interface, hence globally continuous. The locally uniform one-sided Lipschitz bounds up to the interface, together with continuity of the traces, imply by a standard segment argument that $\DBv(t,\cdot)$ is locally Lipschitz also across the interface. Hence the initial value problems (\ref{DBeq:kinematic-ODE1}) are uniquely solvable
for all initial values. Due to $\dot m =0$, we have $\DBv^\pm \cdot \DBn_\Sigma =V_\iface$. Hence solutions starting in $\bar{\Omega}^\pm (t_0)$ at $t=t_0$ stay in  $\bar{\Omega}^\pm (t)$ for all $t$. This property is called \emph{flow invariance}.
Consequently, solutions starting in 
$\iface (t_0)$ at $t=t_0$ stay in  $\iface (t)$ for all $t$, i.e.\ the interface is always composed of the same fluid
elements. In this case, the interface is called a \emph{material interface}.
Note that, as a consequence of the above, $\bar{\Omega}^\pm (\cdot) \setminus \iface (\cdot)$ is flow invariant as well.
In particular, solutions cannot reach $\iface$ from inside $\Omega^\pm$ in finite time.\\[1ex]
{\bf 2. Slip and no phase change.} In this case, the one-sided limits of the bulk velocity fields are, in general, not 
equal, i.e.\ $\DBv^+ \neq \DBv^-$ at $\iface$. Nevertheless, as $\DBv^\pm \cdot \DBn_\Sigma =V_\iface$ due to $\dot m =0$, 
the closure of the bulk phases, $\bar{\Omega}^\pm (\cdot)$, are flow invariant as solutions cannot cross the interface.
With $\vbf^\Sigma_{||}$ given by (\ref{eq-interface-velo}), the interfacial kinematic ODE, more precisely the initial value problems
\begin{equation}\label{DBeq:IVP-surface-full}
\dot \DBx^\iface (t) = \DBv^\iface (t, \DBx^\iface (t)) \; \mbox{ for } t\in J,\quad \DBx^\iface (t_0)=\DBx^\iface_0\in \iface(t_0),
\end{equation}
are well-posed due to
existence theorems for differential equations under time-dependent constraints; cf.\ \cite{Bo2} (specialized to single-valued ODEs). If, on the other hand, both one-sided velocity limits $\DBv^+,\, \DBv^-$ are admitted to govern the motion of fluid elements inside $\iface$, the interfacial kinematics is, in general, not uniquely determined. Indeed, at an interfacial initial value where $\DBv^+ \neq \DBv^-$, a continuum of solutions is possible; see the top row in Figure~\ref{DBFig_two-phase-comoving} and cf.\ the multivalued regularization introduced below.
However, as all these solutions stay in $\iface$, this has no impact on the kinematics inside the bulk phases.
Moreover, $\iface$ is again a material interface.\\[1ex]
{\bf 3. No-slip and phase change.} 
We consider the physically realistic situation in which the transmission condition (\ref{E3}) is satisfied for some locally Lipschitz functions $\rho^\pm >0$ (but not necessarily representing mass densities). In this case, the jump condition (\ref{E3})
implies the transversality condition
\begin{equation}\label{eq-transversality}
\mbox{sgn}\big( (\DBv^+ -\DBv^\Sigma) \cdot \DBn_\Sigma \big) = \mbox{sgn}\big( (\DBv^- -\DBv^\Sigma) \cdot \DBn_\Sigma \big).
\end{equation}
Evidently, any solution which reaches $\iface$ at some $t=t_0$ with non-vanishing relative normal velocity crosses the interface and enters the opposite bulk phase with possibly different, but again non-vanishing relative normal velocity.
Consequently, a problem concerning unique solvability can only occur if a solution reaches the interface with zero relative normal velocity.
Under the full transmission condition (\ref{E3}) together with no-slip, the analysis in the proof of Theorem~1 in \cite{Bo-2PH-ODE} shows that such tangential touching of the interface does not destroy the unique solvability. 
Hence, also  in the present case, the initial value problems (\ref{DBeq:kinematic-ODE1}) admit unique strong solutions
for all initial values.\\[1ex]
{\bf 4. Slip and phase change.} 
As in the case with no-slip and phase change, assuming the jump condition (\ref{E3}) to be satisfied for some locally Lipschitz functions $\rho^\pm >0$, solutions that reach the interface with non-zero relative normal speed will cross the interface and can be uniquely continued at least locally.
Non-uniqueness is therefore only possible if a solution reaches $\iface$ in some $(t_0,\DBx_0)$ such that $\DBv^\pm \cdot \DBn_\Sigma= V_\iface$ there, i.e.\ the solution touches the (moving) interface tangentially. Here the phase index $\pm$ attains the value of the bulk phase from which the solution reaches the interface. In cases where the solution runs through both bulk domains within arbitrarily small time intervals before reaching the interface, both phase indices are admissible.

Now, somewhat surprisingly, non-uniqueness can indeed happen even in 2D and for an autonomous flow field with $C^1$ one-sided velocities inside the phases.
To see this, we let
\begin{equation}\label{eq-counter-setting}
\iface (t) \equiv \R \times \{ 0 \}, \; \Omega^-=\R \times (-\infty, 0) \mbox{ and } \Omega^+=\R \times (0,\infty).
\end{equation}
The velocity field is defined as $\DBv(\DBx)=({\rm sgn}(x_2),x_1)$, i.e.\
\begin{equation}\label{eq-counter-ex}
\DBv^+ (x_1,x_2) = ( 1 , x_1 ),
\quad \DBv^- (x_1,x_2) = ( -1 , x_1 ).
\end{equation}
%
%

The velocity fields $\DBv^\pm$ are both continuously differentiable in their respective domains.
Evidently, there is slip between the fluid phases at $\iface$.
The velocity field satisfies the mass transfer jump condition for the constant density $\rho \equiv 1$ in both phases.
Moreover, the continuity equation is satisfied for this pair of $\rho$ and $\DBv$ as the velocity field satisfies $\DBdiv \DBv =0$.

The local mass transfer flux satisfies $\dot m \neq 0$ for all $\DBx \in \iface$ with $x_1 \neq 0$.
Consequently, non-uniqueness can only occur if a solution starts at or hits the point $\DBx_0 = (0,0)$.
In fact, for initial value $\DBx_0=(0,0)$, the initial value problem (\ref{DBeq:kinematic-ODE1}) 
has infinitely many strong solutions. Here, by a strong solution, we refer to an absolutely continuous function that satisfies the differential equation for almost all $t\geq t_0$. For the specific discontinuous ODE under consideration, strong solutions will be differentiable in every $t\geq t_0$ except for, possibly, a single time instant $\tau \geq t_0$ at which only one-sided derivatives exist.
The set of all solutions depends on which values we admit for $\DBv$ on the interface.
If we let $\DBv(x_1,0)=(0,x_1)$, or more generally use set-valued interface values all of whose vectors have second component $x_1$ and which contain $(0,x_1)$, the solution set
is the union of two one-parameter families $\DBx^\pm (\cdot \,;\tau)$, defined as 
\begin{equation}\label{DBeq:one-para-solution}
\DBx^\pm (t;\tau) =
\begin{cases}
(0, 0) & \mbox{ for } t_0 \leq t \leq \tau,\\
(\pm (t-\tau), \pm \frac 1 2 (t-\tau)^2) & \mbox{ for } t >\tau,
\end{cases}
\end{equation}
where the parameter $\tau$ runs through $[t_0,\infty]$, with $\tau=\infty$ representing the stationary solution $\DBx(t)\equiv(0,0)$.

We next make precise in what sense this elementary counterexample is compatible with the physical sharp-interface conditions specified above. Fix the orientation $\DBn_\Sigma=\DBe_2$. For the field \eqref{eq-counter-ex} one has $\dot m=x_1$. Take $\rho^\pm=1$, equal constant viscosities $\eta^\pm=\eta>0$, $\alpha^\pm=\eta$, $\DBv^\Sigma=0$, and
\[
        p^+(x_1,x_2)=p_0-x_2,
        \qquad
        p^-(x_1,x_2)=p_0+x_2.
\]
Then the bulk mass and momentum equations hold with $\DBb^\pm=0$, the interfacial mass balance and both Navier laws are satisfied, the corresponding entropy productions in \eqref{eq-mom-trans-entropy} are non-negative, and the normal component of the interfacial momentum balance holds. For constant surface tension, however, its tangential component leaves precisely the term
\[
        \dot m(\DBv^+-\DBv^-)_{||}
        +(\DBS^+\DBn^+)_{||}+(\DBS^-\DBn^-)_{||}
        =2x_1\DBe_1.
\]
Thus only this tangential interfacial balance remains to be repaired.

There are at least two particularly simple fixes which leave the velocity field \eqref{eq-counter-ex} unchanged. First, allow a Marangoni force and choose
\begin{equation}\label{eq:localized-sigma}
        \sigma(x_1)=\sigma_0+x_1^2,\qquad \sigma_0>0,
        \qquad \partial_{x_1}\sigma=2x_1 .
\end{equation}
Then the missing tangential term is exactly $\nabla_\Sigma\sigma$, while all other choices above remain unchanged. Second, one may keep $\sigma$ constant and, locally on $|x_1|<\eta_0$, take
\[
        \eta^+(x_1)=\eta_0+x_1,
        \qquad
        \eta^-(x_1)=\eta_0-x_1,
        \qquad
        \alpha^+=\alpha^-=\eta_0,
        \qquad
        \DBv^\Sigma_{||}=-\frac{x_1}{\eta_0}\DBe_1 .
\]
The two Newtonian tangential tractions then add up to $-2x_1\DBe_1$, so the constant-$\sigma$ tangential momentum balance and both Navier laws hold. In fact, with constant pressure the variable-viscosity stress divergence balances the convective acceleration, so the bulk momentum equations hold with $\DBb^\pm=0$. These two elementary variants already show that the non-uniqueness mechanism is compatible with the sharp-interface balance laws.

A third variant is useful because it keeps both the viscosity and the surface tension constant. Modify the bulk velocities away from the interface according to
\begin{equation}\label{eq:localized-traces}
\widetilde{\DBv}^{\,\pm}(x_1,x_2)
=
\Big(
\pm\big(1+\frac{x_1x_2}{\eta}\big),
\;x_1\mp\frac{x_2^2}{2\eta}
\Big).
\end{equation}
They are divergence-free and have exactly the same interface traces $(\pm1,x_1)$ as \eqref{eq-counter-ex}. With $\rho^\pm=1$, $\eta^\pm=\alpha^\pm=\eta$, constant $\sigma$, $p^\pm=0$, and $\DBv^\Sigma_{||}=-(x_1/\eta)\DBe_1$, their tangential Newtonian tractions are
\begin{equation}\label{eq:localized-entropy}
        (\DBS^+\DBn^+)_{||}=-(\eta+x_1)\DBe_1,
        \qquad
        (\DBS^-\DBn^-)_{||}=(\eta-x_1)\DBe_1.
\end{equation}
Hence the tangential momentum jump and both Navier laws hold exactly; the normal jump is also satisfied on the flat interface. Finally, define the manufactured body forces by
\begin{equation}\label{eq:localized-pressure}
        p^\pm=0,
        \qquad
        \DBb^\pm
        =(\widetilde{\DBv}^{\,\pm}\!\cdot\nabla)\widetilde{\DBv}^{\,\pm}
        -\eta\,\Delta\widetilde{\DBv}^{\,\pm}.
\end{equation}
Then \eqref{eq:localized-traces} is an exact steady solution of the forced incompressible two-phase Navier--Stokes system with constant density, viscosity and surface tension, including the interfacial mass and momentum balances and the Navier slip relations. Its interface traces are the same as those of \eqref{eq-counter-ex}; in particular, the waiting-time non-uniqueness at $(0,0)$ persists, although the departing trajectories are no longer parabolas.

The preceding explicit realizations are intentionally elementary. They are posed on an unbounded domain, and already the velocity field \eqref{eq-counter-ex} grows unboundedly with $|x_1|$; in the first repair \eqref{eq:localized-sigma} the surface tension is unbounded as well. Likewise, the choice $\rho^+=\rho^-=1$ is made only for simplicity. For arbitrary prescribed constants $\rho^\pm>0$, the same mechanism can be realized by prescribing a common interfacial mass flux $\dot m$ and choosing the one-sided normal velocities according to $(\DBv^\pm-\DBv^\Sigma)\cdot\DBn_\Sigma=\dot m/\rho^\pm$, with the normal stresses adjusted accordingly.

All unboundedness can also be removed. To see this without pursuing the construction in detail, start from the constant-$\eta$, constant-$\sigma$ variant \eqref{eq:localized-traces}. Choose a bounded smooth domain and a smooth closed stationary interface containing a flat arc on which the desired local kinematics is retained, and orient $\DBn_\Sigma=\DBn^-$. With a periodic arclength coordinate $s$, prescribe smooth bounded interface traces of the form
\[
        \DBv^\pm|_\Sigma=\pm a(s)\DBtau+h(s)\DBn^-,
        \qquad
        \DBv^\Sigma_{||}=-\frac{h(s)a(s)}{\eta}\DBtau,
        \qquad
        \DBv^\Sigma\cdot\DBn_\Sigma=0.
\]
where $a=1$ and $h=s$ on the relevant flat arc and $h$ is chosen with zero mean around the closed interface. The first normal derivatives can be chosen so that the Newtonian tangential tractions satisfy
\[
        (\DBS^+\DBn^+)_{||}=-a(\eta+h)\DBtau,
        \qquad
        (\DBS^-\DBn^-)_{||}=a(\eta-h)\DBtau.
\]
These are exactly the two Navier laws for the preceding traces, and their sum cancels the tangential transfer term $2ha\DBtau$, so the interfacial momentum balance holds with constant $\sigma$. The zero-mean condition on $h$ is the flux compatibility needed for smooth divergence-free phasewise extensions; in two dimensions such extensions can be constructed by stream functions whose boundary derivatives are chosen to realize the prescribed velocity traces and tangential stresses, with the velocity cut off away from the interface. The pressure jump can then be chosen to satisfy the normal interfacial balance, the pressures extended smoothly into the bulk, and the body forces defined from the steady Navier--Stokes equations. On the resulting bounded configuration all fields and forces are bounded, while $\rho$, $\eta$ and $\sigma$ may be kept constant and the local non-uniqueness mechanism is unchanged.

For the remainder of the manuscript, including Example~\ref{DBex:co-moving-tthm-example} and Appendix~\ref{app:attainable-sets-cusp}, we nevertheless return to the elementary field \eqref{eq-counter-ex}. This is solely for technical simplicity: its trajectories and attainable sets admit particularly transparent explicit formulas.
\section{Two-phase Flow Map}\label{sec:flow-map}
We have seen above that the kinematic differential equation (\ref{DBeq:kinematic-ODE1}), which in general has a discontinuous right-hand side, might not be uniquely solvable. A well-established concept to appropriately treat such discontinuous ODEs uses the theory of differential inclusions, based on an appropriate
passage from a discontinuous to a multi-valued right-hand side.
Roughly speaking, the right-hand side is modified by filling in the jumps at points of discontinuity. 
In mathematical terms this means to replace a
given (measurable) function $\DBf:J \times \Omega \to \R^n$ which is locally bounded in space, i.e.\ for every $t\in J$ and $\DBx\in\Omega$ there is $r>0$ such that $\DBf(t,B_r(\DBx)\cap\Omega)$ is bounded, by the
set-valued map\footnote{This is often named Krasovskii regularization or Krasovskii convexification.} $\DBF:J \times \Omega \to 2^{\R^n}\setminus\{\emptyset\}$, defined as
\begin{equation}\label{DBdef:setvalued-regularization}
\DBF(t,\DBx):= \bigcap\limits_{\delta>0} \overline{{\rm conv}}\, \DBf\big(t,B_\delta(\DBx)\cap \Omega \big) \quad\text{ for } t \in J,\; \DBx \in \Omega.
\end{equation}

One then considers the differential inclusion
\begin{equation}\label{DBeq:MVIVP}
\dot \DBx \in \DBF(t,\DBx(t))  \;\mbox{ a.e.\ on } J, \;\; \DBx(t_0)=\DBx_0
\end{equation}
instead of the initial value problem for the corresponding ODE with $\DBf$. If $I\subset J$ is an interval containing $t_0$, a \emph{strong solution} on $I$
of (\ref{DBeq:MVIVP}) is understood to satisfy
\begin{equation}\label{DBeq:DI-integrated}
\DBx(t) = \DBx_0 + \int_{t_0}^t \DBw(s) \,ds \; \mbox{ for all } t\in I
\end{equation}
with $\DBw\in L^1 (I;\R^n)$ being a so-called \emph{selection} of $\DBF(\cdot,\DBx(\cdot))$, i.e.\
\begin{equation}\label{DBeq:DI-selection-property}
\DBw(t) \in  \DBF(t,\DBx(t)) \mbox{ a.e.\ on } I.
\end{equation}
For a global strong solution one takes $I=J$.
By local boundedness, for all sufficiently small $\delta>0$ the sets in the intersection in (\ref{DBdef:setvalued-regularization}) are nonempty compact convex sets and are nested as $\delta\downarrow0$; restricting the intersection to such $\delta$ does not change it. Hence $\DBF(t,\DBx)$ is nonempty, compact and convex. It also follows that the set-valued maps $\DBF (t,\cdot):\Omega\to 2^{\R^n} \setminus \{\emptyset \}$ are \emph{upper semicontinuous}
(\emph{usc} for short), meaning that $\{\DBx \in \Omega : \DBF (t,\DBx ) \cap A \neq \emptyset \}$
is closed (in $\Omega$) for every closed $A\subset \R^n$.
Indeed, if $A\subset \R^n$ is closed and there is a sequence $(\DBx_m)\subset \Omega$ with
$\DBx_m \to \DBx \in \Omega$ and there are $\DBy_m \in \DBF (t,\DBx_m ) \cap A $ for every $m\in \N$, then
$(\DBy_m)$ is bounded, hence $\DBy_{m_k} \to \DBy$ for some subsequence and some $\DBy \in A$.
Fix $\delta>0$. For all sufficiently large $k$, one has $|\DBx_{m_k}-\DBx|<\delta/2$, and hence
\[
\DBy_{m_k}\in\DBF(t,\DBx_{m_k})
\subset \overline{{\rm conv}}\,\DBf\bigl(t,B_{\delta/2}(\DBx_{m_k})\cap\Omega\bigr)
\subset \overline{{\rm conv}}\,\DBf\bigl(t,B_\delta(\DBx)\cap\Omega\bigr).
\]
Passing to the limit gives $\DBy\in\overline{{\rm conv}}\,\DBf(t,B_\delta(\DBx)\cap\Omega)$. Since $\delta>0$ was arbitrary, $\DBy\in\DBF(t,\DBx)$.
It follows that $\DBy \in \DBF (t,\DBx)\cap A$, thus $\{\DBx \in \Omega : \DBF (t,\DBx ) \cap A \neq \emptyset \}$ is closed in $\Omega$.

{
For the specific two-phase multifunction introduced below, no separate general
selection criterion is needed. Under Assumption~\ref{DBass:two-phase-kinematic-setting},
its values are nonempty, compact and convex, and the joint continuity of the one-sided
traces together with the smooth motion of the interface implies that its graph is closed
on compact subsets of $J\times\overline\Omega\times\R^n$. Together with local
boundedness, this gives upper semicontinuity in $(t,\DBx)$ and, in particular, Borel
measurability of $t\mapsto\widehat{\DBv}(t,\DBx)$ for every fixed $\DBx$.
These are the measurability and semicontinuity properties used in the existence theorem
for the differential inclusion below. We therefore do not formulate a separate general
measurability criterion here. The relevant existence theory can be found, e.g., in the
monographs \cite{Ac, MDE, DBFil}.
}

To apply the regularization (\ref{DBdef:setvalued-regularization}) literally to the two-phase velocity field $\DBv$ from (\ref{DBeq:2ph-velocity1}), one may first assign an arbitrary value on $\iface(t)$ lying in the convex hull of the two one-sided traces. The resulting Krasovskii regularization is independent of this auxiliary interface assignment and is given by\vspace{-0.15in}
\begin{equation}\label{DBeq:2ph-multivalued-velocity}
\hat{\DBv}(t,\DBx)= \begin{cases}
\{\DBv^+(t,\DBx)\} &\text{ if } \DBx \in \Omega^+(t),\\
{\rm conv}\{\DBv^+(t,\DBx), \DBv^-(t,\DBx)\} &\text{ if } \DBx \in \iface(t),\\
\{\DBv^-(t,\DBx)\} &\text{ if } \DBx \in \Omega^-(t).
\end{cases}
\end{equation}
At points of $\partial\Omega$, which do not meet $\iface(t)$ by the standing assumption that the interface lies inside $\Omega$, we define $\hat{\DBv}$ by the corresponding one-sided phase trace. The impermeability condition in Assumption~\ref{DBass:two-phase-kinematic-setting} makes this boundary definition compatible with the state constraint $\bar\Omega$.
Note that $\hat{\DBv}(t,\DBx)$ is a singleton for $\DBx$ away from $\iface (t)$, but may be multivalued for $\DBx \in \iface (t)$. The latter happens, in general, if there is slip at the interface or if there is mass transfer across the interface.
For a smooth transition across the interfacial layer, the velocity may attain intermediate values between the two bulk traces. In the sharp-interface representation, only the one-sided traces remain. The Krasovskii regularization assigns their convex hull at the interface and thereby provides a set-valued interpolation between these traces. This motivates its use as the kinematic regularization considered below.
\begin{assum}[Two-phase kinematic setting]\label{DBass:two-phase-kinematic-setting}
Let $J=(a,b)\subset\R$ and let $\Omega\subset\R^n$ be open with $C^1$-boundary and outer normal $\DBn_\Omega$. Let $\{\iface(t)\}_{t\in J}$ be a $\mathcal C^{1,2}$-family of moving closed hypersurfaces inside $\Omega$, with normal field $\DBn_\Sigma$, which decomposes $\Omega$ into $\Omega^+(t)\cup\Omega^-(t)\cup\iface(t)$. Let the one-sided velocity fields
$\DBv^\pm:\gr(\overline{\Omega^\pm})\to\R^n$ be jointly continuous and locally uniformly Lipschitz in the spatial variable on the phase closures; that is, for every compact interval $I\Subset J$ and every compact set $K\Subset\Omega$ there is $L=L(I,K)$ such that
\[
 |\DBv^\pm(t,\DBx)-\DBv^\pm(t,\DBy)|\le L|\DBx-\DBy|
\]
for all $t\in I$ and all $\DBx,\DBy\in K\cap\overline{\Omega^\pm(t)}$. We also assume the impermeability condition\footnote{If $\partial\Omega$ is piecewise $C^1$ with possible edges or corners, this is replaced by a standard subtangential condition; cf.\ \cite{MDE}.}  $\DBv\cdot\DBn_\Omega=0$ on $\partial\Omega$. 
We use $\hat{\DBv}$ from \eqref{DBeq:2ph-multivalued-velocity}. When global-in-time attainable sets are considered, we additionally impose the linear growth bound
\begin{equation}\label{DBeq:linear-growth-velocity}
\sup \{\|\DBw \|:\DBw\in\hat{\DBv}(t,\DBx)\}
\le c \,(1+\|\DBx \|)
\qquad (t\in J,\, \DBx\in\overline\Omega)
\end{equation}
with some $c>0$.
\end{assum}
In this setting, Theorem~5.2 in \cite{MDE} yields the following result.
\begin{lemma}[Two-phase kinematic differential inclusion]\label{DBlem:kinematic-DI-existence}
Under the Assumption~\ref{DBass:two-phase-kinematic-setting}, the initial value problem
\begin{equation}\label{DBeq:kinematic-DI-IVP}
\dot{\DBx}(t)\in \hat{\DBv} (t,\DBx (t)) \; \mbox{ for } t\in J, \quad \DBx(t_0)=\DBx_0
\end{equation}
for the kinematic differential inclusion has a local strong solution for every $t_0\in J$ and $\DBx_0 \in \bar{\Omega}$, forward and backward in time.
If, in addition, the linear growth bound \eqref{DBeq:linear-growth-velocity} holds, then every local strong solution can be extended to a strong solution on all of $J$.
\end{lemma}
Indeed, the contingent-cone condition required in \cite[Theorem~5.2]{MDE} is automatic in the interior of $\bar\Omega$. At a point of the $C^1$-boundary $\partial\Omega$, every value of $\hat{\DBv}$ is the corresponding one-sided phase trace and satisfies $\DBv\cdot\DBn_\Omega=0$, hence belongs to the contingent cone of $\bar\Omega$. This gives forward local existence. Applying the same theorem to the time-reversed multifunction $-\hat{\DBv}(-t,\DBx)$ gives backward local existence, and the linear growth bound yields continuation on all of $J$.
The preceding section has shown that we cannot expect unique solvability of the two-phase kinematic differential inclusion in case of slip at the interface. However, due to Lemma~\ref{DBlem:kinematic-DI-existence} there is a well-defined--albeit multivalued--flow map associated to (\ref{DBeq:kinematic-DI-IVP}), which also gives rise to well-defined co-moving volumes.

\begin{definition}[Two-phase flow map and co-moving sets]\label{DBdef:two-phase-flow-map}
Assume the two-phase kinematic setting of Assumption~\ref{DBass:two-phase-kinematic-setting} and the linear growth bound \eqref{DBeq:linear-growth-velocity}.
We then define the \emph{two-phase flow map} $\hat{\Phi}_{t_0}^t :\bar{\Omega} \to 2^{\bar{\Omega}}\setminus \{ \emptyset \}$, associated to the two-phase velocity field, via
\begin{equation}\label{DBeq:two-phase-flow-map}
\hat{\Phi}_{t_0}^t (\DBx_0) =\{ \DBx (t;t_0,\DBx_0) : \DBx (\cdot;t_0,\DBx_0) \mbox{ is a global strong solution on $J$ of } (\ref{DBeq:kinematic-DI-IVP}) \}.
\end{equation}
The set $\hat{\Phi}_{t_0}^t (\DBx_0)$ is also called \emph{attainable} or \emph{reachable set} of the
underlying initial value problem.
If $t_0$ is understood from the context, we also use the shorter notation $\hat{\Phi}^t (\DBx_0)$.
For a given $t_0\in J$ and $G_0\subset \bar{\Omega}$, we define the \emph{co-moving set} $G(t)$, emanating from the set $G_0$  at time $t_0$, as\vspace{-0.05in}
\begin{equation}\label{DBeq:two-phase-co-moving-set}
G(t) = \hat{\Phi}_{t_0}^t (G_0) \mbox{ for } t\in J.
\end{equation}
\end{definition}
\noindent
Both objects are well-defined due to Lemma~\ref{DBlem:kinematic-DI-existence}, also globally as we impose
the growth condition (\ref{DBeq:linear-growth-velocity}).
Note that the flow map is also defined for $t<t_0$, but $\hat{\Phi}_{t}^{t_0}$ is not the inverse of $\hat{\Phi}_{t_0}^t$. This can happen because of the non-uniqueness of solutions: in the simplest (artificial) case of $\hat{\DBv} (t,\DBx)\equiv \bar{B}_1 (0)$, say, it holds that $\hat{\Phi}_{t_0}^t (\DBx_0)=\bar{B}_{|t-t_0|} (\DBx_0)$ for every $t\in \R$. Then, $\hat{\Phi}_{t}^{t_0} (\hat{\Phi}_{t_0}^t (\DBx_0))=\bar{B}_{2|t-t_0|} (\DBx_0)$. What holds true instead is
$\DBx_0 \in \hat{\Phi}_{t}^{t_0} (\hat{\Phi}_{t_0}^t (\DBx_0))$.
We collect this and further useful properties of the two-phase flow map, which all follow from the theory of differential inclusions; cf.\ Chapter~7 in \cite{MDE}.
Recall that the Hausdorff metric $d_H$, defined for (closed) bounded subsets of a metric space (in our case the Euclidean space $\R^n$),
is given by\vspace{-0.05in}
\begin{equation}\label{eq_Hausdorff-metric}
d_H (A,B) = \max \{ \sup_{\DBx \in A} d (\DBx ,B), \sup_{\DBx \in B} d (\DBx ,A)\}
\end{equation}
for compact $A,B\subset \R^n$ with $d(\DBx,M)$ denoting the distance from the point $\DBx$ to a set $M$.
\begin{proposition}\label{DBprop:two-phase-flow-map-properties}
Impose the Assumption~\ref{DBass:two-phase-kinematic-setting}, including the linear growth bound \eqref{DBeq:linear-growth-velocity}. Then $\hat{\Phi}_s^t :\bar{\Omega} \to 2^{\bar{\Omega}}\setminus \{ \emptyset \}$,
the associated two-phase flow map, has the following properties:
\begin{enumerate}
  \item[(a)]
  $\hat{\Phi}_s^t (\DBx) \neq \emptyset$ is compact for all $s,t\in J$ and $\DBx\in \bar{\Omega}$;\vspace{2pt}
  \item[(b)]
  $\hat{\Phi}_s^s (\DBx) = \{ \DBx \}$ for all $s\in J$ and $\DBx\in \bar{\Omega}$;\vspace{2pt}
  \item[(c)]
  $\hat{\Phi}_s^t (\DBx) = \hat{\Phi}_\tau^t \big( \hat{\Phi}_s^\tau (\DBx) \big)$ for all $s,\tau,t\in J$ with $s\leq \tau \leq t$ and all $\DBx\in \bar{\Omega}$;\vspace{2pt}
  \item[(d)]
  $\DBx \in \hat{\Phi}_{t}^{s} (\hat{\Phi}_{s}^t (\DBx))$ for all $s,t\in J$ and $\DBx\in \bar{\Omega}$;\vspace{2pt}
  \item[(e)]
  $t\to \hat{\Phi}_s^t (\DBx)$ is locally Lipschitz with respect to $d_H$ for every $s\in J$, $\DBx\in \bar{\Omega}$;\vspace{2pt}
  \item[(f)]
  $(s,\DBx)\to \hat{\Phi}_s^t (\DBx)$ is usc for every $t\in J$.\vspace{2pt}
\end{enumerate}
As a consequence, as a function of $t$, the co-moving volume $G(t)$, emanating from some given compact set $G_0 \subset \overline{\Omega}$, is locally Lipschitz continuous with respect to the Hausdorff metric $d_H$.
Indeed, for every compact $G_0 \subset \overline{\Omega}$ and every compact interval $I \subset J$, the linear growth bound \eqref{DBeq:linear-growth-velocity} and Gronwall's lemma yield a uniform bound on $|\dot x(t)|$ for all strong solutions with $x(t_0) \in G_0$, and hence
\[
d_H(G(t),G(s)) \le C_I |t-s| \qquad \text{for } s,t \in I.
\]
\end{proposition}
Let us also record the corresponding local connectedness statement. Fix $s\in J$ and a compact set $K\Subset\Omega$.
Choose a compact neighborhood $K_1$ with $K\subset K_1^\circ\Subset\Omega$. On a sufficiently short time interval about $s$, the multifunction $\hat{\DBv}$ is uniformly bounded on $K_1$. After shortening this interval once more, the speed bound implies that every solution branch starting in $K$ remains in $K_1$. Choose a spatial cut-off which is one on $K_1$ and has compact support in $\Omega$. Multiplying $\hat{\DBv}$ by this cut-off and extending by zero outside $\Omega$ gives an upper semicontinuous multifunction on $\R^n$ with compact convex values. By the same speed estimate, all branches of the extended inclusion starting in $K$ remain in $K_1$ on the chosen interval, so the constrained and extended solution sets agree there. The solution funnel of the resulting unconstrained differential inclusion in $\R^n$ is connected; cf.\ \cite[Cor.~7.2]{MDE}. Hence, for such $s,t$ and $\DBx\in K$, the attainable sets $\hat{\Phi}_s^t(\DBx)$ are compact and connected.

Furthermore, Proposition~\ref{DBprop:two-phase-flow-map-properties} also implies that $\hat\Phi_s^t(K)$ is compact whenever $K\subset\overline\Omega$ is compact. This yields Borel measurability of the sets $\hat\Phi_s^t(G_0)$ for compact $G_0$. 
If one starts from the corresponding open domain $G_0^\circ$ with compact closure, then $G_0^\circ$ may be exhausted by compact sets $K_j\Subset G_0^\circ$, and\vspace{-0.15in}
\[
        \hat\Phi_s^t(G_0^\circ)=\bigcup_{j=1}^\infty \hat\Phi_s^t(K_j),\vspace{-0.05in}
\]
so that Borel measurability follows again.
\vskip1mm
\begin{figure}
\includegraphics[width=14cm]{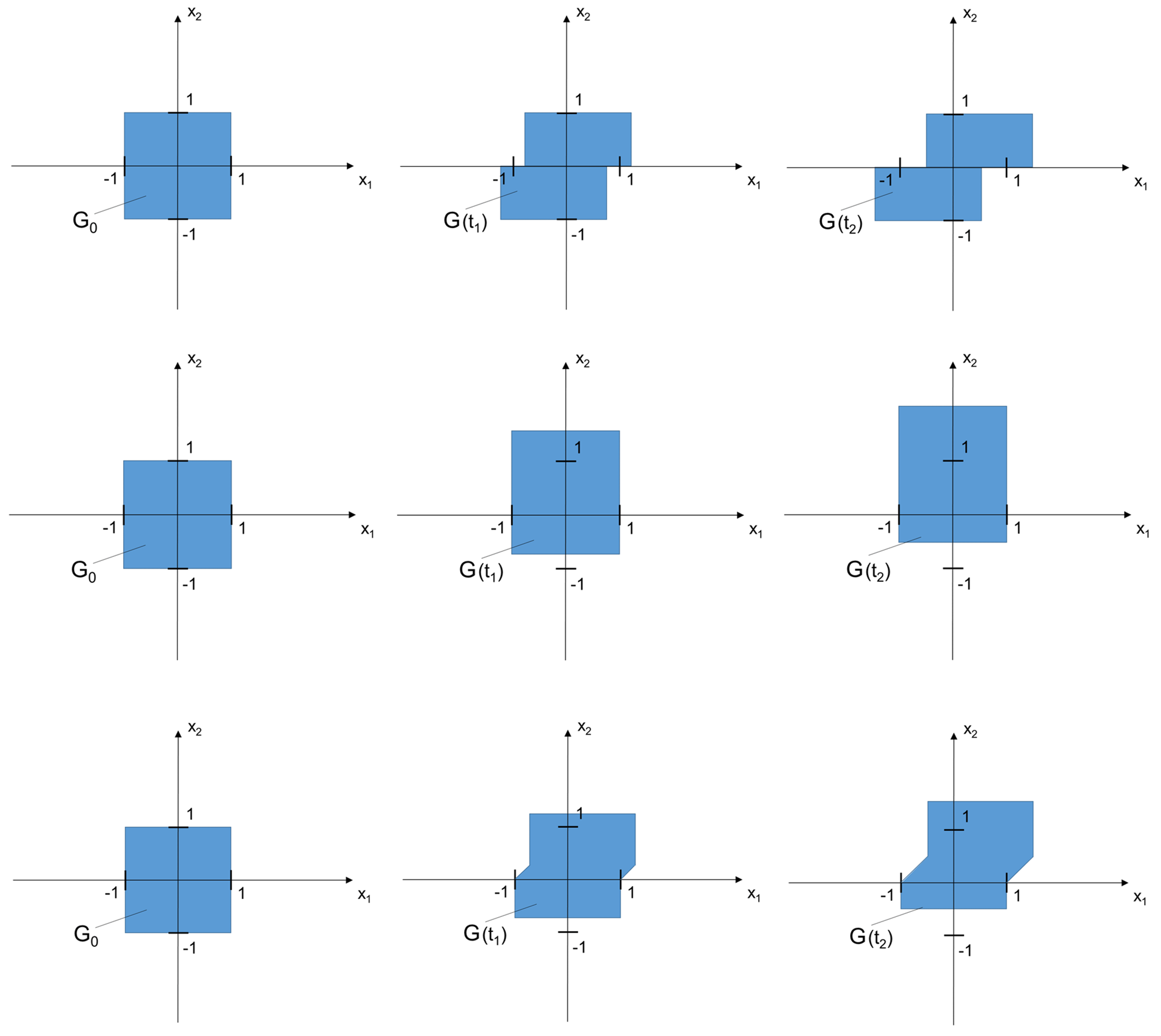}\vspace{-0.1in}
\caption{Simple two-phase co-moving volumes. Time proceeds from left to right. Top row:
zero normal velocity with interfacial slip; middle row: phase change without slip; bottom row: phase change with slip.}\label{DBFig_two-phase-comoving}
\end{figure}
\noindent
We are now aiming for a transport theorem for two-phase co-moving volumes, a two-phase RTT.
To see what can be expected in the general case, i.e.\ without unique solvability and, hence, for multivalued flow maps, let us start with some simple, but nevertheless prototypical examples.
\begin{example}\label{DBex:co-moving-tthm-example}
For simpler visualisation, we consider the two-dimensional case, where we let $\Omega=\R^2$
and $\iface \equiv \R \times \{0\}$.
In all cases, we let $\Omega^\pm \equiv \{\DBx \in \R^2: x_2 \gtrless 0\}$.
For the velocity fields, we first consider three simple cases:\\[1ex]
\indent
Case a):  $\DBv^- \equiv (-1,0)^{\sf T}$ and  $\;\DBv^+ \equiv (1,0)^{\sf T}$;\\[1ex]
\indent
Case b):  $\DBv^- \equiv (0,1)^{\sf T}\;\;\hspace{1pt}$ and  $\;\DBv^+ \equiv (0,2)^{\sf T}$;\\[1ex]
\indent
Case c):  $\DBv^- \equiv (0,1)^{\sf T}\;\;\,$ and  $\;\DBv^+ \equiv (1,1)^{\sf T}$.\\[1ex]
\noindent
In case a), there is no phase change, but slip at the interface; case b) is with phase change, but without slip; case c) is with phase change and also with slip at the interface.
We let $t_0=0$ and choose $G_0 = [-1,1]\times [-1,1]$ for easy calculation of the co-moving sets; similar pictures can be obtained for smooth initial sets such as the unit disc.
Figure~\ref{DBFig_two-phase-comoving} shows the corresponding co-moving sets at time instances $t=0$, $t_1=0.25$, and $t_2=0.5$.
Cases a) and c) illustrate that slip at the interface can create new edges at the boundary of the co-moving volume. In case a), even new boundary segments are formed, caused by the multivaluedness of
the velocities at the interface. Case b), in the absence of slip but with phase change, shows that
the jump in the normal velocity leads to a gain or loss of volume at the interface.
\begin{figure}
\includegraphics[width=5.8in]{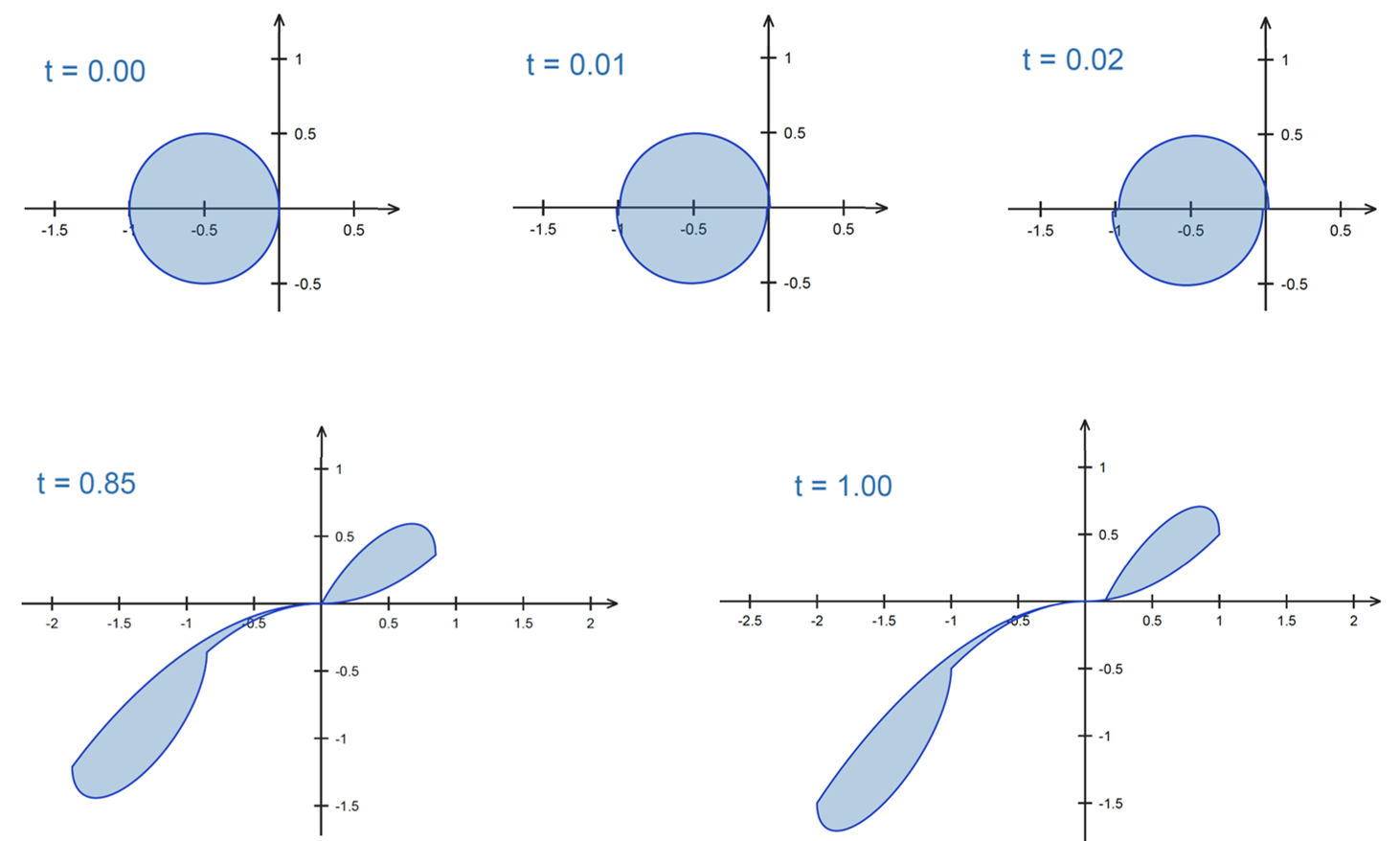}
\caption{Co-moving volumes associated with the velocity field from (\ref{eq-counter-ex}). The initial set is the disk $G(0)=G_0=\{(a,b)\in\R^2:(a+\tfrac12)^2+b^2\le\tfrac14\}$. 
The top row shows the kinematically transported set at the early times $t=0$, $t=0.01$, and $t=0.02$. The formation of four vertices is a signature of interfacial slippage. The bottom row shows the development of a cusp at $t \approx 0.85$; see Appendix~\ref{app:attainable-sets-cusp} for a proof. Hence the co-moving set ceases to be a Lipschitz domain after finite time. Subsequently, the attainable set remains connected only through a one-dimensional arc, while its interior has two connected components.}\label{DBFig_two-phase-comoving-nonunique}
\end{figure}
\indent
Case d): A somewhat more involved case is associated with the velocity field as defined above in equation (\ref{eq-counter-ex}).
Due to the non-uniqueness of the corresponding initial value problems for $\DBx_0=(0,0)$, the initial set $G_0:=\{(a,b)\in\R^2:(a+\tfrac12)^2+b^2\le\tfrac14\}$ evolves in such a manner that (i) the boundary $\partial G(t)$ of the co-moving sets $G(t) = \hat{\Phi}_{0}^t (G_0)$ instantaneously loses the $C^2$-regularity due to the interfacial slippage; (ii) from the interfacial boundary point $(0,0)$ an arc of admissible points emanates, forming an additional part of the boundary $\partial G(t)$ for $t >0$; (iii) the reduced Lipschitz regularity of $\partial G(t)$
drops further at $t=t_* \approx 0.85$ when a cusp is formed; (iv) the interior of $G(t)$ undergoes a topological change at $t=t_*$, where it changes from a connected set into a union of two connected components.
Figure~\ref{DBFig_two-phase-comoving-nonunique} displays the sets $G(t)$ for several time instances at which these changes occur;
see also Figure ~\ref{DBFig_lower-part-cusp}.
More details and concrete calculations are provided in Appendix~\ref{app:attainable-sets-cusp}.
Note that this example requires the generalisation of co-moving sets using set-valued theory as described above.

For case d), the discussion in Section~\ref{sec:kinematics} shows both that the elementary field \eqref{eq-counter-ex} admits local realizations satisfying these physical conditions and that a bounded localization retaining the local kinematics of the closely related field \eqref{eq:localized-traces}, with the same interface traces, can be realized as a forced incompressible Navier--Stokes flow with constant viscosity and surface tension. Thus the non-uniqueness and the loss of regularity are not artifacts of violating these basic physical conditions.
\end{example}
\section{Two-Phase Transport Theorem for Co-Moving Volumes}\label{sec:two-phase-RTT}
Especially case d) of Example~\ref{DBex:co-moving-tthm-example} indicates that a transport theorem for co-moving volumes in two-phase flows must be formulated with some care: even starting from a smooth control volume, the boundary of the transported set can lose regularity immediately. For the continuum-mechanical applications intended here, it is enough to evaluate the transport identity at the initial time $t_0$, for an initially regular control volume $G_0$. 
We assume below that \(G_0\) is the closure of a bounded connected
\(C^2\)-domain whose boundary cuts the interface transversally; this is
precisely the geometric condition used to control the interfacial strip in the
proof. The theorem is stated for \(n=3\), the physical case. An analogous higher-dimensional
formulation would require the corresponding dimensional modifications in the
geometric estimates.

The primary result is stated in boundary-integral form. This version only involves $\partial_t\phi$, the one-sided traces of $\phi$ and the normal velocities of the relevant boundary pieces. In particular, it does not require spatial differentiability of $\phi$ or a divergence of the one-sided velocity fields. The divergence form is then obtained as a separate corollary, under the additional assumptions needed to apply the phasewise divergence theorem.

\begin{theorem}[Boundary form of the two-phase transport theorem]\label{DBthm:2ph-transport-theorem-comovingCV}
Assume the two-phase kinematic setting of Assumption~\ref{DBass:two-phase-kinematic-setting}, including the linear growth bound \eqref{DBeq:linear-growth-velocity}, with $n=3$.
Let $t_0\in J$ and let  \(G_0\subset\Omega\) be the closure of a bounded connected \(C^2\)-domain,
with \(\DBn\) denoting its outer unit normal.
We assume that $\partial G_0$ cuts the initial interface transversally, i.e.\ $\DBn$ and $\DBn_\Sigma$ are not colinear on $\partial G_0\cap\iface_0$, where $\iface_0=\iface(t_0)$.
Let $G(t)=\hat\Phi_{t_0}^t(G_0)$ be the co-moving volume emanating from $G_0$ at $t_0$.
Let $\phi^\pm\in C(\gr(\overline{\Omega^\pm}))$ be bounded, and let $\phi$ denote the corresponding phasewise field on $J\times\Omega\setminus\gr(\iface)$; its values on the interface are immaterial. Assume that the phasewise time derivatives $\partial_t\phi^\pm$ exist in $\gr(\Omega^\pm)$ and extend continuously to $\gr(\overline{\Omega^\pm})$.
Then the two-sided derivative exists and
\begin{equation}\label{DBeq:boundary-transport-theorem-comoving}
\Big[\frac{d}{dt}\int_{G(t)}\phi\,d\DBx\Big]_{|t=t_0}
=\int_{G_0\setminus\iface_0}\partial_t\phi\,d\DBx
 +\int_{\partial G_0\setminus\iface_0}\phi\,\DBv\cdot\DBn\,dS
 -\int_{G_0\cap\iface_0}[\![\phi]\!]\,V_\Sigma\,dS .
\end{equation}
Here $V_\Sigma$ is the speed of normal displacement of the interface. All integrands on the right-hand side are evaluated at time $t_0$.
\end{theorem}

%
\noindent
Before proving the theorem, let us explain why one cannot simply apply the single-phase RTT to the two bulk parts of $G(t)$, separately.
A closer look at the example of case d) above shows that this strategy is not admissible within
the class of Lipschitz domains: there, the part of the co-moving set lying in the bulk phase $\{x_2 < 0\}$
instantaneously develops a cusp. This is illustrated in Figure ~\ref{DBFig_lower-part-cusp}, a proof is contained in Theorem~\ref{thm:geometry}.
\begin{figure}
\includegraphics[width=11cm]{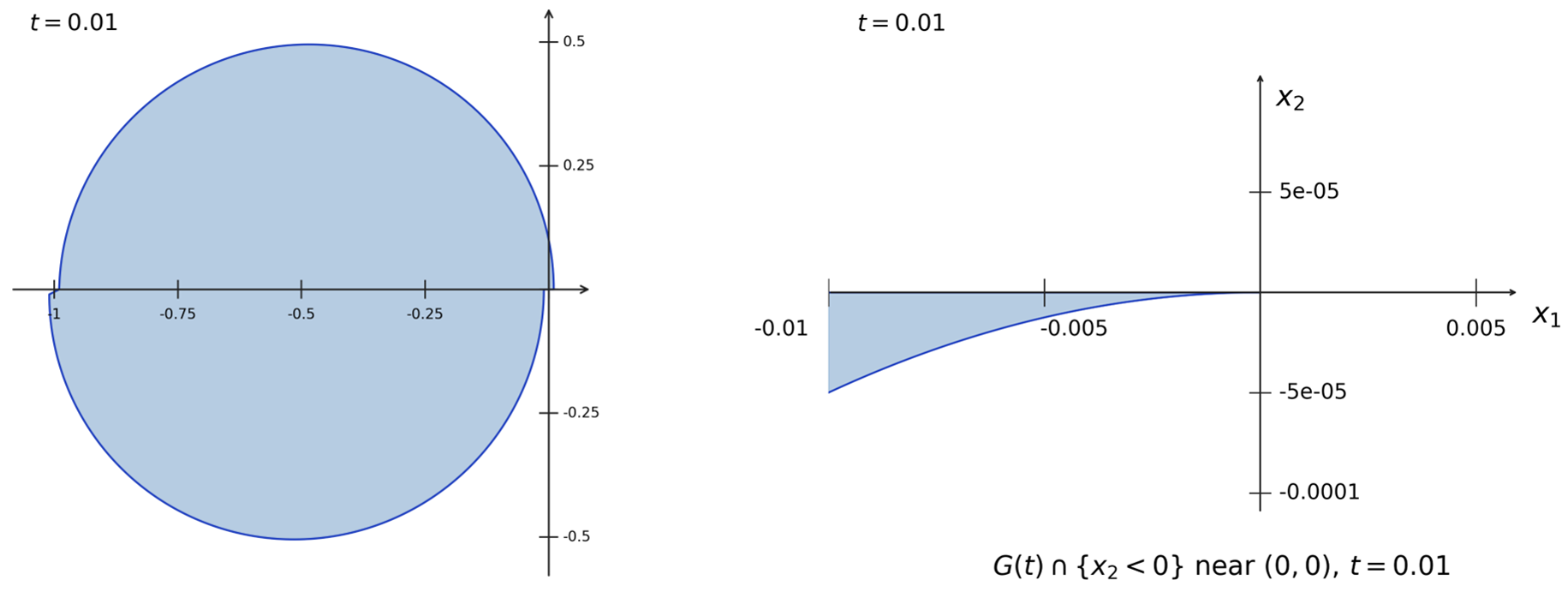}\vspace{-0.1in}
\caption{
Snapshot at time \(t=0.01\) of the co-moving set (left), and a magnified view of
the lower-phase part near the origin, showing the cusp of
\(G(t)\cap\{x_2<0\}\).}\label{DBFig_lower-part-cusp}
\end{figure}

Before entering the proof of Theorem~\ref{DBthm:2ph-transport-theorem-comovingCV}, let us briefly explain the underlying geometric idea. The only region where the boundary of the co-moving set may be difficult to control is an interfacial strip: points ending very close to $\iface(t)$ may have reached the interface and may therefore have been transported by the multivalued part of the velocity field. Away from this strip, however, the motion is governed by the single-valued one-sided velocity fields, and the usual boundary form of the transport theorem can be applied. For intuition, one may imagine first straightening and freezing the interface by a local change of variables, so that $\iface(t)$ becomes the fixed plane $\{x_3=0\}$.
Also let $t_0 =0$. Then, choosing $M$ larger than the local velocity bound, no trajectory which has touched the interface can reach the region $|x_3|>Mt$ at time $t$. Hence one cuts away the strip $|x_3|\le Mt$. The removed parts on the positive and negative signed-distance sides are then replaced by extrusions from the artificial cut surfaces back to the interface; cf.\ Figure~\ref{DBFig_cusp-removal} for an illustration.
These modified sets have Lipschitz space-time boundaries and are close to the true co-moving set up to an error of order $t^2$. Therefore this modification does not affect the derivative at the initial time, while it permits a direct (phasewise) application of the boundary-integral Reynolds transport theorem for the single-phase case. In the proof below the same construction is carried out intrinsically, using signed distance to the moving interface instead of actually performing such a flattening transformation.
\begin{proof}
We first prove equation (\ref{DBeq:boundary-transport-theorem-comoving}) for the right derivative. Both sides of
\eqref{DBeq:boundary-transport-theorem-comoving} are invariant under replacing
$\DBn_\Sigma$ by $-\DBn_\Sigma$: then both $V_\Sigma$ and the jump bracket
$[\![\phi]\!]$ change sign. The orientation may therefore be chosen so that
$\DBn_\Sigma$ points from $\Omega^-(t)$ into $\Omega^+(t)$. After translating time, we may assume $t_0=0$.

\vskip2mm
\emph{Preparatory reductions.}
Put
\[
        \iface_0:=\iface(0),\qquad \Omega_0^\pm:=\Omega^\pm(0),
        \qquad G(s):=\hat\Phi_0^s(G_0),
        \qquad \Gamma:=\partial G_0\cap\iface_0.\vspace{-0.05in}
\]
Let\vspace{-0.05in}
\[
        D_0:=G_0\cap\iface_0.
\]
Since \(\partial G_0\) is compact and \(C^2\), we may choose a tubular
neighborhood \(U_q\) of \(\partial G_0\) on which the signed distance to
\(\partial G_0\) is \(C^2\). We write this signed distance as
\[
        q:U_q\to\R,
\]
with the sign convention \(q<0\) in \(U_q\cap G_0^\circ\) and
\(q>0\) in \(U_q\setminus G_0\). Thus
\[
        q=0 \quad\text{on }\partial G_0,\qquad
        |\nabla q|=1 \quad\text{in }U_q,\qquad
        \nabla q=\DBn \quad\text{on }\partial G_0 .
\]
By the transversality assumption, \(\DBn\) and \(\DBn_\Sigma\) are not
colinear on \(\Gamma=\partial G_0\cap\iface_0\). Equivalently,
\[
        \nabla_{\iface_0}q\ne0
        \qquad\text{on }\Gamma .
\]
Hence \(0\) is a regular value of \(q|_{\iface_0}\) near \(\Gamma\), and
\(\Gamma\) is a compact \(C^2\) curve in \(\iface_0\), with only finitely many
connected components. Moreover,
\[
        \partial_{\iface_0}D_0=\Gamma .
\]
Indeed, every point of \(\iface_0\setminus\Gamma\) lies either in
\(G_0^\circ\cap\iface_0\), hence is a relative interior point of \(D_0\) in
\(\iface_0\), or in \(\iface_0\setminus G_0\), hence is a relative exterior
point.

We decompose \(D_0\) into its connected components in the relative topology of
\(\iface_0\). The components whose relative boundary is nonempty meet
\(\Gamma\), and there are only finitely many of them because \(\Gamma\) has only
finitely many connected components. The remaining components have empty relative
boundary in \(\iface_0\); they are therefore connected components of
\(\iface_0\) contained in \(G_0^\circ\). 
Since \(\iface_0\) is a closed embedded hypersurface, its connected components are
locally finite in \(\Omega\), i.e.\ every compact subset of \(\Omega\) intersects only
finitely many of them. As \(G_0\) is compact, only finitely many components of
\(\iface_0\) can be contained in \(G_0^\circ\).
Consequently, \(D_0\) has only finitely many connected components.

{
Let $D_{\mathrm{int}}$ be the union of the connected components of $D_0$ that
are contained in $G_0^\circ$, and write these components as
$C_1,\ldots,C_N$. By compactness, local finiteness of the components of
$\iface_0$, and their positive distance from $\partial G_0$, one may choose
pairwise disjoint bounded open sets $V_j$ with smooth boundary such that
\[
 C_j\subset V_j\Subset G_0^\circ,
 \qquad \overline V_j\cap\iface_0=C_j,
 \qquad \partial V_j\cap\iface(s)=\emptyset
 \quad (|s|<T)
\]
after decreasing $T>0$. Choose open sets $W_j$ with
$C_j\subset W_j\Subset V_j$ and functions $\chi_j\in C_c^\infty(V_j)$
such that $\chi_j=1$ on $W_j$. After decreasing $T$, the component of
$\iface(s)$ emanating from $C_j$ remains in $W_j$. Set
$\chi_{\mathrm{int}}:=\sum_{j=1}^N\chi_j$. Shrinking $T$ once more, the
positive distance of $\operatorname{supp}\chi_{\mathrm{int}}$ from
$G_0^\complement$, the local speed bound, and backward solvability of the
differential inclusion imply
\[
        \operatorname{supp}\chi_{\mathrm{int}}\subset G(s)
        \qquad (|s|<T).
\]
Indeed, a backward solution through a point of this support moves by less than
its distance from $G_0^\complement$ and therefore has its time-zero value in
$G_0$. Consequently,
\[
 \int_{G(s)}\chi_{\mathrm{int}}\phi(s,\DBx)\,d\DBx
 =\int_{\bigcup_j V_j}\chi_{\mathrm{int}}\phi(s,\DBx)\,d\DBx
 \qquad (|s|<T).
\]
Applying the phasewise fixed-control-volume formula
\eqref{DBeq:RTT_fixedCV}, equivalently Theorem~\ref{thm:single-phase-boundary-rtt}
on the two smooth phasewise space-time domains, to
$\chi_{\mathrm{int}}\phi$ gives at $s=0$
\[
 \frac{d}{ds}\bigg|_{s=0}
 \int_{G(s)}\chi_{\mathrm{int}}\phi\,d\DBx
 =\int_{G_0\setminus\iface_0}
   \chi_{\mathrm{int}}\partial_t\phi\,d\DBx
  -\int_{D_{\mathrm{int}}}[\![\phi]\!]V_\Sigma\,dS,
\]
because $\chi_{\mathrm{int}}=1$ on $D_{\mathrm{int}}$ and its support
meets no other component of $\iface_0$. The remaining integral has integrand
$(1-\chi_{\mathrm{int}})\phi$, which vanishes in a fixed neighborhood of
all internal components. Since the subsequent construction and all its
estimates are local and linear in the integrand, those components and the
corresponding cut-and-fill pieces may be omitted when the argument is applied
to $(1-\chi_{\mathrm{int}})\phi$. Adding the two localized identities then
recovers the full volume and interface terms. Thus it is enough in the
remaining argument to assume that $D_0$ has no connected component contained
in $G_0^\circ$. If no component with nonempty relative boundary is present,
the preceding localization, together with the ordinary single-phase RTT away
from the interface, already proves the formula. If $G_0$ has positive distance
from the interface, the assertion is only the single-phase RTT for small
times. Hence, after this reduction, we may and do assume that
}
\[
        \Gamma\neq\emptyset,\qquad \partial_{\Sigma_0}D_0=\Gamma,
\]
and \(D_0=G_0\cap\Sigma_0\) has no internal components.

\begin{figure}
\includegraphics[width=13cm]{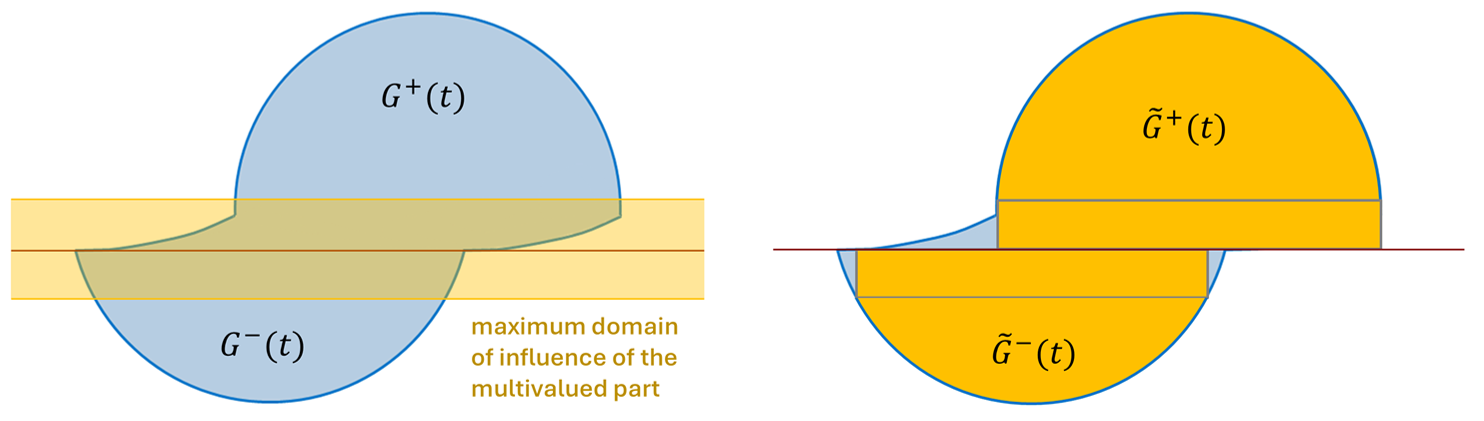}\vspace{-0.1in}
\caption{
Illustration of the cusp removal and refill by extrusion strategy.}\label{DBFig_cusp-removal}
\end{figure}

\vskip2mm
\emph{Standing local notation and constants.}
\vskip1mm

Since \(G_0\subset\Omega\) is compact, the local boundedness of the differential
inclusion and the linear growth bound allow, after decreasing \(T>0\), the
choice of a compact set \(K\subset\Omega\) with the following properties. All
solution branches which occur below, forward or backward between times \(0\)
and \(s\) with \(|s|<T\), remain in the interior of \(K\). Moreover, \(K\)
contains, with positive distance from its boundary, the compact set swept out
from \(D_0\) by the geometric surface flow introduced below, for \(|s|<T\),
together with a fixed spatial neighborhood of this swept set; after further
decreasing \(T\), this neighborhood contains all normal segments used in the
cut-and-fill construction below.

On \(K\), let \(d(t,\cdot)\) denote the signed distance to \(\iface(t)\), positive
in \(\Omega^+(t)\) with the above orientation.
Decreasing \(T\) further, if necessary, there exists \(r_*>0\) such that the
signed-distance coordinates are valid on the fixed set
\[
        \mathcal N_{K,r_*}
        :=\{(t,\DBx)\in(-T,T)\times K:\ |d(t,\DBx)|<r_*\},
\]
which is a tubular neighborhood of
\(\gr(\iface)\cap((-T,T)\times K)\). We also write
\[
        \mathcal N_{K,r_*}(s)
        :=\{\DBx\in K:(s,\DBx)\in\mathcal N_{K,r_*}\}
        =\{\DBx\in K:|d(s,\DBx)|<r_*\}
\]
for its time slices. For
every point \((t,\DBx)\in\mathcal N_{K,r_*}\), $\DBx$ has a unique representation
\[
        \DBx=\mathbf p+r\DBn_\Sigma(t,\mathbf p),\qquad
        \mathbf p\in\iface(t),\quad |r|<r_* ,
\]
where \(r=d(t,\DBx)\) and \(\mathbf p=:\pi_t(\DBx)\) is the nearest point
projection. 
The corresponding signed-distance coordinates are
\[
        (\mathbf p,r)=\bigl(\pi_t(\DBx),d(t,\DBx)\bigr).
\]
On the set \(\mathcal N_{K,r_*}\), the signed-distance function
\(d\) and the projection \(\pi_t\) are continuously differentiable. On the
interface, we have
\[
        \nabla_\DBx d(t,\mathbf p)=\DBn_\Sigma(t,\mathbf p),
        \qquad
        \partial_t d(t,\mathbf p)=-V_\Sigma(t,\mathbf p).
\]
The purely geometric normal velocity of the interface will be denoted by
\[
        \DBw_\Sigma:=V_\Sigma\DBn_\Sigma .
\]
By the regularity of a \(\mathcal C^{1,2}\)-family of moving hypersurfaces and the compactness of $D_0$, after decreasing $T$ if necessary, the
constrained ODE associated with \(\DBw_\Sigma\) generates, for \(|s|<T\), a
\(C^1\) surface flow $\Psi_s$ on a neighborhood of $D_0$ in $\iface_0$, with values in $\iface(s)$; cf.\ \cite{Bo-2PH-ODE}. Set
\[
        D_s:=\Psi_s(D_0).
\]
By the reduction above, \(D_0\) is a finite union of compact \(C^2\) surface
domains\footnote{Here and below, a compact \(C^2\) surface domain in a \(C^2\)
surface \(S\) means a compact regular closed subset of \(S\), possibly with
finitely many connected components, whose relative interior is a \(C^2\) domain
in \(S\), equivalently in local surface charts, and whose relative boundary is
a \(C^2\) curve.} in \(\iface_0\), with relative boundary \(\Gamma\). Hence
\(D_s\) is a finite union of compact \(C^1\) surface domains in \(\iface(s)\),
with relative boundary \(\Psi_s(\Gamma)\). For a set \(D\subset\iface(s)\) and
\(\rho>0\), the notation
\[
        D^{+\rho}:=\{\mathbf p\in\iface(s):
        \operatorname{dist}_{\iface(s)}(\mathbf p,D)<\rho\},
        \quad
        D^{-\rho}:=\{\mathbf p\in D:
        \operatorname{dist}_{\iface(s)}(\mathbf p,\iface(s)\setminus D)\ge\rho\},
\]
will be used, where \(\operatorname{dist}_{\iface(s)}\) denotes the intrinsic,
i.e.\ geodesic, distance on \(\iface(s)\).

All constants below are local to \((-T,T)\times K\). The letter \(C\)
denotes a generic positive constant whose value may change from line to line.
Such constants may depend on the compact set \(K\), on the time interval
\((-T,T)\), on \(\mathcal N_{K,r_*}\) and finite coordinate atlases, on the
local geometry of \(G_0\) and of the moving interface, and on the local bounds
of the relevant velocity fields. They are independent of the small parameters
\(s,r,\rho,\varepsilon,h\) and of the particular admissible solution branch.

Constants carrying labels, subscripts, or descriptive names, such as
\(C_0\), \(C_{\rm tub}\), \(C_*\), \(C_R\), \(\rho_0\), \(T_0\), \(m\), and
\(M\), are not generic once fixed. The same applies to the strip constant
\(\omega\) introduced in Claim~2 below. Within the local argument in which
such a constant is introduced, it may be enlarged finitely many times. Once a
later constant is chosen depending on it, the earlier constant is kept fixed.
The strip constant \(\omega\) appears in \(D_s^{-\omega s}\) and
\(D_s^{+\omega s}\). If \(\omega\) is enlarged later, these sets are always
understood with the enlarged value. This preserves all previously established
inclusions, because increasing \(\omega\) shrinks \(D_s^{-\omega s}\) and
enlarges \(D_s^{+\omega s}\).
Since all solution branches considered below remain in \(K\), and since the
one-sided velocity fields are bounded on the compact phase closures there, we
can choose constants \(0<m<M\) such that
\begin{equation}\label{DBeq:cutfill-M-boundary-new}
        |\DBw|\le m
        \quad\text{for all }(t,\DBx)\in(-T,T)\times K,\quad
        \DBw\in\widehat{\DBv}(t,\DBx),
\end{equation}
and
\begin{equation}\label{eq:speed-bound}
        |\partial_t d(t,\DBx)+\nabla_\DBx d(t,\DBx)\cdot \DBw|
        \le m
        \qquad
        \text{on }\mathcal N_{K,r_*}
        \text{ for all }\DBw\in\widehat{\DBv}(t,\DBx).
\end{equation}
After decreasing \(T\) once more, we may assume \(2MT<r_*\). Hence, for every
\(0<s<T\), the signed-distance levels
\[
        d(s,\DBx)=\pm Ms,
\]
which play a particular role below, lie inside \(\mathcal N_{K,r_*}(s)\). The
estimate \eqref{eq:speed-bound} above is the speed bound for the signed distance
along admissible trajectories inside \(\mathcal N_{K,r_*}\). The strict
inequality \(m<M\) is essential: once a trajectory has met the interface, it
cannot reach the levels \(d(s,\DBx)=\pm Ms\) by time \(s\).

\vskip2mm
\emph{Step 1: surface geometry of the transported interfacial section.}
\vskip1mm

The first claim compares \(D_s\subset\Sigma(s)\) with \(D(s,r)\), the set of points on \(\Sigma(s)\) whose normal displacement by signed distance
\(r\) belongs to \(G_0\).

\emph{Claim 1.} There are constants \(C_0>0\), \(\rho_0>0\) and \(T_0\in(0,T)\)
such that, for \(0<s<T_0\) and \(|r|<\rho_0\), the set
\[
        D(s,r):=\{\mathbf p\in\iface(s):
        \mathbf p+r\DBn_\Sigma(s,\mathbf p)\in G_0\}
\]
satisfies
\begin{equation}\label{DBeq:cutfill-slice-boundary-new}
        D_s^{-C_0(s+|r|)}
        \subset D(s,r)
        \subset D_s^{+C_0(s+|r|)} .
\end{equation}
Moreover,
\begin{equation}\label{DBeq:cutfill-layer-boundary-new}
        \mathcal H^2\big(D_s^{+\rho}\setminus D_s^{-\rho}\big)
        \le C_0\rho
        \qquad
        (0<\rho<\rho_0,\ 0<s<T_0).
\end{equation}

We use the signed-distance defining function \(q\) chosen above. Let \(U_\Gamma\subset\iface_0\)
be a fixed compact surface neighborhood of \(\Gamma\), chosen so small that
\(q\) is defined in a spatial neighborhood of \(U_\Gamma\) and
\[
        D_0\cap U_\Gamma
        =
        \{\mathbf z\in U_\Gamma:q(\mathbf z)\le0\}.
\]
After decreasing \(T_0\) and \(\rho_0\), if necessary, all points
\[
        \Psi_s(\mathbf z)+r\DBn_\Sigma(s,\Psi_s(\mathbf z)),
        \qquad
        \mathbf z\in U_\Gamma,\quad 0\le s<T_0,\quad |r|<\rho_0,
\]
lie in \(U_q\). Thus the functions
\[
        q_{s,r}(\mathbf z)
        :=
        q\bigl(\Psi_s(\mathbf z)+r\DBn_\Sigma(s,\Psi_s(\mathbf z))\bigr),
        \qquad \mathbf z\in U_\Gamma,
\]
are well defined. For \(\mathbf z\in U_\Gamma\) one has
\[
        \Psi_s(\mathbf z)\in D(s,r)
        \;\Leftrightarrow \;
        \Psi_s(\mathbf z)+r\DBn_\Sigma(s,\Psi_s(\mathbf z))\in G_0 
        \;\Leftrightarrow \;
        q_{s,r}(\mathbf z)\le0 .
\]
Thus, near the transported boundary curve, \(q_{s,r}\) is the level-set
description of \(D(s,r)\), pulled back from \(\iface(s)\) to the fixed surface
\(\iface_0\).

Since \(\partial_{\iface_0}D_0=\Gamma\) and
\(\nabla_{\iface_0}q\ne0\) on \(\Gamma\), compactness of $\Gamma$ gives
\[
        |\nabla_{\iface_0}q|\ge c_\Gamma>0
        \qquad\text{on }\Gamma
\]
for some \(c_\Gamma>0\). The \(C^1\)-dependence of the signed-distance
coordinates, of \(\DBn_\Sigma\), and of the surface flow \(\Psi_s\) gives, for
some constant \(C_q>0\) independent of \(s\) and \(r\),
\[
        \|q_{s,r}-q|_{\iface_0}\|_{C^1(U_\Gamma)}
        \le C_q(s+|r|)
\]
for all sufficiently small \(s\ge0\) and \(|r|\). Here, and in the following,
\(C^1\)-norms on \(U_\Gamma\subset\iface_0\) are computed in a fixed finite
system of surface charts. The
implicit function theorem with parameters, applied in finitely many surface
charts covering \(\Gamma\), therefore implies that the zero set of \(q_{s,r}\)
is a \(C^1\) graph over the corresponding part of \(\Gamma\), with graph norm
bounded by \(C_\Gamma(s+|r|)\) for some constant \(C_\Gamma>0\). Returning to
\(\iface(s)\) by \(\Psi_s\), this gives
\[
        \partial_{\iface(s)}D(s,r)
        \subset
        \{\mathbf p\in\iface(s):
        \operatorname{dist}_{\iface(s)}(\mathbf p,
        \partial_{\iface(s)}D_s)
        \le C_\Gamma(s+|r|)\},
\]
and conversely
\[
        \partial_{\iface(s)}D_s
        \subset
        \{\mathbf p\in\iface(s):
        \operatorname{dist}_{\iface(s)}(\mathbf p,
        \partial_{\iface(s)}D(s,r))
        \le C_\Gamma(s+|r|)\}.
\]
Thus the two boundary curves are close in the intrinsic metric of \(\iface(s)\),
with distance controlled by \(s+|r|\).

Away from the fixed neighborhood \(U_\Gamma\) of the boundary curve, the sets
under consideration are separated from their relative boundary by a positive
intrinsic distance. Equivalently, on compact subsets of
\(\operatorname{int}_{\iface_0}D_0\) and of \(\iface_0\setminus D_0\) which are
disjoint from a neighborhood of \(\Gamma\), membership in \(G_0\), respectively
in the complement of \(G_0\), is stable under the small perturbations generated
by \(\Psi_s\) and by the normal shift \(|r|<\rho_0\). Hence \(D(s,r)\) and
\(D_s\) can differ only in the intrinsic boundary strip around
\(\partial_{\iface(s)}D_s\) whose width is \(O(s+|r|)\), as estimated above.
Combining this with the preceding boundary-graph comparison proves
\eqref{DBeq:cutfill-slice-boundary-new} after choosing \(C_0\ge C_\Gamma\)
sufficiently large.

Finally, \(\partial_{\iface(s)}D_s=\Psi_s(\Gamma)\) is a compact embedded
\(C^1\) curve, with only finitely many connected components, on the \(C^2\)
surface \(\iface(s)\). Since \(\Gamma\) is compact and \(\Psi_s\) depends
continuously on \(s\) in \(C^1\) on compact sets, there is, after decreasing
\(T_0\) if necessary, a finite family of surface charts covering
\(\partial_{\iface(s)}D_s\) for all \(0<s<T_0\). In these charts the curve is
represented by Lipschitz graphs with uniformly bounded graph constants and
uniformly bounded total length. The intrinsic \(\rho\)-strip around such a
finite family of graphs has surface area bounded by \(C\rho\), uniformly for
small \(s\). Increasing
\(C_0\), if necessary, we choose it larger than all constants arising in the
preceding finite chart comparison and in this strip estimate. This proves
\eqref{DBeq:cutfill-layer-boundary-new}.

\vskip2mm
\emph{Step 2: reachability estimates in the interfacial strip.}
\vskip1mm

The signed-distance speed bound will be used to localize all genuinely
multivalued effects to the interfacial strip. We first record the elementary
global reachability estimate
\begin{equation}\label{DBeq:cutfill-outer-reach-boundary-new}
        G(s)\subset G_0+\overline B_{Ms}(0),
\end{equation}
which follows directly from the speed bound $|\dot\gamma|\le m<M$ along
admissible trajectories.

\emph{Claim 2.} There is a constant $\omega\ge C_0$ such that, after possibly
decreasing $T_0$, the following inclusions hold for all $0<s<T_0$:
\begin{equation}\label{DBeq:cutfill-outer-strip-boundary-new}
        G(s)\cap\{\DBx\in\mathcal N_{K,r_*}(s): |d(s,\DBx)|\le Ms\}
        \subset
        \{\mathbf p+r\DBn_\Sigma(s,\mathbf p):
        \mathbf p\in D_s^{+\omega s},\ |r|\le Ms\},
\end{equation}
and
\begin{equation}\label{DBeq:cutfill-inner-strip-boundary-new}
        \{\mathbf p+r\DBn_\Sigma(s,\mathbf p):
        \mathbf p\in D_s^{-\omega s},\ |r|\le Ms\}\subset G(s).
\end{equation}

In the proof of the claim we decrease $T_0$, if necessary, so that
$MT_0<\rho_0$; hence Claim~1 applies to all normal levels $|r|\le Ms$
considered for $0<s<T_0$.

We first prove \eqref{DBeq:cutfill-outer-strip-boundary-new}. Let
$\DBx\in G(s)\cap\mathcal N_{K,r_*}(s)$ with $|d(s,\DBx)|\le Ms$. Then
$\DBx=\gamma(s)$ for an admissible trajectory with
$\gamma(0)=\DBx_0\in G_0$, and $|\DBx-\DBx_0|\le ms$. Since
\[
        \operatorname{dist}(\DBx_0,\iface(s))
        \le |\DBx_0-\DBx|+\operatorname{dist}(\DBx,\iface(s))
        \le ms+Ms<2Ms<r_*,
\]
the point $\DBx_0$ also lies in $\mathcal N_{K,r_*}(s)$, after the choice of
$T$ made above. We may therefore write, in the signed-distance coordinates associated with $\iface(s)$,
\[
        \DBx=\mathbf p+r\DBn_\Sigma(s,\mathbf p),
        \qquad
        \DBx_0=\mathbf q+\rho\DBn_\Sigma(s,\mathbf q).
\]
The inverse signed-distance coordinate map
\[
        \DBx\mapsto(\pi_s(\DBx),d(s,\DBx))
\]
is Lipschitz on $\mathcal N_{K,r_*}(s)$, with Lipschitz constant bounded
uniformly for small $s$. Since $|\DBx-\DBx_0|\le ms$, this gives
\[
        \operatorname{dist}_{\iface(s)}(\mathbf p,\mathbf q)+|r-\rho|
        \le C_{\rm tub}s .
\]
Together with $|r|\le Ms$, and after increasing $C_{\rm tub}$ if necessary,
we also obtain $|\rho|\le C_{\rm tub}s$. The constant $C_{\rm tub}$ is
independent of $s$, of the admissible trajectory, and of the final choice of
the strip constant $\omega$.
After decreasing $T_0$, we may assume $C_{\rm tub}s<\rho_0$ for $0<s<T_0$.
Since $\DBx_0\in G_0$, one has $\mathbf q\in D(s,\rho)$. Claim~1 gives
\[
        \mathbf q\in D_s^{+C_0(s+|\rho|)}
        \subset D_s^{+C_0(1+C_{\rm tub})s}.
\]
Together with
$\operatorname{dist}_{\iface(s)}(\mathbf p,\mathbf q)\le C_{\rm tub}s$, this
implies
\[
        \mathbf p\in
        D_s^{+\{C_0(1+C_{\rm tub})+C_{\rm tub}\}s}.
\]
Choose a constant $C_*>0$ such that
\[
        C_*\ge C_0(1+C_{\rm tub})+C_{\rm tub}.
\]
Then $\mathbf p\in D_s^{+C_*s}$. The strip constant $\omega$ asserted in Claim~2 will
be chosen with $\omega\ge C_*$, and this proves
\eqref{DBeq:cutfill-outer-strip-boundary-new}.

It remains to prove \eqref{DBeq:cutfill-inner-strip-boundary-new}. We first
derive a purely geometric distance estimate with a depth parameter
$\Lambda\ge C_0$, keeping all constants in that estimate independent of
$\Lambda$. Only after this estimate has been obtained is the final strip
constant $\omega$ chosen, and the estimate is then applied with $\Lambda=\omega$; the
final time interval may be decreased after this choice.

Let
\[
        \DBx=\mathbf p+r\DBn_\Sigma(s,\mathbf p),
        \qquad
        \mathbf p\in D_s^{-\Lambda s},
        \qquad |r|\le Ms .
\]
After possibly decreasing $T_0$, the signed-distance parametrization and its inverse,
\[
        (\mathbf p,r)\mapsto \mathbf p+r\DBn_\Sigma(s,\mathbf p),
        \qquad
        \DBx\mapsto(\pi_s(\DBx),d(s,\DBx)),
\]
have Lipschitz constants bounded uniformly on $\mathcal N_{K,r_*}(s)$. Let
$U_\Gamma$ be the fixed neighborhood of $\Gamma$ used in Claim~1. If
$\mathbf p$ lies outside the transported neighborhood $\Psi_s(U_\Gamma)$,
then, by the reduction $\partial_{\iface_0}D_0=\Gamma$, the points
$\mathbf p\in D_s^{-\Lambda s}$ under consideration lie in a compact subset
of the relative interior of $D_s$, uniformly for small $s$. Since
$D_s=\Psi_s(D_0)$ converges to $D_0$ and $|r|\le Ms$, the points
$\mathbf p+r\DBn_\Sigma(s,\mathbf p)$ have positive distance from
$G_0^\complement$, uniformly for sufficiently small $s$. After the final
value of $\omega$ has been chosen, $T_0$ may therefore be decreased so that this
distance is at least $2Ms$ for all $0<s<T_0$. This proves the required distance
lower bound in this case. It remains to consider
$\mathbf p\in\Psi_s(U_\Gamma)$.

By Claim~1, the relative boundary $\partial_{\iface(s)}D(s,r)$ is contained
in the intrinsic $C_0(s+|r|)$-neighborhood of
$\partial_{\iface(s)}D_s$. Therefore the intrinsic distance of $\mathbf p$
from $\iface(s)\setminus D(s,r)$ is bounded from below by
\[
        \Lambda s-C_0(s+|r|).
\]
For the final choice of the strip constant below this lower bound will be
positive. In the local coordinates used in Claim~1, the function
\[
        \mathbf p'\mapsto
        q\bigl(\mathbf p'+r\DBn_\Sigma(s,\mathbf p')\bigr)
\]
is a defining function for $D(s,r)$ whose tangential derivative along the zero
set is bounded away from zero, uniformly for small $s$ and $|r|$. Hence, after
decreasing the coordinate neighborhoods if necessary, the uniform comparison
between this defining function and the intrinsic distance to its zero set gives
\[
        q(\DBx)\le -c\bigl(\Lambda s-C_0(s+|r|)\bigr).
\]
For the final choice of the strip constant below, the right-hand side is
negative. Since $q$ is the signed distance to $\partial G_0$, with negative
sign in $G_0^\circ$, this implies
\[
        \operatorname{dist}(\DBx,G_0^\complement)
        = -q(\DBx)
        \ge c\bigl(\Lambda s-C_0(s+|r|)\bigr).
\]
Since $|r|\le Ms$, this yields
\[
        \operatorname{dist}(\DBx,G_0^\complement)
        \ge c\bigl(\Lambda-C_0(1+M)\bigr)s .
\]
Now choose the final strip constant $\omega$ in Claim~2 so large that
\[
        \omega\ge C_*,
        \qquad \omega\ge C_0,
        \qquad c\bigl(\omega-C_0(1+M)\bigr)\ge 2M .
\]
With this final choice of $\omega$, decrease $T_0$ if necessary so that all
local coordinate estimates used in the preceding estimate apply with
$\Lambda=\omega$. The estimate above with $\Lambda=\omega$ hence gives, for
every $\mathbf p\in D_s^{-\omega s}$ and $|r|\le Ms$,
\[
        \operatorname{dist}(\DBx,G_0^\complement)\ge 2Ms .
\]

By Lemma~\ref{DBlem:kinematic-DI-existence}, together with the preceding choice
of $K$ and the reduction of $T_0$, there is an admissible solution branch
through $\DBx$ at time $s$, defined backward down to time $0$. Equivalently,
this gives an admissible absolutely continuous curve on $[0,s]$ whose endpoint
at time $s$ is $\DBx$. The speed bound shows that its time-zero value lies
within distance at most $ms<Ms$ from $\DBx$, hence belongs to $G_0$. Viewed on
the interval $[0,s]$, the same branch is therefore a solution of the original
inclusion with initial value in $G_0$ and endpoint $\DBx$ at time $s$. Thus
$\DBx\in G(s)$.

\vskip2mm
\emph{Step 3: the one-sided flows outside the strip.}
\vskip1mm
This step provides a fixed-domain parametrization of the parts of the
attainable set lying outside the interfacial strip. This will be used below
to verify the Lipschitz space-time geometry of the cut pieces and to identify
the boundary contributions coming from the transported parts of the original
boundary \(\partial G_0\).

Define the parts of the transported set outside the artificial strip, within the
fixed compact set \(K\), by
\[
        A_s^\pm:=G(s)\cap\{\DBx\in K:\ \pm d(s,\DBx)>Ms\}.
\]
The point of the preceding strip construction is that a trajectory which has
touched the interface cannot reach \(A_s^\pm\) at time \(s\). Hence the dynamics
on these outside pieces is governed by the corresponding one-sided velocity
field. We use auxiliary single-valued extensions of \(\DBv^\pm\) to parametrize
these outside pieces by classical flow maps. The extensions are chosen so that they agree with \(\DBv^\pm\) on all relevant
one-sided phase trajectories and so that, in a smaller tubular neighborhood,
the signed-distance speed estimate used above remains valid for the auxiliary
fields.

More precisely, on \(\mathcal N_{K,r_*}\) write\vspace{-0.1in}
\[
        \DBx=\mathbf p+r\DBn_\Sigma(t,\mathbf p)\vspace{-0.05in}
\]
and set\vspace{-0.05in}
\[
        r^\pm:=\max\{\pm r,0\},
        \qquad
        \Pi_t^\pm(\DBx):=\mathbf p\pm r^\pm\DBn_\Sigma(t,\mathbf p).
\]
Then \(\Pi_t^\pm\) is the identity on \(\overline{\Omega^\pm(t)}\) and maps
\(\Omega^\mp(t)\) to \(\iface(t)\) along the normal projection. In a smaller
tubular neighborhood of \(\gr(\Sigma)\cap((-T,T)\times K)\) we set
\[
        \widetilde{\DBv}^{\,+}(t,\DBx)
        :=\DBv^+(t,\Pi_t^+(\DBx)),\qquad
        \widetilde{\DBv}^{\,-}(t,\DBx)
        :=\DBv^-(t,\Pi_t^-(\DBx)).
\]
After decreasing \(T_0\) and the tubular radius, if necessary, these fields are
jointly continuous and locally uniformly Lipschitz in the spatial variable on
the smaller tube. They agree there with \(\DBv^\pm\) on the corresponding closed
phase.

We now extend \(\widetilde{\DBv}^{\,\pm}\) to single-valued vector fields on a
neighborhood of \((-T,T)\times K\), still denoted by
\(\widetilde{\DBv}^{\,\pm}\), which are jointly continuous and locally uniformly
Lipschitz in space, and which satisfy
\[
        \widetilde{\DBv}^{\,\pm}=\DBv^\pm
        \quad\text{on } \overline{\Omega^\pm(t)}\cap K,
        \qquad
        \widetilde{\DBv}^{\,\pm}(t,\DBx)
        =\DBv^\pm(t,\Pi_t^\pm(\DBx))
        \quad\text{in a smaller interfacial tube.}
\]
To construct them, choose two nested signed-distance collars. On the inner
collar use the fields defined above. On the opposite phase in the annular part
of the larger collar, blend these fields, by a Lipschitz cut-off depending only
on the signed distance, with a fixed bounded vector; outside the larger collar
on that side use this fixed vector, while on the original phase retain
\(\DBv^\pm\). Since the collar extension agrees with \(\DBv^\pm\) on the
closed original phase, this gives jointly continuous fields. The local uniform
spatial Lipschitz bounds of \(\DBv^\pm\), together with the uniform Lipschitz
bounds for the signed-distance coordinates, give the required uniform spatial
Lipschitz bounds. Thus the flows introduced below represent the original
one-sided phase dynamics whenever the corresponding trajectory stays in that
phase; after decreasing \(T_0\) once more, all trajectories used below remain
in the neighborhood where the preceding agreement holds.

Let \(X_s^\pm\) denote the classical flows generated by
\(\widetilde{\DBv}^{\,\pm}\), starting at time \(0\). Since
\(\widetilde{\DBv}^{\,\pm}\) are bounded and locally Lipschitz in space on
the compact region under consideration, Gronwall's inequality gives, after
decreasing \(T_0\) if necessary,
\begin{equation}\label{DBeq:cutfill-flow-small-new}
        \operatorname{Lip}(X_s^\pm-\operatorname{Id})\le Cs,
        \qquad
        \sup_{\xi\in G_0}|X_s^\pm(\xi)-\xi|\le Cs .
\end{equation}
This is the standard short-time flow estimate for bounded vector fields that are
uniformly Lipschitz in space: if \(L\) is a spatial Lipschitz bound for
\(\widetilde{\DBv}^{\,\pm}\) on the relevant compact region, then
\[
        \operatorname{Lip}(X_s^\pm-\operatorname{Id})
        \le e^{Ls}-1\le Cs,
        \qquad
        \sup_{\xi\in G_0}|X_s^\pm(\xi)-\xi|
        \le s\|\widetilde{\DBv}^{\,\pm}\|_\infty\le Cs .
\]

The signed distance is now used in two different ways. On the fixed compact set
\(K\), for times in \((-T,T)\), it is used only as a Lipschitz function which
detects on which side of the interface a point lies. Differentiability of \(d\), of \(\pi_s\), and of
the signed-distance coordinates is used only in the fixed tubular neighborhood where
these objects are \(C^1\).

Define
\[
        f_s^\pm(\xi):=d(s,X_s^\pm(\xi)).
\]
We write this level function as a perturbation of the initial signed distance by setting
\[
        \mathcal R_s^\pm(\xi):=f_s^\pm(\xi)-d(0,\xi)
        =d(s,X_s^\pm(\xi))-d(0,\xi).
\]
Since the flow displacement is \(O(s)\) uniformly for \(\xi\in G_0\), and since
the signed-distance functions \(d(s,\cdot)\) vary by \(O(s)\) on the fixed compact
set \(K\), this remainder satisfies
\begin{equation}\label{DBeq:cutfill-level-small-new}
        f_s^\pm(\xi)=d(0,\xi)+\mathcal R_s^\pm(\xi),
        \qquad
        \sup_{\xi\in G_0}|\mathcal R_s^\pm(\xi)|\le Cs .
\end{equation}
In addition, the small Lipschitz estimates for \(\mathcal R_s^\pm\) hold locally
near the initial interface. More precisely, fix
\[
        U_{\rm tub}:=\{\xi\in K: |d(0,\xi)|<r_*/4\}.
\]
Since the auxiliary fields and their flows are defined on a neighborhood of
\(K\), the estimates used here are not restricted to points of \(G_0\). After
decreasing \(T_0\), if necessary,
\[
        X_s^\pm(U_{\rm tub})\subset\mathcal N_{K,r_*}(s),
        \qquad
        \sup_{\xi\in U_{\rm tub}}|\mathcal R_s^\pm(\xi)|\le Cs
        \qquad (0<s<T_0),
\]
and the \(C^1\) regularity of the signed-distance coordinates, together with
\eqref{DBeq:cutfill-flow-small-new}, yields
\[
        \operatorname{Lip}\bigl(\mathcal R_s^\pm;U_{\rm tub}\bigr)\le Cs .
\]
Moreover, after decreasing \(T_0\) if necessary, the same argument gives the
joint local estimate
\begin{equation}\label{DBeq:cutfill-level-joint-lip-new}
        |\mathcal R_s^\pm(\xi)-\mathcal R_\tau^\pm(\eta)|
        \le C|s-\tau|+C T_0|\xi-\eta|
\end{equation}
whenever \(0<s,\tau<T_0\) and \(\xi,\eta\in U_{\rm tub}\). Indeed, for the spatial part one uses the preceding estimate
\(\operatorname{Lip}(\mathcal R_s^\pm;U_{\rm tub})\le Cs\), hence
\[
        |\mathcal R_s^\pm(\xi)-\mathcal R_s^\pm(\eta)|
        \le Cs|\xi-\eta|\le CT_0|\xi-\eta| .
\]
For the time part, the uniform Lipschitz dependence of
\(s\mapsto X_s^\pm(\eta)\), together with the bounded first derivatives of \(d\)
in \(\mathcal N_{K,r_*}\), gives
\[
        |\mathcal R_s^\pm(\eta)-\mathcal R_\tau^\pm(\eta)|
        \le C|s-\tau|.
\]
Combining the two estimates yields
\eqref{DBeq:cutfill-level-joint-lip-new}.

These local estimates are sufficient for the cut geometry in Step~4. Indeed, if
\(\pm f_s^\pm(\xi)=Ms\), then the estimate in
\eqref{DBeq:cutfill-level-small-new} gives
\[
        |d(0,\xi)|\le (M+C)s .
\]
After decreasing \(T_0\), all such points, and a fixed small neighborhood of the
corresponding level sets, lie in \(U_{\rm tub}\). Away from this initial tubular
neighborhood the sign of \(\pm f_s^\pm-Ms\) is fixed for small \(s\), and no
differentiability of the signed distance is used.

The tubular radius and the constants are chosen so that,
after possibly increasing \(m\) while keeping \(m<M\), the auxiliary fields
satisfy the same signed-distance speed bound,
\[
        |\partial_t d(t,\DBx)+\nabla_\DBx d(t,\DBx)\cdot
        \widetilde{\DBv}^{\,\pm}(t,\DBx)|\le m<M,
\]
for all \((t,\DBx)\) in the smaller interfacial tube used above. On the interface
this follows from \eqref{eq:speed-bound} applied to the one-sided traces, and in
the tube it follows by continuity after choosing the tubular radius sufficiently
small.

The next claim makes precise that, outside the strip, the multivalued
attainable-set evolution is represented by the corresponding one-sided flow. It
also represents the artificial cut faces as images of level sets in the fixed
initial set \(G_0\). This fixed-domain parametrization will be used in the
Lipschitz-regularity analysis in the next step.

\emph{Claim 3.} For \(0<s<T_0\),
\begin{equation}\label{DBeq:cutfill-As-flow-representation-new}
        A_s^\pm=X_s^\pm(E_s^\pm),
        \qquad
        E_s^\pm:=\{\xi\in G_0:\ \pm f_s^\pm(\xi)>Ms\}.
\end{equation}
Moreover, define the artificial cut faces by
\begin{equation}\label{DBeq:cutfill-face-definition-new}
        P_s^\pm:=G(s)\cap\{\DBx\in K:\ \pm d(s,\DBx)=Ms\}.
\end{equation}
Then these faces are represented as
\begin{equation}\label{DBeq:cutfill-face-flow-representation-new}
        P_s^\pm=X_s^\pm(F_s^\pm),
        \qquad
        F_s^\pm:=\{\xi\in G_0:\ \pm f_s^\pm(\xi)=Ms\}.
\end{equation}

It suffices to prove the plus-sign statements; the corresponding minus-sign statements follow analogously. Let
\(\DBx\in A_s^+\). Then \(\DBx=\gamma(s)\) for an admissible trajectory with
\(\gamma(0)\in G_0\). If this trajectory met the interface at some time
\(\tau\in[0,s]\), then, since \(d(s,\gamma(s))>Ms\), there would be a first time
\(\sigma\in(\tau,s]\) after \(\tau\) at which
\(|d(\sigma,\gamma(\sigma))|=Ms\). On \([\tau,\sigma]\) the curve lies in the
tubular neighborhood, and the signed-distance speed bound gives
\[
        Ms\le m(\sigma-\tau)\le ms<Ms,
\]
a contradiction. Hence the trajectory never meets the interface. Since \(d(s,\gamma(s))>Ms\),
its endpoint belongs to \(\Omega^+(s)\), and therefore the trajectory stays in
\(\Omega^+(\theta)\) for \(0\le \theta\le s\). There
the differential inclusion is single-valued and agrees with the ODE generated by
\(\DBv^+\), equivalently with the auxiliary flow \(X^+\). Thus
\[
        \gamma(\theta)=X_\theta^+(\gamma(0))
        \qquad (0\le\theta\le s),
\]
and \(\DBx\in X_s^+(E_s^+)\).

Conversely, let \(\xi\in E_s^+\). If the auxiliary trajectory
\(\theta\mapsto X_\theta^+(\xi)\) met the interface at some time
\(\tau\in[0,s]\), then the same first-exit argument, using the auxiliary
signed-distance speed bound in the tubular neighborhood, would contradict
\(f_s^+(\xi)>Ms\). Thus this auxiliary trajectory does not meet the interface.
Since its endpoint satisfies \(d(s,X_s^+(\xi))>Ms\), it stays in
\(\Omega^+(\theta)\) for \(0\le \theta\le s\). On this interval the auxiliary
velocity agrees with \(\DBv^+\),
hence the trajectory is an admissible branch of the differential inclusion. Its
endpoint therefore belongs to \(A_s^+\). This proves the plus case of
\eqref{DBeq:cutfill-As-flow-representation-new}; the minus case follows analogously.

The representation of the artificial cut face in
\eqref{DBeq:cutfill-face-flow-representation-new} follows from the same
first-exit argument with equality in the terminal level. Indeed, an admissible
trajectory ending in \(P_s^+\), or an auxiliary trajectory starting from
\(F_s^+\), cannot have met the interface before time \(s\): otherwise, for the
first time \(\sigma\in(\tau,s]\) after such a contact at which
\(|d(\sigma,\cdot)|=Ms\), the signed-distance speed bound gives
\[
        Ms\le m(\sigma-\tau)<M(\sigma-\tau)\le Ms,
\]
a contradiction. Hence the relevant trajectory stays in \(\Omega^+(\theta)\) for
\(0\le\theta\le s\), where it is represented by \(X^+\) and agrees with an
admissible branch of the differential inclusion. This proves the plus case of
\eqref{DBeq:cutfill-face-flow-representation-new}; the minus case follows
analogously.

\vskip2mm
\emph{Step 4: Lipschitz regularity of the cut-and-fill geometry.}
\vskip1mm
The artificial cut faces \(P_s^\pm\) are now projected onto the interface; the resulting sets will serve as the bases of the normal fills added back to the interface. Define
\[
        B_s^\pm:=\pi_s(P_s^\pm),
        \qquad
        C_s^\pm:=\{\mathbf p\pm r\DBn_\Sigma(s,\mathbf p):
        \mathbf p\in B_s^\pm,
        \ 0<r<Ms\}.
\]
Finally set\vspace{-0.1in}
\[
        \widetilde G_s^\pm:=A_s^\pm\cup C_s^\pm .
\]
At this point we adopt a boundary convention to avoid unnecessary notation.
For the cut-and-fill sets we do not distinguish notationally between representatives
which differ only by subsets of the Lipschitz boundary faces just introduced. Whenever
a Lipschitz domain is meant, the corresponding open representative is used; whenever
a compact Lipschitz domain is required, its closure is used. These representatives differ
only on sets of zero volume, so all volume integrals are unchanged. In particular,
$P_s^\pm$ and $B_s^\pm$ are retained as the artificial cut and interface faces in the
boundary geometry below.

This is the only point where the detailed Lipschitz geometry of the cut-and-fill
construction is needed; it ensures that Theorem~\ref{thm:single-phase-boundary-rtt}
can be applied on each interval $(\varepsilon,h)$.

\emph{Claim 4.} For all $0<s<T_0$, the sets $A_s^\pm$, $C_s^\pm$ and
$\widetilde G_s^\pm$, understood with the preceding boundary convention, are compact
Lipschitz domains, and $B_s^\pm$ are compact Lipschitz surface domains on $\iface(s)$.
If $0<\varepsilon<h<T_0$, then the closures of the corresponding swept space-time tubes generated by
\[
\begin{gathered}
        \{(s,\DBx):\varepsilon<s<h,
        \ \DBx\in A_s^\pm\},\qquad
        \{(s,\DBx):\varepsilon<s<h,
        \ \DBx\in C_s^\pm\},\\
        \{(s,\DBx):\varepsilon<s<h,
        \ \DBx\in\widetilde G_s^\pm\}
\end{gathered}
\]
are compact Lipschitz domains in space-time. Moreover, their boundaries have the natural
time-slice/lateral decomposition, up to null sets. The Lipschitz constants of the graph
representations, the areas of $B_s^\pm$, and the lengths of $\partial B_s^\pm$ are bounded
uniformly for $0<s<T_0$. {The normal velocities of the swept normal lateral
faces $L_s^\pm$ are also bounded uniformly for $0<s<T_0$, away from their edge sets.}
The only non-uniform geometric quantity is the height of the fills $C_s^\pm$, which is
$Ms$; on each interval $\varepsilon<s<h$ this height is positive, and all estimates used
below are independent of $\varepsilon$.

We prove the claim for the plus sign; the minus-sign case follows analogously
after reversing the signed-distance coordinate. The proof has four parts, corresponding to the four
geometric operations used below: first the cut is described in the fixed initial coordinates,
then it is transported by the one-sided flow, then the normal fill is added, and finally the
outside piece and the fill are glued along the artificial cut face.

First, we analyze the cut face before applying the one-sided flow. The only delicate
point is its interaction with $\partial G_0$ near $\Gamma$. Away from an $O(s)$-neighborhood
of $\Gamma$, the artificial cut either lies in the relative interior of $G_0$, where it is just
a single Lipschitz graph, or it does not meet $G_0$. Thus no two-boundary transversality
issue arises there. Near a point of $\Gamma=\partial G_0\cap\iface_0$ choose
signed-distance coordinates $(y_1,y_2,y_3)$ associated with $\iface_0$, so that
$y_3=d(0,\xi)$ and $\iface_0$ is given by $y_3=0$. We work in a fixed coordinate
cylinder $U'\times(-\delta,\delta)$ in these coordinates, whose closure is contained in $U_{\rm tub}$ and which is chosen small enough that the
estimates below hold there.

The decomposition in \eqref{DBeq:cutfill-level-small-new} gives
$f_s^+=d(0,\cdot)+\mathcal R_s^+$. Hence, in these coordinates, the equation
$f_s^+(\xi)=Ms$ becomes
\begin{equation}\label{DBeq:cutfill-cut-equation-new}
        y_3+\mathcal R_s^+(y_1,y_2,y_3)=Ms .
\end{equation}
We call \eqref{DBeq:cutfill-cut-equation-new} the cut equation. Choose
$C_R>0$ such that
\[
        \|\mathcal R_s^+\|_{L^\infty(U'\times(-\delta,\delta))}\le C_Rs
\]
for all $0<s<T_0$. After decreasing $T_0$, if necessary, we may assume
\[
        (M+C_R)T_0<\delta
\]
and, using \(\operatorname{Lip}(\mathcal R_s^+;U_{\rm tub})\le Cs\),
that
\[
        |\mathcal R_s^+(y',z_1)-\mathcal R_s^+(y',z_2)|
        \le \frac12 |z_1-z_2|
\]
in the coordinate cylinder.

For fixed $s$ and $y'=(y_1,y_2)$, define
\[
        T_{s,y'}(z):=Ms-\mathcal R_s^+(y',z)
\]
on the closed interval
\[
        I_s:=[Ms-C_Rs,\,Ms+C_Rs].
\]
The preceding choice of $T_0$ ensures that $I_s\subset(-\delta,\delta)$. Moreover,
for $z\in I_s$,
\[
        |T_{s,y'}(z)-Ms|
        =
        |\mathcal R_s^+(y',z)|
        \le C_Rs,
\]
and hence $T_{s,y'}(I_s)\subset I_s$. Since the $y_3$-Lipschitz constant of
$\mathcal R_s^+$ is smaller than $1/2$, the map $T_{s,y'}$ is a contraction on $I_s$.
Banach's fixed point theorem therefore gives a unique solution of
\eqref{DBeq:cutfill-cut-equation-new} in $I_s$. 
Conversely, every solution of
\eqref{DBeq:cutfill-cut-equation-new} in the coordinate cylinder belongs to
$I_s$, by the bound $|\mathcal R_s^+|\le C_Rs$.
We write this solution as
\begin{equation}\label{DBeq:cutfill-cut-graph-new}
        y_3=Ms+\rho(s,y')
\end{equation}
and obtain
\begin{equation}\label{DBeq:cutfill-rho-estimates-new}
        |\rho(s,y')|\le Cs,
        \quad
        |\rho(s,y')-\rho(s,z')|\le Cs|y'-z'|,
        \quad
        |\rho(s,y')-\rho(\tau,y')|\le C|s-\tau| .
\end{equation}
Indeed, the first estimate follows from the fact that the fixed point belongs
to $I_s$. For the two Lipschitz estimates, write the fixed-point equation as
\[
        \rho(s,y')=-\mathcal R_s^+\bigl(y',Ms+\rho(s,y')\bigr).
\]
For fixed $s$ and $y',z'$, comparison of the two equations gives
\[
\begin{aligned}
|\rho(s,y')-\rho(s,z')|
&\le Cs|y'-z'|+\frac12|\rho(s,y')-\rho(s,z')|,
\end{aligned}
\]
and hence
\[
        |\rho(s,y')-\rho(s,z')|\le Cs|y'-z'|.
\]
Similarly, using the joint estimate \eqref{DBeq:cutfill-level-joint-lip-new},
for fixed $y'$ one obtains
\[
\begin{aligned}
|\rho(s,y')-\rho(\tau,y')|
&\le C|s-\tau|+CT_0\bigl(M|s-\tau|
       +|\rho(s,y')-\rho(\tau,y')|\bigr).
\end{aligned}
\]
After decreasing $T_0$, if necessary, the last term is absorbed into the
left-hand side, yielding
\[
        |\rho(s,y')-\rho(\tau,y')|\le C|s-\tau|.
\]

Thus $F_s^+$ is, in these signed-distance coordinates, represented by
\eqref{DBeq:cutfill-cut-graph-new}, with uniformly bounded spatial Lipschitz
constants and with Lipschitz dependence on $s$. Equivalently, the swept cut face
\[
        \{(s,\xi):\varepsilon<s<h,\ f_s^+(\xi)=Ms\}
\]
is locally a Lipschitz graph in the variables $(s,y')$.

It remains to incorporate the constraint $\xi\in G_0$. Let $q$ be the
signed-distance defining function for $\partial G_0$ introduced above. Along
$\Gamma=\partial G_0\cap\iface_0$, transversality gives
\[
        \nabla_{\iface_0}q\ne 0 .
\]
Fix $\zeta\in\Gamma$. In signed-distance coordinates $(y_1,y_2,y_3)$
associated with $\iface_0$, with $y_3=d(0,\xi)$, we may, after a rotation of the
tangential variables, assume that $\partial_{y_2}q(\zeta)\ne0$. Hence, after
shrinking the coordinate neighborhood, $\partial G_0$ is represented as
\[
        y_2=g(y_1,y_3)
\]
with $g$ of class $C^2$. By \eqref{DBeq:cutfill-cut-graph-new}, the cut face
$f_s^+=Ms$ is represented as a Lipschitz graph with spatial Lipschitz constant
$O(s)$. Therefore, after decreasing $T_0$ if
necessary, the cut face $f_s^+=Ms$ and $\partial G_0$ are uniformly transverse
near $\Gamma$.

By compactness of $\Gamma$, finitely many such neighborhoods suffice. In each of
them one may choose a fixed coordinate direction which is transverse both to
$\partial G_0$ and to all cut faces $f_s^+=Ms$, $0<s<T_0$. In the corresponding
coordinates $\eta=(\eta',\eta_3)$, the two faces are represented as
\[
        \eta_3=g(\eta'),\qquad \eta_3=k(s,\eta'),
\]
where $g$ is $C^2$, while $k$ is Lipschitz in $(s,\eta')$, with Lipschitz
constants bounded independently of $s$.

{
The local part of $G_0$ on which $f_s^+>Ms$ is the intersection of the inward
sides of these two graph faces. At points where only one face is active this is
a single Lipschitz graph domain. At an intersection point, orient both unit
normals inward. Uniform transversality and compactness give, after fixing the
finite cover, a spatial unit vector $e$ on each chart and $c_0>0$ such that
$e\cdot\nu_j\ge c_0$ for every a.e. inward unit normal of either active face,
uniformly in $s$: take $e$ in the direction of the sum of the two reference
normals at the chart center and use continuity of the normal to $\partial G_0$
and the $O(s)$ slope of the cut graph. Thus the corresponding local defining
functions, with inward side positive, are uniformly strictly increasing on
lines parallel to $e$. Each face is therefore a Lipschitz graph over $e^\perp$
and the two inward sides are simultaneous graph inequalities, with uniform
constants. Lemma~\ref{lem:uniform-finite-face-gluing} applies. Since $k$ is
Lipschitz in $(s,\eta')$, the same $e$ gives uniform space-time lateral graphs.
At an artificial time face with inward unit time normal $\tau$, a fixed
$e+\beta\tau$, $\beta>0$ sufficiently small, is a common graph direction for
the time face and all active lateral faces. Hence the lemma also gives Lipschitz
graph representations for the space-time set
}
\[
        E_{\varepsilon,h}^+
        :=\{(s,\xi):\varepsilon<s<h,\ \xi\in G_0,\ f_s^+(\xi)>Ms\}.
\]

Away from an $O(s)$-neighborhood of $\Gamma$, the constraint $\xi\in G_0$ and
the inequality $f_s^+>Ms$ do not interact. There, either the outer boundary
$\partial G_0$ is absent from the local chart, or the sign of $f_s^+-Ms$ is
fixed for all sufficiently small $s$. Thus only a single Lipschitz graph face has to be
considered. A finite covering therefore shows that the sets $E_s^+$ are
Lipschitz domains with uniform graph constants, and that $E_{\varepsilon,h}^+$
is a Lipschitz domain in space-time.

Moreover, the local graph constants obtained above give a cone opening angle which is
uniform for $0<s<T_0$ and for $0<\varepsilon<h<T_0$. The cone height for the
space-time sets may depend on $\varepsilon$ and $h$, but this is harmless: on each fixed
interval $(\varepsilon,h)$ the height is positive, and the small-perturbation threshold in
Theorem~\ref{DBthm:small-biLip-perturbation-cone} depends only on the opening angle.
Hence this threshold can be chosen independently of $s$, $\varepsilon$, and $h$.

Second, we pass from the preimage sets to the transported outside pieces. Under
the present hypotheses the maps $X_s^+$ are bi-Lipschitz flow maps, but they
need not be $C^1$-diffeomorphisms, since $\widetilde{\DBv}^{\,+}$ is only
Lipschitz in the spatial variable. We therefore use the small-perturbation
stability of Lipschitz domains from Theorem~\ref{DBthm:small-biLip-perturbation-cone}.
For fixed $s$, the map $X_s^+=\operatorname{Id}+(X_s^+-\operatorname{Id})$ is a
Lipschitz perturbation of the identity with perturbation constant $O(s)$ by
\eqref{DBeq:cutfill-flow-small-new}. For small $T_0$ this is below the uniform
perturbation threshold from Theorem~\ref{DBthm:small-biLip-perturbation-cone}. Hence
$A_s^+=X_s^+(E_s^+)$ is a compact Lipschitz domain, in the boundary convention fixed
above.

The corresponding space-time tube is also needed. Fix an auxiliary scale $a>0$ and write
$s=a\sigma$. On the rescaled cylinder consider\footnote{For more details concerning this scaling argument see the proof of Corollary~\ref{cor:flow-map-rtt}.}
\[
        \mathcal X_a(\sigma,\xi):=(\sigma,X_{a\sigma}^+(\xi)).
\]
Since \(\widetilde{\DBv}^{\,+}\) is bounded and
\eqref{DBeq:cutfill-flow-small-new} holds uniformly,
\[
\big| (X_{a\sigma}^+(\xi)-\xi)
      -(X_{a\tau}^+(\eta)-\eta) \big|
\le Ca|\sigma-\tau|+Ch|\xi-\eta|
\]
for $a\sigma,a\tau\in(\varepsilon,h)$. Thus
$\mathcal X_a-\operatorname{Id}$ has Lipschitz constant at most $C(a+h)$.
Under the rescaling $s=a\sigma$, the spatial graph constants of the preimage
domains are unchanged, while the graph slopes in the time variable are multiplied
by $a$; the product-type corners at the artificial time faces remain uniformly
Lipschitz. Hence, for $0<a\le a_0$, the small-perturbation threshold in
Theorem~\ref{DBthm:small-biLip-perturbation-cone} can be bounded from below
independently of $a$, $\varepsilon$ and $h$. Choose first $a>0$ sufficiently small
and then decrease $T_0$ once more so that $C(a+T_0)$ is below this threshold.
Then the image of the rescaled space-time domain is Lipschitz for every
$0<\varepsilon<h<T_0$. Scaling time back by the $C^1$ diffeomorphism
$(\sigma,\DBx)\mapsto(a\sigma,\DBx)$ preserves Lipschitz regularity. Hence the
swept closed tube of the sets $A_s^+$ is a compact Lipschitz domain.
The graph construction also shows that its boundary consists, up to null sets, of the two
time-slice faces and the lateral faces described above.

Third, we study the cut face and the fill. The set $F_s^+$ is obtained by
restricting the cut graph \eqref{DBeq:cutfill-cut-graph-new} to $G_0$. Hence
$F_s^+$ is a Lipschitz surface domain. Its relative boundary is contained in
the intersection of this cut graph with $\partial G_0$, and the preceding finite
graph representations give a uniform bound for its length. Its image
$P_s^+=X_s^+(F_s^+)$ is again a Lipschitz surface
domain. Since $P_s^+$ lies on the level $d(s,\cdot)=Ms$, the projection
$\pi_s:P_s^+\to\iface(s)$ is a bi-Lipschitz $C^1$ perturbation of the identity, uniformly
for small $s$. Hence $B_s^+=\pi_s(P_s^+)$ is a compact Lipschitz surface domain on
$\iface(s)$, and both $\mathcal H^2(B_s^+)$ and $\mathcal H^1(\partial B_s^+)$ are
uniformly bounded. For the swept bases, one composes the joint Lipschitz graph
representation of the cut faces with the $C^1$ map $(s,\DBx)\mapsto
(s,\pi_s(\DBx))$. On the levels $d(s,\DBx)=Ms$ this projection is uniformly
bi-Lipschitz for small $s$, because it is a $C^1$ perturbation of the identity in
the signed-distance coordinates. Thus the swept bases also have Lipschitz space-time
representations.

The signed-distance parametrization
$(\mathbf p,r)\mapsto\mathbf p+r\DBn_\Sigma(s,\mathbf p)$ is a $C^1$ diffeomorphism for
$|r|<r_*$. Therefore the closed fill $C_s^+$ is a compact Lipschitz domain. Its height is
$Ms$, so its cone height may degenerate as $s\downarrow0$, but its graph slopes and the
area estimates below are uniform. In finitely many surface charts the swept bases are therefore represented by
Lipschitz graph domains
\[
        \mathcal B_{\varepsilon,h}^+
        :=\{(s,\mathbf p):\varepsilon<s<h,
        \ \mathbf p\in B_s^+\}
        \subset \gr(\iface).
\]
In the same surface charts the swept fill is the image of the Lipschitz set
\[
        \{(s,z,r):(s,z)\in \mathcal B_{\varepsilon,h}^+,
        \ 0\le r\le Ms\}
\]
under the local $C^1$ diffeomorphism
$(s,\mathbf p,r)\mapsto(s,\mathbf p+r\DBn_\Sigma(s,\mathbf p))$. Since
$\varepsilon<s<h$, the height is strictly positive on the considered interval, and the
graph representations remain Lipschitz. Hence the swept closed fill is a compact
Lipschitz domain, again with the natural time-slice/lateral boundary decomposition up to
null sets.

Fourth, we form the union of the outside piece and the fill. The sets
$A_s^+$ and $C_s^+$ meet only along the common face $P_s^+$. In the
signed-distance coordinates $(\DBp,r)$ associated with \(\iface(s)\), with
$r=d(s,\DBx)$, this face is given by
$P_s^+=\{\DBp+Ms\,\DBn_\Sigma(s,\DBp):\DBp\in B_s^+\}$. Over relative interior
points of $B_s^+$, the fill is represented by
$\{(\DBp,r):0\le r\le Ms\}$, whereas the outside piece is represented by
$\{(\DBp,r):r\ge Ms\}$. Hence their union is locally represented by
$\{(\DBp,r):r\ge0\}$. Thus the common face $P_s^+$ is not part of the boundary
of $\widetilde G_s^+=A_s^+\cup C_s^+$, except along its edge over
$\partial_{\iface(s)}B_s^+$.

{
There are two types of edges of the resulting boundary. Along the lower edge
$r=0$, $\mathbf p\in\partial_{\iface(s)}B_s^+$, the interface face and the
normal lateral face form a product corner in signed-distance and
surface-boundary coordinates. Along the upper edge
$r=Ms$, the common cut face has disappeared and the normal lateral face meets
the material or cut-generated face. In the local coordinates constructed in
the first part, the latter two faces have the form
\[
        \eta_2=b_s(\eta_1),\qquad
        \eta_2=a_s(\eta_1,r),\quad r\ge Ms,
        \qquad a_s(\eta_1,Ms)=b_s(\eta_1),
\]
with uniform Lipschitz bounds, and the union lies on the same side of both
faces. Thus the active inward sides admit a common spatial graph direction at
both types of edge. Compactness and the finite atlas make the corresponding
graph bounds uniform. By the boundary convention above, $\widetilde G_s^+$ is the closure of the regularized union of the corresponding open outside and fill domains. Lemma~\ref{lem:uniform-finite-face-gluing} therefore
shows that $\widetilde G_s^+$ is a compact Lipschitz domain.

The same common spatial graph directions work after $s$ is added as a tangential
variable. The swept common cut face is an interior hypersurface of the union of
the swept outside tube and the swept fill tube, except along its swept edge;
there the remaining lateral faces satisfy the finite-face criterion of
Lemma~\ref{lem:uniform-finite-face-gluing}. At the time edges, the same
$e+\beta\tau$ construction gives a common graph direction uniformly in
$\varepsilon$ and $h$. Hence the swept closed tube of
$\widetilde G_s^+$ is a compact regular closed Lipschitz domain with the
natural time-slice/lateral boundary decomposition, up to $\mathcal H^3$-null sets.
Moreover, the swept normal lateral face is locally of the form
$y_3=\ell(s,y')$ with $y_3$ a spatial coordinate and with
$\operatorname{Lip}_s\ell\le C$. The last assertion of
Lemma~\ref{lem:uniform-finite-face-gluing} gives
$|V_{L,s}^+|\le C$ almost everywhere away from the edge sets. This proves
Claim~4.
}

\vskip2mm
\emph{Step 5: comparison with the true co-moving set in the strip.}
\vskip1mm
The bases \(B_s^\pm\) are obtained by projecting the artificial cut faces \(P_s^\pm\) to the interface. Claim~2 therefore implies that these bases can differ from \(D_s\) only in the intrinsic boundary layer of width \(O(s)\).
From Claim 2 and the definitions of the bases, after increasing $\omega$ if necessary,\vspace{-0.05in}
\begin{equation}\label{DBeq:cutfill-base-inclusion-boundary-new}
        D_s^{-\omega s}\subset B_s^\pm\subset D_s^{+\omega s}.
\end{equation}
For the first inclusion, use \eqref{DBeq:cutfill-inner-strip-boundary-new} on the level
$r=\pm Ms$; for the second, use \eqref{DBeq:cutfill-outer-strip-boundary-new}. Combining
\eqref{DBeq:cutfill-base-inclusion-boundary-new} with
\eqref{DBeq:cutfill-layer-boundary-new} gives
\begin{equation}\label{DBeq:cutfill-base-symdiff-new}
        \mathcal H^2(B_s^\pm\triangle D_s)\le Cs .
\end{equation}
Outside the strip \(|d(s,\cdot)|<Ms\), the set
\(\widetilde G_s^+\cup\widetilde G_s^-\) agrees with \(G(s)\) up to null sets. Inside
the strip, Claim~2 and \eqref{DBeq:cutfill-base-inclusion-boundary-new} imply,
again up to null sets,
\[
        G(s)\triangle(\widetilde G_s^+\cup\widetilde G_s^-)
        \subset
        \{\DBp+r\DBn_\Sigma(s,\DBp):
        \DBp\in D_s^{+\omega s}\setminus D_s^{-\omega s},\ |r|\le Ms\} .
\]
Since the signed-distance parametrization has uniformly bounded Jacobian on
\(\mathcal N_{K,r_*}(s)\), \eqref{DBeq:cutfill-layer-boundary-new} gives
\begin{equation}\label{DBeq:cutfill-symdiff-boundary-new}
        |G(s)\triangle(\widetilde G_s^+\cup\widetilde G_s^-)|
        \le Cs\,\mathcal H^2(D_s^{+\omega s}\setminus D_s^{-\omega s})\le Cs^2 .
\end{equation}
Consequently, since $\phi$ is bounded,
\begin{equation}\label{DBeq:cutfill-symdiff-integral-boundary-new}
        \int_{G(h)}\phi(h,\DBx)d\DBx
        -\sum_\pm\int_{\widetilde G_h^\pm}\phi^\pm(h,\DBx)d\DBx
        =O(h^2).
\end{equation}

\vskip2mm
\emph{Step 6: application of the boundary form of the RTT to the modified sets.}
\vskip1mm
Fix $0<\varepsilon<h<T_0$. 
Here the sets are used with the boundary convention from Step~4 only in order
to apply the boundary-form RTT. This does not change any volume integral, since
only Lipschitz faces of zero three-dimensional measure are added. By Claim~4,
the swept tubes of $\widetilde G_s^\pm$ over $(\varepsilon,h)$ are compact
Lipschitz domains with the natural time-slice/lateral boundary decomposition,
up to null sets. Moreover, on the lateral pieces one has $\DBn_M^\DBx\ne0$ except
possibly on edge sets of codimension two, which do not contribute to the
surface integrals. Hence Theorem~\ref{thm:single-phase-boundary-rtt} applies to
$\widetilde G_s^\pm$ on every interval $(\varepsilon,h)$.

The common cut face \(P_s^\pm\) is an interior face of
\(\widetilde G_s^\pm=A_s^\pm\cup C_s^\pm\), except along edges. Thus the
lateral boundary relevant for the RTT consists, up to null sets, of the
material part \(\Lambda_s^\pm\), the interface base \(B_s^\pm\), and the normal
lateral face
\[
        L_s^\pm
        :=\{\DBp\pm r\DBn_\Sigma(s,\DBp):
        \DBp\in\partial_{\iface(s)}B_s^\pm,
        \ 0\le r\le Ms\}.
\]
On \(\Lambda_s^\pm\) the normal velocity is
\(\DBv^\pm\cdot\DBn_{A_s^\pm}\). On \(B_s^\pm\) it is \(\mp V_\Sigma\), because
the outer normal of \(\widetilde G_s^\pm\) on the interface is
\(\mp\DBn_\Sigma\). Let \(V_{L,s}^\pm\) denote the normal velocity of \(L_s^\pm\),
with the sign convention of Theorem~\ref{thm:single-phase-boundary-rtt}.

The local parametrizations from Claim~4 give
\[
        \mathcal H^2(L_s^\pm)
        \le Cs\,\mathcal H^1(\partial_{\iface(s)}B_s^\pm)
        \le Cs,
        \qquad |C_s^\pm|=O(s),
\]
{and Claim~4 gives $|V_{L,s}^\pm|\le C$ almost everywhere on the
corresponding lateral space-time graphs.} Hence
\[
        R_\pm(s):=
        \int_{C_s^\pm}\partial_t\phi^\pm\,d\DBx
        +\int_{L_s^\pm}\phi^\pm V_{L,s}^\pm\,dS
\]
is measurable on \((\varepsilon,h)\) and satisfies, independently of
\(\varepsilon\),
\[
        |R_\pm(s)|\le Cs
\]
for a.e. \(s\in(\varepsilon,h)\). Applying the boundary-form RTT to
\(\widetilde G_s^\pm\) gives, for a.e. \(s\in(\varepsilon,h)\),
\begin{equation}\label{DBeq:cutfill-pm-boundary-new}
\frac{d}{ds}\int_{\widetilde G_s^\pm}\phi^\pm d\DBx
=\int_{A_s^\pm}\partial_t\phi^\pm d\DBx
 +\int_{\Lambda_s^\pm}\phi^\pm\DBv^\pm\cdot\DBn_{A_s^\pm}dS
 \mp\int_{B_s^\pm}\phi^\pm V_\Sigma dS
 +R_\pm(s).
\end{equation}

Integrate \eqref{DBeq:cutfill-pm-boundary-new} for both signs from
\(\varepsilon\) to \(h\) and then let \(\varepsilon\downarrow0\). The initial terms converge to
$\int_{G_0\cap\Omega_0^\pm}\phi^\pm(0,\DBx)d\DBx$, while the fill volumes vanish as
$s\downarrow0$. Moreover,
\[
        \int_0^h |R_+(s)|+|R_-(s)|\,ds\le Ch^2.
\]
Using \eqref{DBeq:cutfill-symdiff-integral-boundary-new}, and replacing
$A_s^\pm$ by $G(s)\cap\Omega^\pm(s)$ in the volume integral at an integrated error
$O(h^2)$, this gives
\begin{align}\label{DBeq:cutfill-main-integrated-boundary-new}
&\int_{G(h)}\phi(h,\DBx)d\DBx-
  \int_{G_0}\phi(0,\DBx)d\DBx =
\int_0^h\int_{G(s)\setminus\iface(s)}\partial_t\phi\,d\DBx ds \nonumber\\
&\qquad+
\int_0^h\left(
\int_{\Lambda_s^+}\phi^+\DBv^+\cdot\DBn_{A_s^+}dS
+
\int_{\Lambda_s^-}\phi^-\DBv^-\cdot\DBn_{A_s^-}dS
\right)ds \\
&\qquad-
\int_0^h\int_{B_s^+}\phi^+V_\Sigma dS\,ds
+
\int_0^h\int_{B_s^-}\phi^-V_\Sigma dS\,ds
+O(h^2). \nonumber
\end{align}
The replacement of $A_s^\pm$ by $G(s)\cap\Omega^\pm(s)$ is legitimate because the
symmetric difference is contained in an interfacial strip of thickness $O(s)$ over a
uniformly bounded surface set, and hence has volume $O(s)$. Since
$\partial_t\phi^\pm$ is locally bounded, the resulting integrated error is $O(h^2)$.

\vskip2mm
\emph{Step 7: passage to the initial time for the right derivative.}
\vskip1mm
First, \eqref{DBeq:cutfill-outer-reach-boundary-new} gives
$G(s)\subset G_0+\overline B_{Ms}(0)$. Conversely, let
$\DBx\in G_0$ satisfy $\operatorname{dist}(\DBx,G_0^\complement)>2Ms$. Choose a local
solution branch of the differential inclusion which has value $\DBx$ at time $s$ and is
defined backwards to time $0$. Its time-zero value lies within distance at most
$ms<Ms$ from $\DBx$ and therefore belongs to $G_0$. Viewed on $[0,s]$, this branch is
an admissible solution from $G_0$ to $\DBx$, and hence $\DBx\in G(s)$. Thus
$G(s)\triangle G_0$ is contained in an $O(s)$-neighborhood of $\partial G_0$. Since
$\iface(s)$ converges to $\iface_0$ in the signed-distance coordinates on $\mathcal N_{K,r_*}$, it follows that
\[
        \chi_{G(s)\cap\Omega^\pm(s)}\to
        \chi_{G_0\cap\Omega_0^\pm}
        \qquad\hbox{a.e. in }K
        \quad\mbox{as }s\to0+,
\]
the exceptional set being contained in $\partial G_0\cup\iface_0$. Dominated convergence
therefore yields
\begin{equation}\label{DBeq:cutfill-volume-limit-new}
\int_{G(s)\setminus\iface(s)}\partial_t\phi\,d\DBx
\to
\int_{G_0\setminus\iface_0}\partial_t\phi(0,\DBx)d\DBx
\quad\mbox{as }s\to0+ .
\end{equation}

Second, consider the material boundary terms. By Claim~3, together with the
fixed-domain boundary description used in Claim~4, $\Lambda_s^\pm=X_s^\pm(S_s^\pm)$,
where $S_s^\pm$ is the part of $\partial G_0\cap\Omega_0^\pm$ which is not cut off by
$f_s^\pm=\pm Ms$. The removed part is contained in an $O(s)$-neighborhood of
$\Gamma=\partial G_0\cap\iface_0$ on the $C^2$ surface $\partial G_0$, hence has
surface area $O(s)$.

We orient $S_s^\pm\subset\partial G_0$ by the outer normal $\DBn$ of $G_0$.
For small $s$, the restriction $X_s^\pm|_{S_s^\pm}$ is bi-Lipschitz and transports this
orientation to the material face $\Lambda_s^\pm$. This transported orientation agrees
with the measure-theoretic outer orientation of $A_s^\pm$ on $\Lambda_s^\pm$, except on
edge sets of zero surface measure, since $X_s^\pm-\Id$ has Lipschitz constant $O(s)$.
Let $\DBn_s^\pm$ denote the corresponding transported outer unit normal, pulled back to
$S_s^\pm$, and let $J_s^\pm$ denote the tangential surface Jacobian of $X_s^\pm$ on
$S_s^\pm$, in the sense of Proposition~\ref{prop:surface-area-formula-lipschitz}.
The area formula for Lipschitz maps gives
\begin{equation*}
\int_{\Lambda_s^\pm}\phi^\pm\DBv^\pm\cdot\DBn_{A_s^\pm}\,dS
=
\int_{S_s^\pm}
\phi^\pm(s,X_s^\pm(\xi))\,
\DBv^\pm(s,X_s^\pm(\xi))\cdot\DBn_s^\pm(\xi)
J_s^\pm(\xi)\,dS_\xi .
\end{equation*}
Since $X_s^\pm-\Id$ has Lipschitz constant $O(s)$, the estimates in
Proposition~\ref{prop:surface-area-formula-lipschitz} yield $J_s^\pm=1+O(s)$ and
$\DBn_s^\pm=\DBn+O(s)$ a.e. The uniform convergence $X_s^\pm\to\Id$ and the one-sided
continuity of $\phi^\pm$ and $\DBv^\pm$ now imply
\begin{equation}\label{DBeq:cutfill-material-boundary-limit-new}
\int_{\Lambda_s^\pm}\phi^\pm \DBv^\pm \cdot\DBn_{A_s^\pm}\,dS
\to
\int_{\partial G_0\cap\Omega_0^\pm}\phi^\pm \DBv^\pm \cdot\DBn\,dS
\quad\mbox{as }s\to0+ .
\end{equation}
All quantities on the right are evaluated at time $0$.

Third, \eqref{DBeq:cutfill-base-symdiff-new} and the continuity on
$\gr(\iface)$ of the one-sided traces
$(s,\DBp)\mapsto\phi^\pm(s,\DBp)V_\Sigma(s,\DBp)$ yield
\[
        \int_{B_s^\pm}\phi^\pm V_\Sigma dS
        -\int_{D_s}\phi^\pm V_\Sigma dS\to0
        \quad\mbox{as }s\to0+ .
\]
The surface change of variables under the $C^1$ diffeomorphism
$\Psi_s:D_0\to D_s$ has Jacobian converging uniformly to $1$, and the pulled-back
integrands converge uniformly on $D_0$. Hence
\begin{equation}\label{DBeq:cutfill-interface-limit-new}
\int_{B_s^+}\phi^+V_\Sigma dS \to
\int_{D_0}\phi^+V_\Sigma dS,
\qquad
\int_{B_s^-}\phi^-V_\Sigma dS \to
\int_{D_0}\phi^-V_\Sigma dS
\quad\mbox{as }s\to0+ .
\end{equation}

Consequently, after division by \(h\), the time averages in
\eqref{DBeq:cutfill-main-integrated-boundary-new} converge to the corresponding
initial-time integrals.

Divide \eqref{DBeq:cutfill-main-integrated-boundary-new} by $h$ and let
$h\downarrow0$. The limits
\eqref{DBeq:cutfill-volume-limit-new}--\eqref{DBeq:cutfill-interface-limit-new}
show, since $D_0=G_0\cap\iface_0$, that the right derivative equals, after translating back time, the right-hand side in
\eqref{DBeq:boundary-transport-theorem-comoving}.

\vskip2mm
\emph{Left derivative and conclusion.}
\vskip1mm
Apply the right-derivative formula just proved to the time-reversed data
\[
        \widetilde{\iface}(s):=\iface(-s),\qquad
        \widetilde{\phi}(s,\DBx):=\phi(-s,\DBx),\qquad
        \widetilde{\DBv}^{\,\pm}(s,\DBx):=-\DBv^\pm(-s,\DBx).
\]
For the reversed problem, the co-moving set at time $s>0$ is $G(-s)$, while
$\partial_t\phi$, $\DBv$, and $V_\Sigma$ all change sign. Hence the negative of
the left derivative equals the negative of the right-hand side in
\eqref{DBeq:boundary-transport-theorem-comoving}. The two one-sided derivatives
therefore coincide. Translating time back gives
\eqref{DBeq:boundary-transport-theorem-comoving} at the original time $t_0$, with
all integrands on the right-hand side evaluated at $t_0$.
\end{proof}

\begin{corollary}[Divergence form]\label{DBcor:2ph-transport-theorem-divergence-form}
In the situation of Theorem~\ref{DBthm:2ph-transport-theorem-comovingCV}, assume in addition that, at the time \(t_0\),
the phasewise vector fields \(\phi^\pm\DBv^\pm\) belong to
\[
  W^{1,1}\bigl( G_0^\circ \cap\Omega^\pm(t_0);\R^n\bigr).
\]
Let \(\DBv^\Sigma\) be an
interfacial velocity field satisfying
$\DBv^\Sigma\cdot\DBn_\Sigma=V_\Sigma$ on $\Sigma(t_0)$. Then
\begin{equation}\label{DBeq:two-phase-divergence-form}
  \left[
  \frac{d}{dt}\int_{G(t)}\phi\,d\DBx
  \right]_{t=t_0}
  =
  \int_{G_0\setminus\Sigma_0}
  \bigl(\partial_t\phi+\operatorname{div}(\phi\DBv)\bigr)\,d\DBx
  +
  \int_{\Sigma_0\cap G_0}
  [\![\phi(\DBv-\DBv^\Sigma)\cdot\DBn_\Sigma]\!]\,dS,
\end{equation}
where the integrands on the right-hand side are evaluated at $t=t_0$.
\end{corollary}

\begin{proof}
By the transversality assumption, $G_0^\circ \cap\Omega^\pm(t_0)$
are bounded Lipschitz domains. Applying the Sobolev Gauss--Green formula
phasewise gives
\[
  \sum_\pm
  \int_{\partial G_0\cap\Omega^\pm_0}
  \phi^\pm\DBv^\pm\cdot\DBn\,dS
  =
  \int_{G_0\setminus\Sigma_0}
  \operatorname{div}(\phi\DBv)\,d\DBx  
-\sum_\pm
  \int_{G_0\cap\Sigma_0}
  \phi^\pm \DBv^\pm\cdot\DBn^\pm\,dS.
\]
Substituting this identity into the boundary form (\ref{DBeq:boundary-transport-theorem-comoving}) and using
$V_\Sigma=\DBv^\Sigma\cdot\DBn_\Sigma$ yields
\[
\begin{aligned}
  \left[
  \frac{d}{dt}\int_{G(t)}\phi\,d\DBx
  \right]_{t=t_0}
  &=
  \int_{G_0\setminus\Sigma_0}
  \bigl(\partial_t\phi+\operatorname{div}(\phi\DBv)\bigr)\,d\DBx  \\
  &\quad -
  \int_{G_0\cap\Sigma_0}
  \Bigl(
    \phi^+\DBv^+\cdot\DBn^+
    +
    \phi^-\DBv^-\cdot\DBn^-
    +
    [\![\phi]\!]\DBv^\Sigma\cdot\DBn_\Sigma
  \Bigr)\,dS  \\
  &=
  \int_{G_0\setminus\Sigma_0}
  \bigl(\partial_t\phi+\operatorname{div}(\phi\DBv)\bigr)\,d\DBx
  +
  \int_{\Sigma_0\cap G_0}
  [\![\phi(\DBv-\DBv^\Sigma)\cdot\DBn_\Sigma]\!]\,dS,
\end{aligned}
\]
where all integrands are evaluated at $t=t_0$.
This proves the assertion.
\end{proof}

We next give a simple application of Corollary~\ref{DBcor:2ph-transport-theorem-divergence-form}.
In general, volume is not conserved within fluid flows, but volume changes are often negligible in flows of liquid or for flows at low Mach numbers, i.e.\ where the velocity is small compared with the speed of sound.
Recall that flows which conserve the volume content of co-moving volumes are called \textit{isochoric}.
In the present attainable-set setting we use the corresponding instantaneous notion: the flow is called isochoric if, at every reference time $t_0$, every admissible co-moving volume emanating from an admissible $G_0$ has vanishing volume derivative at $t=t_0$. Let $G(t)=\hat\Phi^t_{t_0}(G_0)$ be such a co-moving volume. If the additional assumptions of Corollary~\ref{DBcor:2ph-transport-theorem-divergence-form} are fulfilled for $\phi\equiv1$, then
\[
\frac{d}{dt}\int_{G(t)}1\,d\DBx
=\int_{G(t)\setminus\iface(t)}\DBdiv\DBv\,d\DBx
+\int_{\iface(t)\cap G(t)}[\![(\DBv-\DBv^\iface)\cdot\DBn_\Sigma]\!]\,dS
\quad\mbox{at }t=t_0.
\]
Hence, using the localization principle as above, a two-phase flow is isochoric if and only if
\begin{equation}\label{DBeq:isochoric-volume-transmission}
\DBdiv\DBv=0 \; \mbox{ a.e. in } \Omega\setminus\iface
\quad\mbox{and}\quad
[\![(\DBv-\DBv^\iface)\cdot\DBn_\Sigma]\!]=0 \; \mbox{ on }\iface.
\end{equation}
Indeed, the local conditions imply a vanishing instantaneous volume rate for every admissible co-moving set at an arbitrary reference time, directly by the divergence form. Conversely, one applies the identity at an arbitrary reference time to admissible initial volumes shrinking to bulk Lebesgue points, and to transversally cutting volumes shrinking to arbitrary interfacial points, to obtain the two local conditions in (\ref{DBeq:isochoric-volume-transmission}).
While (\ref{DBeq:isochoric-volume-transmission})$_1$ requires the velocity field to be solenoidal, the jump condition (\ref{DBeq:isochoric-volume-transmission})$_2$ is equivalent to $[\![\DBv\cdot\DBn_\Sigma]\!]=0$, i.e.\ to continuity of the normal velocity at $\iface$. To understand how the latter property is related to phase change, we expand the term inside the jump bracket in (\ref{DBeq:isochoric-volume-transmission})$_2$ by $\rho/\rho$. Then, using equality of the one-sided mass transfer fluxes due to (\ref{E3}), we can rewrite (\ref{DBeq:isochoric-volume-transmission})$_2$ as
\begin{equation}\label{DBeq:isochoric-volume-transmission2}
\dot m\,[\![\rho^{-1}]\!]=0 \; \mbox{ on }\iface.
\end{equation}
{
Consequently, isochoricity in the above sense requires the flow fields
to be solenoidal in both phases and, at every point of $\iface$, either
$\dot m=0$ or $[\![\rho]\!]=0$. Equivalently, phase change is excluded on
every subset of the interface on which the two mass densities differ. Continuity
of the mass density across a genuine two-phase interface is very unlikely in
real-world fluid systems.
}
\vskip3mm
\noindent
We also record the following variant of the divergence form with built-in mass conservation.
\begin{corollary}\label{DBcor:2ph-transport-theorem-comovingCV-built-in-mass}
In the situation of Theorem~\ref{DBthm:2ph-transport-theorem-comovingCV} and Corollary~\ref{DBcor:2ph-transport-theorem-divergence-form}, assume that $\rho\psi$ satisfies the regularity and phasewise divergence assumptions imposed there. Assume moreover that $\rho$, $\psi$ and $\DBv$ have enough phasewise regularity for the product rule below to hold in $L^1$ on the two bulk parts; for instance, this is the case under the corresponding phasewise $C^1$ assumptions. Assume also that $\rho$ satisfies the bulk continuity equation
\[
        \partial_t\rho+\DBdiv(\rho\DBv)=0
        \qquad\mbox{in }\Omega\setminus\iface
\]
and the interfacial mass balance
\[
        [\![\rho(\DBv-\DBv^\iface)\cdot\DBn_\Sigma]\!]=0
        \qquad\mbox{on }\iface.
\]
With $\dot m:=\rho^+(\DBv^+-\DBv^\iface)\cdot\DBn_\Sigma
        =\rho^-(\DBv^- -\DBv^\iface)\cdot\DBn_\Sigma$ on $\iface$, and with the phasewise material derivative defined by
\[
        \frac{D\psi}{Dt}:=\partial_t\psi^\pm+\DBv^\pm\cdot\nabla\psi^\pm
        \qquad\mbox{in }\Omega^\pm,
\]
it holds that
\begin{equation}\label{DBeq:bulk-transport-theorem-comoving-built-in-mass}
\Big[ \frac{d}{dt}  \int_{G(t)}\hspace{-3pt}\rho\psi\,d\DBx \Big]_{|t=t_0} \hspace{-4pt}
=\int_{G_0 \setminus \iface(t_0)} \hspace{-3pt} \rho \frac{D\psi}{Dt}\,d\DBx \,
+\int_{\iface(t_0)\cap G_0} \hspace{-3pt}\dot m[\![\psi]\!]\,dS.
\end{equation}
For vector-valued functions, the formula applies componentwise.
\end{corollary}
\begin{proof}
Apply Corollary~\ref{DBcor:2ph-transport-theorem-divergence-form} to the quantity
$\phi=\rho\psi$. In the bulk phases,
\[
\partial_t(\rho\psi)+\DBdiv(\rho\psi\DBv)
=\psi\bigl(\partial_t\rho+\DBdiv(\rho\DBv)\bigr)
 +\rho\bigl(\partial_t\psi+\DBv\cdot\nabla\psi\bigr),
\]
so the first term on the right vanishes by the bulk continuity equation and the remaining
term is $\rho D\psi/Dt$. On the interface, the jump term from
Corollary~\ref{DBcor:2ph-transport-theorem-divergence-form} becomes
\[
[\![\rho\psi(\DBv-\DBv^\iface)\cdot\DBn_\Sigma]\!] =\dot m[\![\psi]\!],
\]
where the interfacial mass balance has been used. This proves
\eqref{DBeq:bulk-transport-theorem-comoving-built-in-mass}.
\end{proof}

This version of the transport theorem for co-moving volume in two-phase flows is tailor-made for an efficient formulation of the momentum balance.
Indeed, if the integral momentum balance (\ref{DBeq:momentum-balance-material}) is imposed at the reference time on admissible attainable-set co-moving volumes, then Corollary~\ref{DBcor:2ph-transport-theorem-comovingCV-built-in-mass}, applied with $\psi=\DBv$, yields the local momentum balance given in (\ref{E2}), (\ref{E4}).
\appendix

\section{Basic facts on Lipschitz domains}\label{app:Lipschitz-domains}

Lipschitz domains form a useful class of domains, for instance in the theory
of Sobolev spaces, trace operators, extension results and divergence theorems.
Below we collect some elementary facts about such domains which are used in
the main text. For those properties for which we do not have a specific
reference to the literature, we include detailed proofs. Much more information
on Lipschitz domains and Sobolev spaces on nonsmooth domains can be found, for
example, in Grisvard~\cite{Grisvard}, McLean~\cite{McLean}, and
Leoni~\cite{leoni2017first}. For the Sobolev Gauss--Green formula on bounded
Lipschitz domains used in the main text, see Section~1.5.3 of
Grisvard~\cite{Grisvard} and Chapter~3, Section~1, Theorem~1.1 of
Ne\v{c}as~\cite{Necas1967}.

It is common to reserve the term Lipschitz domain for open sets. We follow this
convention: unless explicitly stated otherwise, a bounded Lipschitz domain is a
bounded open set in the usual sense. In addition, we use the following compact
version. A compact set \(K\subset\R^n\) is called a compact Lipschitz domain if
\[
  K=\overline{\operatorname{int}K}
\]
and \(\operatorname{int}K\) is a bounded Lipschitz domain\footnote{$\operatorname{int}K$, the interior of $K$, is also abbreviated as $K^\circ$ in the main text.}. 
Whenever Sobolev
spaces, traces, normal fields, or divergence theorems are used for such a set
\(K\), they are understood on the open Lipschitz domain
\(\operatorname{int}K\). Its boundary is
\[
  \partial K=\partial(\operatorname{int}K).
\]
Moreover, since the boundary of a Lipschitz domain has Lebesgue measure zero,
we do not distinguish between \(K\) and \(\operatorname{int}K\) in volume
integrals.

We now start with the standard definition of a bounded Lipschitz domain.
\begin{definition}[bounded Lipschitz domain]
We call a bounded open \(U\subset\R^n\) a bounded Lipschitz domain if,
for every \(x_0\in\partial U\), there exist \(r,h>0\), an orthogonal map
\(Q\in\R^{n\times n}\), and a Lipschitz function
\(\varphi:B'_r(0)\to\R\) such that, with
\[
  y=Q(x-x_0)=(y',y_n)\in\R^{n-1}\times\R,
\]
one has
\[
  Q(U-x_0)\cap \big(B'_r(0)\times(-h,h)\big)
  =
  \{(y',y_n)\in B'_r(0)\times(-h,h):\ y_n>\varphi(y')\}.
\]
\end{definition}

The following uniform two-sided cone condition is sometimes easier to verify.
\begin{definition}[uniform two-sided cone condition]\label{DBdef:uniform-cone-condition}
Let \(\theta\in(0,\pi/2)\) and \(h>0\).
For a unit vector \(\nu\in \R^n\) let
\[
  \Gamma_{\theta,h}(\nu)
  :=
  \{\DBz\in\R^n:0<|\DBz|<h,\ \DBz\cdot \nu>|\DBz|\cos\theta\}.
\]
A bounded open set \(U\subset\R^n\) is said to satisfy a \emph{uniform
two-sided cone condition} if there exist \(\theta\in(0,\pi/2)\) and \(h>0\)
such that for every \(\DBx\in\partial U\) there is a unit vector
\(\nu_{\DBx}\) with
\[
  \DBx+\Gamma_{\theta,h}(\nu_{\DBx})\subset U,
  \qquad
  \DBx-\Gamma_{\theta,h}(\nu_{\DBx})\subset \R^n\setminus \overline U.
\]
\end{definition}

\begin{lemma}\label{lem:cone-characterization-lipschitz}
Let \(U\subset\R^n\) be bounded and open. Then \(U\) is a bounded Lipschitz
domain if and only if \(U\) satisfies the uniform two-sided cone condition.
\end{lemma}

For a proof see, e.g., \cite[Chapter~1]{Grisvard} or
\cite[Section~3.1]{McLean}. There one also finds the following basic fact.

\begin{lemma}[\(C^1\)-invariance of bounded Lipschitz domains]\label{DBthm:C1-invariance-bounded-Lipschitz}
Let \(U\subset\R^n\) be a bounded Lipschitz domain.
Let \(V,W\subset\R^n\) be open neighborhoods of \(\overline U\) and
\(\overline{\Psi(U)}\), respectively, and let
\[
  \Psi:V\to W
\]
be a \(C^1\)-diffeomorphism. Then \(\Psi(U)\) is again a bounded Lipschitz
domain.
\end{lemma}

We also exploit the following basic result.
\begin{proposition}\label{prop:Cartesian-product-Lipschitz-domain}
Let \(U\subset\R^n\) be a bounded Lipschitz domain and let
\(I=(a,b)\subset\R\) be a bounded open interval. Then \(I\times U\) is a
bounded Lipschitz domain in \(\R^{n+1}\).
Moreover, for fixed \(U\), the opening angle in the uniform two-sided cone
condition for \(I\times U\) can be chosen independently of the length of
\(I\). Only the corresponding cone height may depend on \(|I|\).
\end{proposition}

\begin{proof}
Set \(D:=I\times U\). It is clear that \(D\) is bounded and open. Let
\((t_0,x_0)\in\partial D\).

If \(t_0\in(a,b)\) and \(x_0\in\partial U\), use the local Lipschitz graph
representation of \(U\) at \(x_0\). Thus, after a translation and a rigid motion
in the \(x\)-variables, one has locally
\[
  U=\{(x',x_n): x_n>\varphi(x')\}
\]
with \(\varphi\) Lipschitz. Hence, in the variables \((t,x',x_n)\), locally
\[
  D=\{(t,x',x_n): x_n>\varphi(x')\},
\]
which is a Lipschitz graph domain in \(\R^{n+1}\).

If \(t_0\in\{a,b\}\) and \(x_0\in U\), choose \(\rho>0\) such that
\(\overline{B_\rho(x_0)}\subset U\). Then, near \((t_0,x_0)\), the set \(D\)
is just a half-space:
\[
  D=\{(t,x): t>a\}
  \quad\text{or}\quad
  D=\{(t,x): t<b\}.
\]

It remains to consider the corner case \(t_0\in\{a,b\}\) and
\(x_0\in\partial U\). By symmetry, it suffices to treat \(t_0=a\). After
translation in \(t\) and a rigid motion in the \(x\)-variables, we may assume
\(a=0\), \(x_0=0\), and
\[
  U=\{(x',x_n): x_n>\varphi(x')\}
\]
locally near \(0\), with \(\varphi\) Lipschitz. Then, locally near \((0,0)\),
\[
  D=\{(t,x',x_n): t>0,\ x_n>\varphi(x')\}.
\]
Now introduce the linear change of variables
\[
  y'=x',
  \qquad
  y_n=\frac{x_n-t}{\sqrt2},
  \qquad
  y_{n+1}=\frac{x_n+t}{\sqrt2}.
\]
Since
\[
  t=\frac{y_{n+1}-y_n}{\sqrt2},
  \qquad
  x_n=\frac{y_{n+1}+y_n}{\sqrt2},
\]
the inequalities \(t>0\) and \(x_n>\varphi(x')\) become
\[
  y_{n+1}>y_n,
  \qquad
  y_{n+1}>-y_n+\sqrt2\,\varphi(y').
\]
Thus, locally,
\[
  D=\{(y',y_n,y_{n+1}): y_{n+1}>\psi(y',y_n)\},
\]
where
\[
  \psi(y',y_n):=\max\{y_n,\,-y_n+\sqrt2\,\varphi(y')\}.
\]
Since the maximum of two Lipschitz functions is Lipschitz, this again is a
Lipschitz graph domain.

Thus every boundary point of \(I\times U\) admits a local Lipschitz graph
representation, hence \(I\times U\) is a bounded Lipschitz domain.

The same construction gives the asserted independence of the interval length
for the cone opening angle. Indeed, take a finite Lipschitz graph atlas for
\(\partial U\), with all graph Lipschitz constants bounded by a number \(L_U\).
At boundary points with \(t_0\in(a,b)\) the graph constants are those of
\(U\). At boundary points with \(t_0\in\{a,b\}\) and \(x_0\in U\), the local
model is a half-space and the graph constant is zero. In the corner charts the
function
\[
  \psi(y',y_n):=\max\{y_n,\,-y_n+\sqrt2\,\varphi(y')\}
\]
has a Lipschitz constant bounded in terms of \(L_U\) only. Thus the product
charts have Lipschitz constants bounded by a constant depending only on the
chosen atlas of \(U\), and not on \(a\), \(b\), or \(|I|\). The equivalence
between Lipschitz graph representations and the uniform two-sided cone
condition therefore gives a cone opening angle depending only on \(U\), not on
\(|I|\). The cone height, however, may have to be chosen smaller than a fixed
fraction of \(|I|\) near the time faces.
\end{proof}

Consequently, if \(K\subset\R^n\) is a compact Lipschitz domain and
\(I=(a,b)\subset\R\) is a bounded open interval, then
\[
  \overline{I\times\operatorname{int}K}=[a,b]\times K
\]
is a compact Lipschitz domain in \(\R^{n+1}\). Moreover, the
small-perturbation threshold obtained from
Theorem~\ref{DBthm:small-biLip-perturbation-cone} and
Corollary~\ref{cor:closed-lipschitz-small-perturbation} for the compact
products \([a,b]\times K\) can be chosen depending only on \(K\), and not on
\(b-a\).

We now consider small bi-Lipschitz perturbations of the identity. Let us note
that bi-Lipschitz images of Lipschitz domains are often called weakly
Lipschitz domains and need not be Lipschitz domains in the graph sense. Since
flow maps for Lipschitz velocity fields are bi-Lipschitz maps, co-moving sets
that emanate from a Lipschitz domain retain the weak Lipschitz property, while
they might stop to be Lipschitz domains. The following result is hence
advantageous.
\begin{theorem}[stability under small Lipschitz perturbations of the identity]\label{DBthm:small-biLip-perturbation-cone}
Let \(U\subset\R^n\) be a bounded Lipschitz domain. Then there exists a constant
\(L_U\in(0,1)\) such that for every Lipschitz map
\[
  F:U\to\R^n
  \qquad\text{with}\qquad
  \Lip(F)<L_U,
\]
the map
\[
  T:=\Id+F
\]
is a bi-Lipschitz homeomorphism of \(U\) onto \(T(U)\), and \(T(U)\) is a
bounded Lipschitz domain.

More precisely, if \(U\) satisfies the uniform two-sided cone condition with
parameters \((\theta_0,h_0)\) and \(\R^n\) is equipped with the Euclidean norm,
then one may choose
\[
  L_U:=\frac{1-\cos\theta_0}{2}.
\]
\end{theorem}

\begin{proof}
Notice first that we may assume \(\R^n\) to be equipped with the Euclidean
norm. We identify \(F\) with its unique continuous extension to \(\overline U\)
and, by Kirszbraun's extension theorem, with an extension to all of \(\R^n\)
with the same Lipschitz constant.

By Lemma~\ref{lem:cone-characterization-lipschitz}, there exist
\(\theta_0\in(0,\pi/2)\) and \(h_0>0\) such that \(U\) satisfies the uniform
two-sided cone condition with parameters \((\theta_0,h_0)\). Let
\[
  \alpha:=\Lip(F)<\frac{1-\cos\theta_0}{2}.
\]

We first show that \(T\) is a bi-Lipschitz homeomorphism of \(\R^n\). For
\(x,y\in\R^n\),
\[
  |T(x)-T(y)|
  \le |x-y|+|F(x)-F(y)|
  \le (1+\alpha)|x-y|,
\]
and
\[
  |T(x)-T(y)|
  \ge |x-y|-|F(x)-F(y)|
  \ge (1-\alpha)|x-y|.
\]
Hence \(T\) is injective and Lipschitz. For fixed \(y\in\R^n\), the map
\[
  S_y(x):=y-F(x)
\]
is a contraction with constant \(\alpha<1\). By Banach's fixed point theorem,
there is a unique \(x\in\R^n\) such that \(S_y(x)=x\), i.e. \(T(x)=y\). Thus
\(T\) is bijective, and
\[
  \Lip(T^{-1})\le \frac{1}{1-\alpha}.
\]
Moreover, writing
\[
  T^{-1}=\Id+H,
  \qquad
  H(y):=-F(T^{-1}(y)),
\]
this gives
\[
  \Lip(H)\le \frac{\alpha}{1-\alpha}=:\beta.
\]
Since
\[
  \alpha<\frac{1-\cos\theta_0}{2},
\]
we have
\[
  \beta=\frac{\alpha}{1-\alpha}
  <
  \frac{1-\cos\theta_0}{1+\cos\theta_0}.
\]

Set
\[
  h_1:=\frac{h_0}{1+\beta},
  \qquad
  \cos\theta_1:=\beta+(1+\beta)\cos\theta_0 .
\]
The above bound on \(\beta\) implies \(\cos\theta_1<1\), hence
\(\theta_1\in(0,\theta_0)\).

Since \(T\) is a homeomorphism of \(\R^n\), one has
\[
  \partial T(U)=T(\partial U).
\]
It is therefore enough to verify the cone condition at points \(y=T(x)\) with
\(x\in\partial U\). Fix \(x\in\partial U\), and let \(\nu_x\) be a unit vector
such that
\[
  x+\Gamma_{\theta_0,h_0}(\nu_x)\subset U,
  \qquad
  x-\Gamma_{\theta_0,h_0}(\nu_x)\subset \R^n\setminus \overline U.
\]
Set \(y:=T(x)\).

We claim that
\[
  y+\Gamma_{\theta_1,h_1}(\nu_x)\subset T(U),
  \qquad
  y-\Gamma_{\theta_1,h_1}(\nu_x)\subset \R^n\setminus \overline{T(U)}.
\]

Take \(z\in\Gamma_{\theta_1,h_1}(\nu_x)\), and define
\[
  w:=T^{-1}(y+z)-x.
\]
Since \(T^{-1}=\Id+H\) and \(\Lip(H)\le\beta\),
\[
  |w-z|=|H(y+z)-H(y)|\le\beta |z|.
\]
Therefore
\[
  |w|\le |z|+|w-z|\le (1+\beta)|z|<h_0,
\]
and
\[
  w\cdot\nu_x
  \ge z\cdot\nu_x-|w-z|
  > |z|\cos\theta_1-\beta |z|
  = (1+\beta)|z|\cos\theta_0
  \ge |w|\cos\theta_0.
\]
Hence \(w\in\Gamma_{\theta_0,h_0}(\nu_x)\), so \(x+w\in U\), and therefore
\[
  y+z=T(x+w)\in T(U).
\]

For the exterior cone inclusion, take again
\(z\in\Gamma_{\theta_1,h_1}(\nu_x)\), and set
\[
  w:=T^{-1}(y-z)-x.
\]
Then
\[
  |w+z|\le\beta |z|,
  \qquad
  |w|\le(1+\beta)|z|<h_0,
\]
and
\[
  (-w)\cdot\nu_x
  \ge z\cdot\nu_x-|w+z|
  > |z|\cos\theta_1-\beta |z|
  = (1+\beta)|z|\cos\theta_0
  \ge |w|\cos\theta_0.
\]
Thus \(-w\in\Gamma_{\theta_0,h_0}(\nu_x)\), hence
\(x+w\in\R^n\setminus\overline U\). Since \(T\) is a homeomorphism of
\(\R^n\), it maps \(\overline U\) onto \(\overline{T(U)}\). Therefore
\[
  y-z=T(x+w)\in \R^n\setminus \overline{T(U)}.
\]

Therefore \(T(U)\) satisfies the uniform two-sided cone condition with
parameters \((\theta_1,h_1)\). By
Lemma~\ref{lem:cone-characterization-lipschitz}, \(T(U)\) is a bounded
Lipschitz domain.
\end{proof}

\begin{corollary}[Compact Lipschitz domains under small perturbations]\label{cor:closed-lipschitz-small-perturbation}
Let \(K\subset\R^n\) be a compact Lipschitz domain. Then there exists
\(L_K\in(0,1)\) such that, for every Lipschitz map
\(F:K\to\R^n\) with \(\operatorname{Lip}(F)<L_K\), the map
\(T:=\operatorname{Id}+F\) is injective on \(K\), and \(T(K)\) is a compact
Lipschitz domain.
\end{corollary}

\begin{proof}
This follows immediately from
Theorem~\ref{DBthm:small-biLip-perturbation-cone} applied to
\(U=\operatorname{int}K\), after extending \(F\) to \(\R^n\) without increasing
its Lipschitz constant. Indeed, the resulting map \(T\) is a global
bi-Lipschitz homeomorphism, and therefore
\[
  T(K)=T(\overline U)=\overline{T(U)},
  \qquad
  \operatorname{int}T(K)=T(U).
\]
Thus
\[
  T(K)=\overline{\operatorname{int}T(K)},
\]
and \(\operatorname{int}T(K)=T(U)\) is a bounded Lipschitz domain. Hence
\(T(K)\) is a compact Lipschitz domain.
\end{proof}

{We next record} a basic approximation result which is tailored to the case
where only selected weak first-order derivatives are prescribed. The proof is
included because the approximation is global up to the boundary, while only the
listed derivatives are controlled.

\begin{lemma}[Approximation with prescribed weak derivatives]\label{lem:one-directional-approximation}
Let \(U\subset\R^m\) be a bounded Lipschitz domain, let
\(\mathcal I\subset\{1,\ldots,m\}\) be nonempty, and let
\(u\in C(\overline U)\). Assume that, for every \(i\in\mathcal I\), the weak
derivative \(D_i u\) of \(u\) exists and is represented by \(g_i\in L^1(U)\), i.e.
\[
  \int_U u\,D_i\zeta\,dz
  =
  -\int_U g_i\,\zeta\,dz
  \qquad
  \text{for all }\zeta\in C_c^\infty(U).
\]
Then there exists a sequence \(u_j\in C^\infty(\R^m)\) such that
\[
  u_j\to u
  \qquad\text{uniformly on }\overline U,
\]
and
\[
  D_i u_j\to g_i
  \qquad\text{in }L^1(U)
  \quad\text{for every }i\in\mathcal I.
\]
In particular, if \(\mathcal I=\{1,\ldots,m\}\), then \(u\in W^{1,1}(U)\) and
the same sequence satisfies
\[
  u_j\to u
  \qquad\text{in }W^{1,1}(U),
\]
in addition to the uniform convergence on \(\overline U\).
\end{lemma}

\begin{proof}
Let \(\eta\in C_c^\infty(B_1(0))\) be a standard mollifier and set
\[
  \eta_\varepsilon(z):=\varepsilon^{-m}\eta(z/\varepsilon).
\]
By the Tietze extension theorem, extend \(u\) from \(\overline U\) to a
bounded continuous function \(\widetilde u\in C(\R^m)\). Let
\[
  u_\varepsilon:=\eta_\varepsilon*\widetilde u
\]
be the corresponding Friedrichs regularization in \(\R^m\). Then
\(u_\varepsilon\in C^\infty(\R^m)\). Since \(\overline U\) is compact,
\(\widetilde u\) is uniformly continuous on a compact neighborhood of
\(\overline U\), and therefore
\[
  u_\varepsilon\to u
  \qquad\text{uniformly on }\overline U .
\]

It remains to prove convergence of the weak derivatives. For
\(\delta>0\), set
\[
  U_\delta:=\{z\in U:\operatorname{dist}(z,\partial U)>\delta\},
  \qquad
  R_\delta:=U\setminus U_\delta .
\]
We first note that the boundary layer has volume of order \(\delta\). Since
\(U\) is a bounded Lip\-schitz domain, \(\partial U\) is compact and can be
covered by finitely many local Lipschitz graph neighborhoods. More precisely,
in each such neighborhood, after a rigid change of variables, there are
\(r,h>0\) and a Lipschitz function \(\phi:B'_r(0)\to\R\) such that
\[
  U\cap\bigl(B'_r(0)\times(-h,h)\bigr)
  =
  \{(y',y_m):y'\in B'_r(0),\ -h<y_m<h,\ y_m>\phi(y')\}.
\]
After shrinking these neighborhoods if necessary, \(R_\delta\) is contained in
their union for all sufficiently small \(\delta>0\). If
\(y=(y',y_m)\) lies within distance \(\delta\) of \(\partial U\), then for some
boundary point \(\widehat y=(\widehat y',\phi(\widehat y'))\) in the same local
coordinates one has \(|y-\widehat y|<\delta\). Hence
\[
\begin{aligned}
  |y_m-\phi(y')|
  &\le |y_m-\phi(\widehat y')|
      +|\phi(\widehat y')-\phi(y')|  \\
  &\le \delta+\operatorname{Lip}(\phi)\,|\widehat y'-y'|
  \le \bigl(1+\operatorname{Lip}(\phi)\bigr)\delta .
\end{aligned}
\]
Thus, in each local graph neighborhood, the boundary layer is contained in a
vertical strip of thickness \(O(\delta)\) around the graph of \(\phi\). Since
only finitely many such neighborhoods are needed, there is a constant \(C>0\)
such that
\[
  \mathcal L^m(R_\delta)\le C\delta
\]
for all sufficiently small \(\delta>0\). Here \(\mathcal L^m\) denotes
\(m\)-dimensional Lebesgue measure.

Fix \(i\in\mathcal I\), and extend \(g_i\) by zero outside \(U\). On
\(U_{2\varepsilon}\), the mollification only uses values inside \(U\), because
\(\eta_\varepsilon\) is supported in \(B_\varepsilon(0)\). Thus, using the weak
identity \(D_i u=g_i\) in \(U\), we have
\[
  D_i u_\varepsilon=\eta_\varepsilon*g_i
  \qquad\text{on }U_{2\varepsilon}.
\]
The \(L^1\)-continuity of Friedrichs mollification gives
\[
  \int_{U_{2\varepsilon}}
  |D_i u_\varepsilon-g_i|\,dz
  \le
  \|\eta_\varepsilon*g_i-g_i\|_{L^1(\R^m)}
  \to0 .
\]

It remains to estimate the boundary layer \(R_{2\varepsilon}\). Let
\(\omega_{\widetilde u}\) denote the modulus of continuity of
\(\widetilde u\) on a fixed compact neighborhood of \(\overline U\). Since
\(\eta_\varepsilon\in C_c^\infty(\R^m)\), integration by parts on \(\R^m\)
without boundary terms at infinity gives
\[
  \int_{\R^m}D_i\eta_\varepsilon(y)\,dy=0.
\]
Therefore
\[
  |D_i u_\varepsilon(z)| =
  \left|
    \int_{\R^m}D_i\eta_\varepsilon(y)
    \bigl(\widetilde u(z-y)-\widetilde u(z)\bigr)\,dy
  \right|  \le C\varepsilon^{-1}\omega_{\widetilde u}(\varepsilon).
\]
Using \(\mathcal L^m(R_{2\varepsilon})\le C\varepsilon\), this gives
\[
  \int_{R_{2\varepsilon}}|D_i u_\varepsilon|\,dz
  \le C\omega_{\widetilde u}(\varepsilon)\to0.
\]
Moreover, since \(g_i\in L^1(U)\) and
\(\mathcal L^m(R_{2\varepsilon})\to0\), the absolute continuity of the
Lebesgue integral yields
\[
  \int_{R_{2\varepsilon}}|g_i|\,dz\to0.
\]
Combining the interior and boundary-layer estimates yields
\[
  D_i u_\varepsilon\to g_i
  \qquad\text{in }L^1(U)
\]
for every \(i\in\mathcal I\). Since \(\mathcal I\) is finite, the same sequence
works for all prescribed weak derivatives simultaneously. Choosing any
sequence \(\varepsilon_j\downarrow0\) and setting
\(u_j:=u_{\varepsilon_j}\) proves the assertion.

If \(\mathcal I=\{1,\ldots,m\}\), then all first-order weak derivatives of
\(u\) belong to \(L^1(U)\), hence \(u\in W^{1,1}(U)\). The uniform convergence
on \(\overline U\) implies convergence in \(L^1(U)\), and the weak-derivative
convergence gives convergence in \(W^{1,1}(U)\).
\end{proof}

{
\begin{lemma}[Uniform finite-face and gluing criterion]
\label{lem:uniform-finite-face-gluing}
Let $\{U_\lambda\}_{\lambda\in\Lambda}$ be a family of bounded open
sets in $\R^m$. Assume that, near every point of $\partial U_\lambda$, the
local set is the intersection of the inward sides of at most $N$ Lipschitz
faces which admit a common graph direction in the following uniform sense:
after a rigid change of variables $y=(y',y_m)$, there are functions
$g_j$, $j=1,\ldots,k$, $k\le N$, with
$\operatorname{Lip}(g_j)\le L$, such that locally
\[
        U_\lambda
        =\{(y',y_m):y_m>g_j(y')\text{ for }j=1,\ldots,k\}.
\]
The reflected alternative, with all inequalities reversed, is also allowed. Then every $U_\lambda$
is a bounded Lipschitz domain. The opening angle in the two-sided cone condition
can be chosen depending only on $L$ and $N$; the cone height may depend on the
local coordinate neighborhood.

The criterion also applies to the \emph{regularized union}
\[
U:=\operatorname{int}\overline{U_1\cup U_2}
\]
obtained by gluing two such domains along a common relatively open boundary face,
provided that the two domains lie on opposite sides of the common face and that,
at the edge of that face, the remaining active boundary faces admit a common
graph direction as above. The relative interior of the common face is then
contained in $U$ and disappears from $\partial U$.

The same statements hold for parameterized space-time sets in $\R\times
\R^m$, with time treated as an additional tangential variable. If a lateral
face is represented in such coordinates by
\[
        y_m=g(t,y'),
\]
where $y_m$ is a spatial coordinate and $g$ is Lipschitz, then its spatial normal
component is nonzero almost everywhere and its normal velocity satisfies
\[
        |V_n|\le \operatorname{Lip}_t(g)
        \qquad\text{a.e. on the face}.
\]
\end{lemma}

\begin{proof}
In the first alternative, the local component of $U_\lambda$ is
\[
        \{(y',y_m):y_m>g(y')\},
        \qquad g(y'):=\max_{1\le j\le k}g_j(y').
\]
The function $g$ is Lipschitz and
$\operatorname{Lip}(g)\le L$. In the reflected alternative one uses
$\min_j g_j$. Thus every boundary point has a Lipschitz graph neighborhood.
The standard graph-to-cone construction gives a cone opening angle depending
only on $L$; allowing at most $N$ faces does not change this bound. Compactness
of each boundary supplies a finite cover and hence a positive cone height for
each fixed domain.

For the gluing assertion, a point in the relative interior of the common face
has a neighborhood in which the two domains occupy opposite sides. Passing to
$U=\operatorname{int}\overline{U_1\cup U_2}$ fills in that face, which therefore
disappears from the boundary. At an edge point, the boundary of the regularized
union consists precisely of the remaining active faces. By the
common-direction hypothesis their inward sides are described by simultaneous
graph inequalities, and the first part applies.

For the space-time assertion the same maximum/minimum argument is used in
$\R^{m+1}$. By Rademacher's theorem, at almost every point of the graph
$y_m=g(t,y')$ a normal is proportional to
\[
        (-\partial_t g,-\nabla_{y'}g,1).
\]
Consequently its spatial component is nonzero, and, with the sign convention of
Definition~\ref{def:speed-of-normal-displacement},
\[
        |V_n|
        =\frac{|\partial_t g|}{\sqrt{1+|\nabla_{y'}g|^2}}
        \le |\partial_t g|
        \le \operatorname{Lip}_t(g)
\]
almost everywhere. This proves all assertions.
\end{proof}
}

\section{Explicit structure of the attainable sets}\label{app:attainable-sets-cusp}

As explained in Section~\ref{sec:kinematics}, we use the elementary field
\eqref{eq-counter-ex} here because it permits fully explicit formulas. The modified
constant-viscosity, constant-surface-tension field \eqref{eq:localized-traces} retains the
same local non-uniqueness and tangential-departure mechanism at $(0,0)$. It also retains
the finite-time tangential-contact mechanism relevant for the later cusp formation. Indeed,
for $\eta=1$ the upper-phase field has the first integral
\[
        H(x_1,x_2)=x_2+\frac12x_1x_2^2-\frac12x_1^2.
\]
The separatrix $H=0$ through the origin meets the upper part of the boundary of the same
initial disk at a second point and reaches $(0,0)$ tangentially in finite time. Indeed, writing
$x_1=-u$ on this branch gives $x_2=u^2/(1+\sqrt{1-u^3})$ and
$\dot x_1=\sqrt{1-u^3}>0$; comparison with the upper semicircle at $u=0.8$ and $u=0.9$
brackets a second intersection. We do not need the full attainable-set geometry for this
modified field and therefore carry out the detailed cusp calculations only for the simpler
field \eqref{eq-counter-ex}.

Consider the differential inclusion
\begin{equation}
\dot x_1(t)\in \mathrm{Sgn}(x_2(t)),\qquad
\dot x_2(t)=x_1(t).
\label{eq:DI}
\end{equation}
Here \(\mathrm{Sgn}(r)=\{-1\}\) for \(r<0\), \(\mathrm{Sgn}(0)=[-1,1]\),
and \(\mathrm{Sgn}(r)=\{1\}\) for \(r>0\).
The initial set is the disk
\[
G_0=\{(a,b)\in\R^2:(a+\tfrac12)^2+b^2\le\tfrac14\}.
\]
For \(t\ge0\) denote by
\[
G(t)=\{x(t):x(\cdot)\text{ is a strong solution of \eqref{eq:DI} with }x(0)\in G_0\}
\]
the attainable set. Away from \(x_2=0\) the right-hand side of \eqref{eq:DI} is single-valued. 
Solutions starting from \((a,b)\) satisfy
\[
\begin{aligned}
x_1(s)&=a+s, & x_2(s)&=b+as+\tfrac12 s^2, && (b>0),\\
x_1(s)&=a-s, & x_2(s)&=b+as-\tfrac12 s^2, && (b<0).
\end{aligned}
\]
Define
\[
\Phi_t^-(a,b)=(a-t,\ b+ta-\tfrac12 t^2),
\qquad
\Phi_t^+(a,b)=(a+t,\ b+ta+\tfrac12 t^2).
\]
If \(b>0\) and the trajectory crosses \(x_2=0\) before time \(t\), the first crossing time is
\[
\tau(a,b)=-a-\sqrt{a^2-2b},
\]
and the endpoint becomes
\[
\Psi_t(a,b)=
\left(
a-t+2\tau(a,b),\
(a+\tau(a,b))(t-\tau(a,b))-\tfrac12 (t-\tau(a,b))^2
\right).
\]
Solutions that reach the origin may remain there for an arbitrary time and then leave with velocity \((\pm1,0)\). 
Their endpoints form
\[
P_t^\pm=\{(\pm r,\pm\tfrac12 r^2):0\le r\le t\}.
\]

Define
\[
m_t(a)=
\begin{cases}
-at-\dfrac{t^2}{2}, & \;\; a\le -t,\\[1mm]
\dfrac{a^2}{2}, & -t\le a\le0 ,
\end{cases}
\]
and
\[
\begin{aligned}
G_0^- &=\{(a,b)\in G_0:\ b\le 0\},\\
G_0^+(t) &=\{(a,b)\in G_0:\ b>0,\ b\ge m_t(a)\},\\
G_0^\times(t) &=\{(a,b)\in G_0:\ b>0,\ b<m_t(a)\}.
\end{aligned}
\]
With these notations, the attainable set decomposes as
\begin{equation}
G(t)=
\Phi_t^-(G_0^-)
\;\cup\;
\Phi_t^+(G_0^+(t))
\;\cup\;
\Psi_t(G_0^\times(t))
\;\cup\;
P_t^+
\;\cup\;
P_t^- .
\label{eq:structure}
\end{equation}

The qualitative mechanism behind the transition is already visible from the quadrant structure. 
In the upper half-plane one has \(\dot x_1=1\), whereas in the lower half-plane one has \(\dot x_1=-1\). 
Moreover, the normal velocity at the switching line is \(\dot x_2=x_1\). 
Thus, for \(x_1<0\), trajectories may cross downward from the upper to the lower phase, while for \(x_1>0\), trajectories may cross upward from the lower to the upper phase. 
In this sense the second and fourth quadrants are evacuated by phase change into the third and first quadrants, respectively. 
The critical time is therefore not the purely horizontal time scale \(1\), but the time at which the last tangential boundary trajectory reaches the origin.

Solutions that hit \(x_2=0\) transversally enter the opposite phase and allow for a unique continuation locally in time. 
Possible non-uniqueness can only occur for solutions that touch \(x_2=0\) tangentially. 
Solutions that start in \((a,b)\) with \(b>0\) can only reach \(x_2=0\) tangentially if
\[
b=\frac{a^2}{2}.
\]
Intersecting this curve with \(\partial G_0\) yields
\[
(a+\tfrac12)^2+\frac{a^4}{4}=\frac14
\quad \Leftrightarrow \quad
a\,(a^3+4a+4)=0.
\]
The root \(a=0\) corresponds to the origin, where the non-unique stationary branch may start. 
The nontrivial point relevant for the first tangential contact of the transported boundary is determined by
\[
a^3+4a+4=0.
\]
Since
\[
\frac{d}{da}(a^3+4a+4)=3a^2+4>0
\quad\text{for all } a\in\R,
\]
this cubic has a unique real root. 
Let \(a_0\in(-1,0)\) be this root and set
\[
t_*:=-a_0.
\]
By Cardano's formula,
\[
a_0=
\sqrt[3]{-2+\sqrt{\frac{172}{27}}}
+
\sqrt[3]{-2-\sqrt{\frac{172}{27}}}
\approx -0.8477076.
\]
The trajectory starting at
\[
p_*=(a_0,\tfrac12 a_0^2)\in\partial G_0
\]
is tangent to the switching line \(x_2=0\) and satisfies
\[
x_1(s)=a_0+s,\qquad
x_2(s)=\frac12(a_0+s)^2.
\]
Hence it reaches the origin exactly at time \(t_*\). 
For \(t\ge t_*\), the remaining time \(t-t_*\) allows the solution to leave the origin with horizontal velocity \(\pm1\), producing the points
\[
q_\pm(t)=(\pm(t-t_*),\pm\tfrac12(t-t_*)^2).
\]
These are precisely the attachment points where the outer boundary arcs meet the curves \(P_t^\pm\).

\begin{theorem}[Geometry of the attainable sets]
\label{thm:geometry}
Let \(G(t)\) be defined by \eqref{eq:structure}. Then the following assertions hold.

\begin{enumerate}

\item For \(0<t<t_*\), the set \(G(t)\) has one connected component with nonempty interior.

\item For \(0<t\le t_*\), the lower boundary of the part of \(G(t)\) with \(x_1\ge0\) is
\[
x_2=\frac{x_1^2}{2},\qquad 0\le x_1\le t .
\]
For \(t>t_*\), only the truncated parabolic segment
\[
\{(r,\tfrac12 r^2):t-t_*\le r\le t\}
\]
bounds the right two-dimensional component.

\item For every \(0<t<t_*\), the lower-phase part
\[
G^-(t):=G(t)\cap\{x_2<0\}
\]
has a cusp at the origin. More precisely, near the origin its boundary contains the two tangentially meeting arcs
\[
\{(x_1,0):-\varepsilon<x_1<0\}
\quad\text{and}\quad
\{(x_1,-\tfrac12 x_1^2):-\varepsilon<x_1<0\}
\]
for sufficiently small \(\varepsilon>0\).

\item At \(t=t_*\), the boundary of \(G(t)\) develops a cusp at the origin.

\item For \(t>t_*\), the interior of \(G(t)\) has two connected components, denoted by
\(G_+(t)\) and \(G_-(t)\). The full attainable set is obtained from the closures of these components together with the one-dimensional connector
\[
P_t^*:=\{(\pm r,\pm\tfrac12 r^2):0\le r\le t-t_*\}.
\]
In other words,
\[
G(t)=\overline{G_+(t)}\cup\overline{G_-(t)}\cup P_t^*.
\]

\item The two-dimensional Lebesgue measure of the interior of \(G(t)\) is independent of \(t\).

\end{enumerate}
\end{theorem}

\begin{proof}
The decomposition \eqref{eq:structure} follows directly from the explicit solutions of the two smooth subsystems together with the possibility of remaining at the origin and leaving it with velocity \((\pm1,0)\).

If \(b\le0\), then \(a\le0\) on \(G_0\). Apart from the origin, the trajectory immediately enters the lower phase and satisfies \(x_1(s)=a-s<0\) for \(s>0\). Hence \(x_2\) decreases strictly and the trajectory never reaches the switching line again. This yields the contribution \(\Phi_t^-(G_0^-)\); the origin is also the starting point of the one-dimensional branches \(P_t^\pm\).

If \(b>0\), the trajectory initially satisfies
\[
x_2(s)=b+as+\tfrac12s^2 .
\]
Minimizing this quadratic on \([0,t]\) gives the threshold \(m_t(a)\). Hence \(G_0^+(t)\) consists precisely of the upper initial points whose trajectories do not cross before time \(t\), while \(G_0^\times(t)\) consists of those which do cross. The corresponding endpoint maps are \(\Phi_t^+\) and \(\Psi_t\).

For \(0<t<t_*\), the tangency trajectory starting at \(p_*\) has not yet reached the origin. Thus the images of the upper, crossing and lower parts of \(\partial G_0\), together with the parabolic branches generated from the origin, form one connected region with nonempty interior.

We next identify the positive-side lower boundary. A crossed trajectory cannot contribute to the region \(x_1>0\). Indeed, if it crosses at time \(\tau<t\), then
\[
x_1(t)=a-t+2\tau=(a+\tau)-(t-\tau),
\]
and at the crossing point
\[
a+\tau=-\sqrt{a^2-2b}\le0.
\]
Therefore \(x_1(t)<0\). Thus every attainable point with \(x_1>0\) which belongs to the two-dimensional part is generated by a non-crossing upper trajectory. Writing \(x_1=a+t\), the condition \(x_1>0\) implies \(a>-t\), and the no-crossing condition is \(b\ge a^2/2\). Minimizing
\[
x_2(t)=b+ta+\tfrac12t^2
\]
under this constraint yields
\[
x_2(t)\ge
\frac{a^2}{2}+ta+\frac{t^2}{2}
=\frac{(a+t)^2}{2}
=\frac{x_1^2}{2}.
\]
Equality is attained by tangency trajectories. This proves the stated parabolic boundary. For \(t>t_*\), the part with \(0\le x_1\le t-t_*\) is no longer boundary of a two-dimensional component; it belongs to the one-dimensional connector.

We now consider the lower-phase part. Fix \(0<t<t_*\). Since \((-t,t^2/2)\) lies in the interior of \(G_0\) for \(t<t_*\), there is \(\varepsilon>0\) such that the following construction starts from points in \(G_0\). Let \(0<r<\varepsilon\) and \(0<\delta<r\). Put
\[
c=-r+\delta,\qquad \tau=t-\delta.
\]
Consider an upper trajectory which crosses the switching line at time \(\tau\) at the point \((c,0)\). Its initial point is
\[
a=c-\tau,\qquad b=-a\tau-\frac{\tau^2}{2},
\]
and, by the preceding choice of \(\varepsilon\), belongs to \(G_0\). After the crossing, the trajectory follows the lower dynamics during the remaining time \(\delta\). Hence its endpoint is
\[
x_1=c-\delta=-r,\qquad
x_2=c\delta-\frac12\delta^2
=-r\delta+\frac12\delta^2.
\]
As \(\delta\) varies in \((0,r)\), this fills precisely
\[
-\frac12r^2<x_2<0.
\]
Consequently, near the origin,
\[
\{(x_1,x_2):-\varepsilon<x_1<0,\ -\tfrac12x_1^2<x_2<0\}
\subset G^-(t).
\]
Conversely, any crossed upper trajectory ending near the origin in the lower half-plane
can be written in terms of its crossing time $\tau$, its crossing point $(c,0)$, and
$\delta:=t-\tau>0$. The other possible lower-phase contributions do not create additional
points sufficiently near the origin: if the initial point lies in $G_0$ with $b<0$, or on
$x_2=0$ with $a<0$, then $x_1(t)\le -t$, while solutions issuing from $(0,0)$ that enter
the lower phase form precisely the branch $P_t^-$. For a crossed upper trajectory one has
\[
        x_1=c-\delta,\qquad
        x_2=c\delta-\frac12\delta^2 .
\]
For a downward crossing one has $c\le0$. Writing $r:=-x_1=\delta-c>0$, this implies
$0\le\delta\le r$, and therefore
\[
        x_2=-r\delta+\frac12\delta^2\ge -\frac12 r^2
        =-\frac12 x_1^2 .
\]
Thus, in a sufficiently small neighborhood of the origin, no attainable point of the lower
phase lies below the parabola $x_2=-x_1^2/2$.

The upper boundary of this region is the interface \(x_2=0\), while the lower boundary is the parabolic branch \(x_2=-x_1^2/2\), generated by the lower branch \(P_t^-\). These two curves meet at the origin with common tangent \(x_2=0\). Therefore \(G^-(t)\) has a cusp at the origin for every \(0<t<t_*\).

We next make the critical-time separation explicit. Suppose an upper trajectory reaches the switching line at time \(\tau>0\) at the point \((c,0)\), where necessarily \(c\le0\). Its initial point is
\[
        a=c-\tau,\qquad b=-c\tau+\frac12\tau^2.
\]
Hence the condition \((a,b)\in G_0\) reads
\[
        (c-\tau+\tfrac12)^2+(-c\tau+\tfrac12\tau^2)^2\le\frac14.
\]
For \(\tau\ge t_*\), the left-hand side is strictly decreasing as a function of \(c\le0\), and is therefore bounded below by its value at \(c=0\). Consequently,
\[
        (\tfrac12-\tau)^2+\frac14\tau^4\le\frac14,
\]
which is equivalent to
\[
        \tau^3+4\tau-4\le0.
\]
Since \(t_*\) is the unique positive root of \(t^3+4t-4=0\), every such hitting time satisfies \(\tau\le t_*\), with equality only for \(c=0\), i.e. for the tangency trajectory starting at \(p_*\).

At \(t=t_*\), the unique trajectory realizing this last contact reaches the origin tangentially. Together with the positive-side and lower-phase parabolic boundary descriptions above, this shows that the two-dimensional parts arriving from the two sides meet there with the same tangent line and zero opening angle; equivalently, no uniform interior cone can be placed at the origin. Thus a cusp forms and the boundary of \(G(t)\) is no longer Lipschitz at that point.

For \(t>t_*\), every crossed upper trajectory has
\[
        x_1(t)=c-(t-\tau)\le -(t-t_*).
\]
The image of the lower initial part satisfies \(x_1(t)=a-t\le-t\), so it lies still farther to the left. On the other hand, a non-crossing upper trajectory must start with \(a>a_0\): for \(a<a_0\), the upper boundary of \(G_0\) lies strictly below the tangency curve \(b=a^2/2\), so every upper initial point crosses, while at \(a=a_0\) the equality case is precisely \(p_*\). Hence every two-dimensional non-crossing upper endpoint satisfies \(x_1=a+t>t-t_*\). Thus no two-dimensional part enters the central strip \(|x_1|<t-t_*\). The remaining time \(t-t_*\) after the tangency trajectory has reached the origin generates the points
\[
q_\pm(t)=(\pm(t-t_*),\pm\tfrac12(t-t_*)^2).
\]
Only the truncated parts of the parabolas with parameter \(r\in[t-t_*,t]\) bound regions of positive area. The central parts with \(0\le r\le t-t_*\) form the one-dimensional connector \(P_t^*\). Consequently, the interior has two connected components, whose closures meet this connector as stated.

Finally, the area assertion can be checked directly from the decomposition \eqref{eq:structure}. The affine maps $\Phi_t^\pm$ have determinant one. On $G_0^\times(t)$ the crossing map $\Psi_t$ is $C^1$ away from the tangency set, and direct differentiation of the formula above gives $\det D\Psi_t=1$. Apart from the one-dimensional set of initial data whose trajectories reach the origin tangentially, forward and backward continuation is unique; hence the three two-dimensional images in \eqref{eq:structure} are pairwise disjoint up to null sets. The area formula therefore yields
\[
\mathcal L^2(G(t)^\circ)
=\mathcal L^2(G_0^-)+\mathcal L^2(G_0^+(t))+\mathcal L^2(G_0^\times(t))
=\mathcal L^2(G_0).
\]
The branches generated from the origin and the relevant tangency sets are one-dimensional and do not change two-dimensional Lebesgue measure. This also makes explicit why transversal crossings cause no loss of area: the one-sided vector fields are divergence-free and have the same normal component $x_1$ at the switching line.
\end{proof}


\section{Surface change of variables under Lipschitz perturbations}\label{app:surface-change-of-variables}

We record the hypersurface change-of-variables formula in the form used in the
proof of Theorem~\ref{DBthm:2ph-transport-theorem-comovingCV}. The change-of-variables
identity itself is a standard consequence of the area formula for Lipschitz maps
on rectifiable sets; the estimates below are elementary consequences of the
smallness assumption on the Lipschitz perturbation. The point in the present
formulation is that the maps used in the proof are only Lipschitz perturbations
of the identity; their tangential derivatives therefore have to be understood
almost everywhere. We include the short proof to keep the precise form used
below self-contained. It uses the Euclidean area formula for Lipschitz maps in
local hypersurface charts; see, e.g., Leoni~\cite{leoni2017first}.

\begin{proposition}[Hypersurface area formula for Lipschitz perturbations]\label{prop:surface-area-formula-lipschitz}
Let $S\subset\R^n$, $n\ge2$, be an oriented $C^1$ hypersurface and let
$E\subset S$ be $\mathcal H^{n-1}$-measurable with finite
$\mathcal H^{n-1}$-measure. Let $\Psi=\Id+R$ be Lipschitz on a neighborhood of
$S$ and assume that $\Lip(R)\le\alpha<1$. Then $\Psi|_S$ is bi-Lipschitz onto
its image. For $\mathcal H^{n-1}$-almost every $\xi\in S$, the tangential
differential
\[
        D_\tau\Psi(\xi):T_\xi S\to\R^n
\]
exists. If $(\DBtau_1,\ldots,\DBtau_{n-1})$ is an oriented orthonormal basis of
$T_\xi S$, define
\[
        J_S\Psi(\xi)
        :=\sqrt{\det\bigl(G_\Psi(\xi)\bigr)},
        \qquad
        (G_\Psi)_{ij}(\xi)
        := D_\tau\Psi(\xi)\DBtau_i\cdot
           D_\tau\Psi(\xi)\DBtau_j .
\]
Equivalently, $J_S\Psi$ is the product of the $n-1$ singular values of
$D_\tau\Psi(\xi)$. This number is independent of the chosen oriented
orthonormal basis. For every measurable $f:\Psi(E)\to\R$ one has
\[
        \int_{\Psi(E)} f(\DBy)\,d\mathcal H^{n-1}(\DBy)
        =
        \int_E f(\Psi(\xi)) J_S\Psi(\xi)\,d\mathcal H^{n-1}(\xi),
\]
whenever one of the two integrals is well-defined.

Moreover, for $\mathcal H^{n-1}$-almost every $\xi\in S$,
\[
        (1-\alpha)^{n-1}\le J_S\Psi(\xi)
        \le (1+\alpha)^{n-1}.
\]
If, in addition, $\alpha\le1/2$, then there is a constant $C_n$, depending only
on $n$, such that
\[
        |J_S\Psi(\xi)-1|\le C_n\alpha .
\]
Under the same additional assumption, if $\DBn_S$ denotes the unit normal of $S$, then the transported unit normal
$\DBn_\Psi$ on $\Psi(S)$, with the orientation induced by $\Psi|_S$, satisfies
\[
        |\DBn_\Psi(\Psi(\xi))-\DBn_S(\xi)|\le C_n\alpha
\]
for $\mathcal H^{n-1}$-almost every $\xi\in S$.
\end{proposition}

\begin{proof}
Since $\Lip(R)<1$, for $\xi,\eta\in S$,
\[
        (1-\alpha)|\xi-\eta|
        \le |\Psi(\xi)-\Psi(\eta)|
        \le (1+\alpha)|\xi-\eta| .
\]
Thus $\Psi|_S$ is bi-Lipschitz onto its image.

We first prove the change-of-variables formula. Cover $S$, up to an
$\mathcal H^{n-1}$-null set, by countably many $C^1$ coordinate charts
$\kappa_j:U_j\subset\R^{n-1}\to S$, restricted if necessary so that each
$\kappa_j$ is Lipschitz on $U_j$. Set
\[
        \widetilde E_j:=E\cap\kappa_j(U_j),\qquad
        E_1:=\widetilde E_1,\qquad
        E_j:=\widetilde E_j\setminus\bigcup_{i<j}\widetilde E_i
        \quad (j\ge2).
\]
Then the $E_j$ are pairwise disjoint, measurable, satisfy
$E_j\subset E\cap\kappa_j(U_j)$, and cover $E$ up to an
$\mathcal H^{n-1}$-null set.

Put $A_j:=\kappa_j^{-1}(E_j)$ and $F_j:=\Psi\circ\kappa_j$. Extending $f$ by
zero outside $\Psi(E)$, the Euclidean area formula for Lipschitz maps, see
e.g. Leoni~\cite{leoni2017first}, gives
\[
        \int_{A_j} f(F_j(u))\,JF_j(u)\,du
        =
        \int_{\R^n} f(\DBy)\,N(F_j,A_j,\DBy)\,
        d\mathcal H^{n-1}(\DBy),
\]
where $N(F_j,A_j,\DBy)$ is the number of preimages of $\DBy$ in $A_j$.
Since both $\kappa_j$ and $\Psi|_S$ are injective, $F_j$ is injective on $A_j$.
Hence the multiplicity is one on $F_j(A_j)=\Psi(E_j)$ and zero outside this
image. Thus
\[
        \int_{\Psi(E_j)} f\,d\mathcal H^{n-1}
        =
        \int_{\kappa_j^{-1}(E_j)}
        f(\Psi(\kappa_j(u)))\,JF_j(u)\,du,
\]
with
\[
        JF_j(u):=\sqrt{\det\bigl(DF_j(u)^TDF_j(u)\bigr)}.
\]
Since
\[
        d\mathcal H^{n-1}|_S
        = J\kappa_j(u)\,du,
        \qquad
        J\kappa_j(u):=\sqrt{\det\bigl(D\kappa_j(u)^TD\kappa_j(u)\bigr)},
\]
and since $JF_j/J\kappa_j$ is precisely $J_S\Psi$ in this chart, summing over
$j$ proves the asserted formula. The existence of the tangential differential
and of the displayed Jacobian follows from Rademacher's theorem in charts.

It remains to prove the estimates. In what follows, $C_n$ denotes a positive
constant depending only on $n$, whose value may change from line to line. At
almost every point of tangential differentiability,
\[
        |D_\tau\Psi(\xi)\DBtau-\DBtau|
        = |D_\tau R(\xi)\DBtau|
        \le \alpha|\DBtau|
        \qquad(\DBtau\in T_\xi S).
\]
Hence all singular values of $D_\tau\Psi(\xi)$ lie between $1-\alpha$ and
$1+\alpha$. Taking their product gives
\[
        (1-\alpha)^{n-1}\le J_S\Psi\le(1+\alpha)^{n-1}.
\]
For $\alpha\le1/2$, say, the estimate $|J_S\Psi-1|\le C_n\alpha$ follows at
once; for instance, the upper bound follows from the binomial expansion of
$(1+\alpha)^{n-1}$, and the lower bound is analogous.

It remains to control the normal. Fix a point of tangential differentiability
and an oriented orthonormal basis
$(\DBtau_1,\ldots,\DBtau_{n-1})$ of $T_\xi S$. Write
\[
        D_\tau\Psi\DBtau_i=\DBtau_i+\DBe_i,
        \qquad |\DBe_i|\le\alpha .
\]
In the splitting $T_\xi S\oplus\operatorname{span}\{\DBn_S\}$, the image tangent
plane is spanned by vectors
\[
        e_i' + B e_i' + a_i\DBn_S,
        \qquad i=1,\ldots,n-1,
\]
where $e_i'$ is the standard basis of $\R^{n-1}$, and where $B$ and $a$ are the
tangential and normal parts of $D_\tau R$. The norms below are operator norms
induced by the Euclidean norm, and
\[
        \|B\|\le\alpha,\qquad \|a\|\le\alpha .
\]
For $\alpha\le1/2$, the map $I+B$ is invertible and
$\|(I+B)^{-1}\|\le2$. Hence the image tangent plane is the graph over
$T_\xi S$ of
\[
        L:=a(I+B)^{-1}:T_\xi S\to\operatorname{span}\{\DBn_S\},
        \qquad \|L\|\le 2\alpha .
\]
Writing $Lu=\ell\cdot u\,\DBn_S$, the oriented unit normal to this graph is
\[
        \frac{-\ell+\DBn_S}{\sqrt{1+|\ell|^2}},
\]
where the orientation agrees with the one induced by $\Psi|_S$, since $I+B$ is
homotopic to the identity through invertible maps. Therefore
\[
        \left|
        \frac{-\ell+\DBn_S}{\sqrt{1+|\ell|^2}}-\DBn_S
        \right|
        \le C_n|\ell|
        \le C_n\alpha .
\]
This proves the asserted estimate for the transported normal.
\end{proof}


\section*{Acknowledgment}
The first author (DB) met the late Herrmann Sohr while he was still a student and took a seminar on differential equations with him in the summer semester of 1985. He is very grateful to him for his invaluable support during difficult times in a later period.

It is our pleasure to acknowledge fruitful discussions with Kohei Soga (Keio University, Tokyo).

\end{document}